\documentclass{amsart}

\usepackage[numbers,sort&compress]{natbib}

\usepackage{amssymb}
\usepackage{amsmath,amsfonts,amssymb,amsthm}
\usepackage{amsaddr}
\usepackage{bbm}
\usepackage{mathtools}
\usepackage{commath}
\usepackage{enumitem}
\usepackage{color}
\usepackage{setspace}
\allowdisplaybreaks[1]

\usepackage[all,cmtip]{xy}
\usepackage{hyperref}

\calclayout

\numberwithin{equation}{section}

\theoremstyle{plain} %% This is the default, anyway
\newtheorem{thm}{Theorem}[section]
\newtheorem{cor}[thm]{Corollary}
\newtheorem{lem}[thm]{Lemma}

\newtheorem{defn}[thm]{Definition}
\newtheorem{rem}[thm]{Remark}
\newtheorem{ex}[thm]{Example}

\begin{document}
%%% In the title, use a double backslash "\\" to show a linebreak:
%%% Use one of the following two forms:
%%% \title{Text of the title}
%%% or
%%% \title[Short form for the running head]{Text of the title}
\title[Measure solutions to perturbed structured population models]{Measure solutions to perturbed structured population models -- differentiability with respect to perturbation parameter}

%%% If there are multiple authors, they're described one at a time:
%%% First author: \author{} \address{} \curraddr{} \email{} \thanks{}
%%% Second author: \author{} \address{} \curraddr{} \email{} \thanks{}
%%% Third author: \author{} \address{} \curraddr{} \email{} \thanks{}
\author{Jakub Skrzeczkowski}

%%% In the address, show linebreaks with double backslashes:
\address{Faculty of Mathematics, Informatics and Mechanics, University of Warsaw \\ Banacha 2, 02--097 Warsaw}

%%% Current address is optional.
%\curraddr{}

%%% Email address is optional.
\email{jakub.skrzeczkowski@student.uw.edu.pl}

%%% If there's a second author:
% \author{}
% \address{}
% \curraddr{}
% \email{}

%%% To have the current date inserted, use \date{\today}:
\date{}

%%% To include an abstract, uncomment the following two lines and type
%%% the abstract in between them:
\begin{abstract}
This paper is devoted to study measure solutions $\mu_t^h$ to perturbed nonlinear structured population models where $t$ denotes time and $h$ controlls the size of perturbation. We address differentiability of the map $h \mapsto \mu_t^h$. After showing that this type of results cannot be expected in the space of bounded Radon measures $\mathcal{M}(\mathbb{R}^+)$ equipped with the flat metric, we move to the slightly bigger spaces $Z = \overline{\mathcal{M}(\mathbb{R}^+)}^{(C^{1+\alpha})^*}$. We prove that when $\alpha > \frac{1}{2}$, the map $h \mapsto \mu_t^h$ is differentiable in $Z$. The proof exploits approximation scheme of a nonlinear problem from previous studies and is based on the iteration of an implicit integral equations obtained from study of the linear equation. The result shows that space $Z$ is a promising setting for optimal control of phenomena governed by such type of models.
\end{abstract}

\maketitle

%%% To include a table of contents, uncomment the following line:
\tableofcontents

%%%-------------------------------------------------------------------
%%%-------------------------------------------------------------------
%%% Start the body of the paper here!  E.G., maybe use:
%%% \section{Introduction}
%%% \label{sec:intro}

%%% For a numbered display, use
%%% \begin{equation}
%%%   \label{something}
%%%   The display goes here
%%% \end{equation}
%%% and you can refer to it as \eqref{something}.

%%% For an unnumbered display, use
%%% \begin{equation*}
%%%   The display goes here
%%% \end{equation*}

%%% To import a graphics file, you must have said
%%% \usepackage{graphicx}
%%% in the preamble (i.e., before the \begin{document}).
%%% Putting it into a figure environment enables it to float to the
%%% next page if there isn't enough room for it on the current page.
%%% The \label command must come after the \caption command.
% \begin{figure}[h]
%   \includegraphics{filename}
%   \caption{Some caption}
%   \label{somelabel}
% \end{figure}

\onehalfspacing

\section{Introduction}
Structured population models are transport--type equations used to describe dynamics of population with respect to a specific structural variable which can be selected quite arbitrarily. Previous research efforts provide models with variables ranging from age through size to cell maturity and phenotypic trait \cite{ulikowska2013structured}. This generality results in their wide applicability in demography \cite{Pe2007}, cell biology \cite{Fo1959}, immunology \cite{McK1926} or ecology \cite{Pe2007}.  Classically, analysis of structured population models was carried out in $L^1$ setting \cite{Webb, Th2003}. This approach was appropriate for considering densities of populations. However, as already suggested in \cite{MeDi1986}, it did not allow to work with less regular distributions used in applications like the Dirac mass. For conservative problems, one can try to consider solutions in the space of measures by exploiting the Wasserstein metrics like 
$$
W_1(\mu, \nu) = \sup{\int_{\mathbb{R}^+} f d(\mu - \nu)}
$$
where the supremum is taken over all Lipchitz functions with Lipschitz constant at most 1. Unfortunately, when $\mu(\mathbb{R}^+) \neq \nu(\mathbb{R}^+)$, we have $W_1(\mu, \nu) = \infty$ so it cannot be used for analysis of non-conservative problems. This issue was finally addressed in the series of papers \cite{GwiazdaThomasEtAl, GwMa2010, CarrilloGwiazda_11}, providing the complete framework for measure solutions, being elements of space of bounded, non-negative Radon measures $\mathcal{M}^+(\mathbb{R}^+)$ endowed with the so-called flat metric
\begin{equation}\label{flat_metric}
p_F (\mu, \nu) = \sup_{f \in C^1 \cap W^{1,\infty},~\norm{f}_{W^{1,\infty}} \leq 1} \int_{\mathbb{R}^+} f d(\mu - \nu),
\end{equation}
as well as well-posedness results in this setting. Note that for every $\mu, \nu \in \mathcal{M}(\mathbb{R}^+)$ we have $p_F(\mu, \nu) < \infty$, where $\mathcal{M}(\mathbb{R}^+)$ stands for the set of bounded Radon measures on $\mathbb{R}^+$.

In this paper, we consider measure solutions to the general class of nonlinear structured population models (see \cite{GwiazdaThomasEtAl} for details):
\begin{equation}\label{spm_non}
\left\{ \begin{array}{lll}
\partial_t \mu_t + \partial_x(b(x, \mu_t) \mu_t) & = c(x, \mu_t)\mu_t & \mathbb{R}^{+} \times [0,T],\\
b(0, \mu_t) D_{\lambda}\mu_t(0) &= \int_{\mathbb{R}^+} a(x,\mu_t) d\mu_t(x) & [0,T], \\
\mu_0 &= \nu & \mathbb{R}^+.
\end{array} \right.
\end{equation}
where $\nu$ is a bounded Radon measure and $D_{\lambda}\mu_t(0)$ is Radon-Nikodym derivative of $\mu_t$ with respect to the Lebesgue measure $\lambda$ on $\mathbb{R}^+$. We focus our attention on solutions $\mu_t^h$ to \eqref{spm_non} with perturbed model functions $a(x, \mu_t) := a^h(x, \mu_t^h) = a^0(x, \mu_t^h) + h a_p(x, \mu_t^h)$ and so on for functions $b$ and $c$. Then, we want to establish differentiability of the map $h \mapsto \mu_t^h$ in appropriate Fr\'echet sense. Unfortunately, the setting of flat metric cannot be used here as Example \ref{flat_metric_is_not_enough} in Section \ref{sect3} shows. Therefore, instead of working in $\mathcal{M}(\mathbb{R}^+)$, we move to the bigger space $Z = \overline{\mathcal{M}(\mathbb{R}^+)}^{(C^{1+\alpha}(\mathbb{R}^+))^*}$ that has been already successfully exploited in similar problems for transport equation \cite{1806.00357}. Under some technical assumptions to be described later, we show that map $h \mapsto \mu_t^h$ is Fr\'echet differentiable in $Z$.

This question, as already suggested in \cite{1806.00357}, is motivated by real-world applications to optimal control in phenomena described by \eqref{spm_non}. In the simplest case, one is interested in choosing value of $h$ minimizing a given functional $\mathcal{J}(t, \mu_t^h, h)$ representing some physical quantities like energy, waste production or poverty. Currently, the strategy for optimal control of structured population models concentrate on the very recent application of Excalator Boxcar Train (EBT) algorithm \cite{CGR18}. However, appropriate characterization of the derivative mentioned above, would allow to exploit gradient type algorithms like {\it steepest descent} providing conceptually easier and more efficient methods. 

The structure of the paper is as follows. In Section \ref{review_sect}, we review the theory of measure solutions to structured population models as well as recent research progress for analysis of perturbations in case of a transport equation. Then, in Section \ref{sect3}, we state and prove the main result (Theorems \ref{lin_per} and \ref{lin_per_1}) for linear equations. We also develop many estimates that are crucial for treatment of a general, nonlinear case. Finally, in Section \ref{sect4}, we state and prove main result for nonlinear model (Theorem \ref{nonlin_per}) while in Section \ref{sect5} we discuss our work as well as future perspectives.
\section{Review of useful results}\label{review_sect}
\subsection{Measure solutions to structured population model}\label{lspm_sect}
Our work is based on the concept of measure solutions and well-posedness in $(\mathcal{M}^+(\mathbb{R}^+), p_F)$ for structured population models \eqref{spm_non} developed in \cite{GwiazdaThomasEtAl}. First, for linear models:
\begin{equation}\label{spm}
\left\{ \begin{array}{lll}
\partial_t \mu_t + \partial_x(b(x) \mu_t) & = c(x)\mu_t & \mathbb{R}^{+} \times [0,T],\\
b(0) D_{\lambda}\mu_t(0) &= \int_{\mathbb{R}^+} a(x) d\mu_t(x) & [0,T], \\
\mu_0 &= \nu & \mathbb{R}^+.
\end{array} \right.
\end{equation}
we have the following concept of measure solutions:
\begin{defn}
We say that $\mu_t: [0,T] \to \mathcal{M}^+(\mathbb{R}^+)$ is a measure solution to \eqref{spm} if $\mu_t$ is narrowly continuous in time and it satisfies \eqref{spm} in the sense of distributions.
\end{defn}
We recall that narrow continuity in time means that map $[0,T] \ni t \mapsto \mu_t$ is continuous with respect to narrow convergence (i.e. in duality with bounded and continuous functions). We have the following result characterizing measure solutions to \eqref{spm}:
\begin{thm}
Suppose that  $a, b, c \in W^{1,\infty}(\mathbb{R}^+)$ where $a \geq 0$, $b(0) > 0$ and $b \in C^1(\mathbb{R}^+)$. Then, the unique measure solution $\mu_t$ to \eqref{spm} is given by identity
\begin{equation}\label{semigroup}
\int_{\mathbb{R}^+} \xi(x) d\mu_t (x)= \int_{\mathbb{R}^+} \varphi_{\xi,t}(0,x) d\nu(x) ~~~~ \text{for all } \xi \in W^{1,\infty}\cap C^1(\mathbb{R}^+),
\end{equation}
where function $\varphi_{\xi,t}(s,x)$ satisfies:
\begin{equation}\label{sol_dual}
\varphi_{\xi,t}(s,x) = \xi(X_b(t-s,x))e^{\int_0^{t-s }c(X_b(u,x)) du} + \int_0^{t-s}a(X_b(u,x))\varphi_{\xi,t}(u+s,0)e^{\int_0^u c(X_b(v,x)) dv} du
\end{equation}
and $X_b(s,x)$ is the curve solving ODE:
\begin{equation}\label{ode}
\partial_s X_b(s,x) = b(X_b(s,x)), \:\:\:\:\:\:\:\: X_b(0) = x.
\end{equation}
\end{thm}
If more model functions are considered, we will write $\varphi_{\xi,t}^{a,b,c}(s,x)$ to make our presentation clear. Formula \eqref{sol_dual} is extremely convenient for analysis of solutions to \eqref{spm}. In particular, it is easy to prove that they are Lipschitz continuous with respect to model functions $a$, $b$ and $c$ as well as the initial datum. More precisely, given two triples of functions $(a,b,c)$ and $(\overline{a\vphantom{b}}, \overline{b}, \overline{c\vphantom{b}})$ satisfying assumptions above while $\mu_t$ and $\nu_t$ are corresponding solutions of \eqref{spm} with the same initial condition $\mu_0$, then
\begin{equation}\label{cont_model}
p_F(\mu_t, \nu_t) \leq t e^{Ct} \norm[1]{\mu_0}_{TV}~\Big( \norm[1]{a - \overline{\vphantom{b}a}}_{\infty} + \norm[1]{b - \overline{b}}_{\infty}  + \norm[1]{c - \overline{\vphantom{b}c}}_{\infty} \Big)
\end{equation}
where $p_F$ is the flat metric defined in \eqref{flat_metric}, $C = C(\norm[1]{a}_{W^{1,\infty}}, \norm[1]{\overline{a\vphantom{b}}}_{W^{1,\infty}}, \norm[1]{b}_{W^{1,\infty}}, \norm[1]{\overline{b}}_{W^{1,\infty}}, \norm[1]{c}_{W^{1,\infty}}, \norm[1]{\overline{\vphantom{b}c}}_{W^{1,\infty}})$ and $\norm[1]{\mu}_{TV} = \abs{\mu}(\mathbb{R}^+)$. Moreover, for two solutions $\mu_t$ and $\nu_t$ with the same model functions starting from different initial conditions $\mu_0$ and $\nu_0$ respectively,
\begin{equation}\label{cont_initial}
p_F(\mu_t, \nu_t) \leq p_F(\mu_0, \nu_0) e^{3(\norm{a}_{W^{1,\infty}} + \norm{b}_{W^{1,\infty}} + \norm{c}_{W^{1,\infty}})t}.
\end{equation}
Finally, for any solution $\mu_t$, we have Lipschitz continuity in time:
\begin{equation}\label{cont_time}
p_F (\mu_t, \mu_{t+\Delta t}) \leq  C\Big(\norm{a}_{W^{1,\infty}}, \norm{b}_{W^{1,\infty}}, \norm{c}_{W^{1,\infty}}\Big)\abs{\Delta t} \norm{\mu_t}_{TV}.
\end{equation}
Moreover, from the proof of the bound \eqref{cont_model}, one sees (but it is also independently proven in \cite{GwMa2010}):
\begin{equation}\label{lip_pert}
\norm[1]{\varphi_{\xi,t}^{a,b,c}(s,x) - \varphi_{\xi,t}^{\overline{\vphantom{b}a}, \overline{b}, \overline{\vphantom{b}c}}(s,x)}_{\infty} \leq t e^{C(a, \overline{\vphantom{b}a}, b, \overline{b}, c, \overline{c\vphantom{b}}) t}~ \Big(\norm[1]{a - \overline{a\vphantom{b}}}_{\infty} + \norm[1]{b - \overline{b}}_{\infty}  + \norm[1]{c - \overline{\vphantom{b}c}}_{\infty} \Big).
\end{equation}
and from \eqref{sol_dual} together with \eqref{semigroup} it is easy to deduce:
\begin{equation}\label{stability}
\norm{\mu_t}_{TV} \leq e^{2(\norm{a}_{L^{\infty}} + \norm{c}_{L^{\infty}})t}
\end{equation}
(see \cite{GwiazdaThomasEtAl} for more details).
%lip_pert   GwMa2010
%together with estimates \eqref{cont_model}, \eqref{cont_initial} and \eqref{stability}
Well--posedness theory can be also established for nonlinear models of the form \eqref{spm_non}:
\begin{thm}\label{well_posedness_1_nonlinear}
Suppose that functions $a,b,c: \mathbb{R}^+ \times \mathcal{M}^+(\mathbb{R}^+) \to \mathbb{R}$ satisfy:
\begin{description}
\item[(W1)] $\sup_{\mu \in \mathcal{M}^+(\mathbb{R}^+)} \norm{a(x,\mu), b(x,\mu), c(x,\mu)}_{W^{1,\infty}} < \infty$,
\item[(W2)] for any $R > 0$ there exists $L_R$ such that for all $\mu, \nu \in \mathcal{M}^+(\mathbb{R}^+)$ with $\norm{\mu}_{TV} \leq R$ and $\norm{\nu}_{TV} \leq R$:
$$
\norm{a(x,\mu) - a(x,\nu)}_{\infty} + \norm{b(x,\mu) - b(x,\nu)}_{\infty} + \norm{c(x,\mu) - c(x,\nu)}_{\infty} \leq L_R~p_F(\mu,\nu),
$$
\item[(W3)] for all $\mu, \nu \in \mathcal{M}^+(\mathbb{R}^+)$ and $x \in \mathbb{R}^+$ $b(0,\mu)>0 $ and $a(x,\mu) \geq 0$.
\end{description}
Then, there exists the unique measure solution to \eqref{spm_non}, i.e. narrowly continuous function $\mu_t: [0,T] \to \mathcal{M}^+(\mathbb{R}^+)$ solving \eqref{spm_non} in the sense of distributions.
\end{thm}
We remark here that it is also possible to prove similar inequalities as \eqref{cont_model}, \eqref{cont_initial} and \eqref{stability} in the nonlinear case (under assumptions (W1)--(W3)) but such results will not be used in this paper.

To prove Theorem \ref{well_posedness_1_nonlinear}, one exploits the following approximating sequence in $C([0,T], (\mathcal{M}^+(\mathbb{R}^+), p_F))$. First, for fixed $k \in \mathbb{N}$, the interval $[0,T]$ is divided into $2^k$ subintervals $[m\frac{T}{2^k}, (m+1)\frac{T}{2^k}]$ where $m = 0, 1, ..., 2^k-1$. Then, when $t \in [m\frac{T}{2^k}, (m+1)\frac{T}{2^k}]$, approximation $\mu_t^k$ is defined as the unique solution to the linear equation:
\begin{equation}\label{spm_non_approx}
\left\{ \begin{array}{lll}
\partial_t \mu_t + \partial_x(b(x, \mu_{m\frac{T}{2^k}}) \mu_t) & = c(x, \mu_{m\frac{T}{2^k}})\mu_t & \mathbb{R}^{+} \times  [m\frac{T}{2^k}, (m+1)\frac{T}{2^k}],\\
b(0, \mu_{m\frac{T}{2^k}}) D_{\lambda}\mu_t(0) &= \int_{\mathbb{R}^+} a(x,\mu_{m\frac{T}{2^k}}) d\mu_t(x) &  [m\frac{T}{2^k}, (m+1)\frac{T}{2^k}], \\
\end{array} \right.
\end{equation}
with the initial measure $\mu_{m\frac{T}{2^k}}$ is obtained from solving analogous problem in the interval $[(m-1)\frac{T}{2^k}, m\frac{T}{2^k}]$.  It can be shown that $\mu_t^k$ converges to the unique solution of \eqref{spm_non} in $C([0,T], (\mathcal{M}^+(\mathbb{R}^+), p_F))$ and inequalities \eqref{cont_model}, \eqref{cont_initial} and \eqref{stability} are preserved \cite{GwiazdaThomasEtAl}.

In our presentation, it will be necessary to substitute the assumption (W1) with:
\begin{description}
\item[(W1a)] $\sup_{\mu \in \mathcal{M}^+(\mathbb{R}^+)} \norm{a(x,\mu), b(x,\mu), c(x,\mu)}_{L^{\infty}} < \infty$
\item[(W1b)] $\norm{\frac{\partial}{\partial x}a(x,\mu), \frac{\partial}{\partial x}b(x,\mu), \frac{\partial}{\partial x}c(x,\mu)}_{L^{\infty}} \leq C_1+ C_2 \norm{\mu}_{TV}$.
\end{description}
Using existence and uniqueness obtained for \eqref{spm_non} in \cite{GwiazdaThomasEtAl}, under assumptions (W1)--(W3), it is possible to prove existence and uniqueness with (W1) replaced by (W1a) and (W1b):
\begin{thm}\label{well_posedness_2_nonlinear}
Suppose that (W2) and (W3) holds true. Moreover, assume that (W1a) and (W1b) are satisfied. Then, there exists the unique measure solution to \eqref{spm_non}.
\end{thm}
For the proof of Theorem \ref{well_posedness_2_nonlinear}, we refer to Appendix \ref{ewptspm}.
\subsection{Perturbations in transport equation} A similar question, as addressed in this paper, has been already studied in \cite{1806.00357} for transport equation:
\begin{equation}\label{transport}
\partial_t \mu_t + \text{div}_x (b \mu_t) = w \mu_t.
\end{equation}
More preciesly, in \cite{1806.00357}, Authors consider the measure solutions $\mu_t^h$ to \eqref{transport} with perturbed velocity field $b(t,x) := b_0(t,x) + h b_1(t,x)$ and study Frech{\'e}t differentiability of the map $h \mapsto \mu_t^h$. The crucial issue to be dealt with is to understand which spaces one should use to address this question. The main contribution is observation that all necessary estimates can be obtained if $\mu_t^h$ is considered as an element of the space $Z = \overline{\mathcal{M}(\mathbb{R}^d)}^{(C^{1+\alpha})^*}$, i.e. closure of the space of bounded Radon measures with respect to the dual norm of $C^{1+\alpha}$. Here, for any set $X \subset \mathbb{R}^d$, space $C^{1+\alpha}(X)$ consists of functions with bounded norm:
\begin{equation}\label{1plusalpha_norm}
\norm{f}_{C^{1+\alpha}(X)} = \sup_{x \in X} \abs{f(x)} + \sup_{x\in X} \abs{Df(x)} + \sup_{x, y \in X, x\neq y} \frac{\abs{f(x) - f(y)}}{\abs{x-y}^{\alpha}}
\end{equation}
and if set $X$ is not specified, we consider the whole space. Moreover, by $\norm{\cdot}_{\alpha}$ we denote H\"older seminorm:
\begin{equation*}
\norm{f}_{\alpha} = \sup_{x, y \in X, x\neq y} \frac{\abs{f(x) - f(y)}}{\abs{x-y}^{\alpha}}.
\end{equation*}
When $f: X_1 \times X_2 \times ... \times X_n \to \mathbb{R}$ depends on $n$ variables, say $x_1, x_2, ..., x_n$, we write 
\begin{equation}\label{Holder_partial_norm}
\norm{f}_{\alpha, x_i} = \sup_{x_k \in X_k: k \neq i} \sup_{x, y \in X_i, x\neq y} \frac{\abs{f(x_1, x_2, ..., x_{i-1}, x, x_{i+1}, ..., x_n) - f(x_1, x_2, ..., x_{i-1}, y, x_{i+1}, ..., x_n)}}{\abs{x-y}^{\alpha}}.
\end{equation}
Space $Z$ is not only convenient for our objectives, but has many properties that are typical for ``good spaces" in Analysis. It was nicely characterized as:
$$
Z = \overline{\text{span}\{\delta_x : x \in \mathbb{Q}^d \}}^{(C^{1+\alpha})^*}
$$
where $\delta_x$ is the Dirac mass at point $x$, implying that $Z$ is a separable space. It was also proved that $Z^*$ is isomorphic to $C^{1+\alpha}$.

The following classical result will be of great importance:
\begin{lem}\label{l1}
For any continuously differentiable function $f$ with H\"older continuous derivative (in particular, $f \in C^{1+\alpha}$) we have:
\begin{equation*}
\abs{\frac{f(h_1+\Delta h_1) - f(h_1)}{\Delta h_1} - \frac{f(h_2+\Delta h_2) - f(h_2)}{\Delta h_2}  } \leq \frac{1}{1+\alpha}\norm{Df}_{\alpha}  \big(\abs{\Delta h_1}^{\alpha} + \abs{\Delta h_2}^{\alpha} \big) + 
\norm{Df}_{\alpha}\abs{h_1 - h_2}^{\alpha},
\end{equation*}
%H\"older
\end{lem}
\begin{proof}
Note that for any continuously differentiable function $f$ with H\"older continuous derivative on the domain of definition, one has:
\begin{multline}\label{clas_est}
f(y) = f(x) + f'(x) (x-y) + \int_0^1 \frac{d}{dt} f(ty + (1-t)x) dt  - f'(x)(x-y) = \\
		= f(x) + f'(x) (x-y) + \underbrace{\int_0^1 (f'(t (y-x) + x) - f'(x))(x-y) dt}_{\leq \int_0^1 t^{\alpha} dt~ \norm{Df}_{\alpha}~\abs{x-y}^{1+\alpha} \leq \frac{1}{1+\alpha} \norm{Df}_{\alpha} ~\abs{x-y}^{1+\alpha}  }.
\end{multline}
Applying \eqref{clas_est} twice, for $y = h_1 + \Delta h_1, x = h_1$ and $y = h_2 + \Delta h_2, x = h_2$, we directly obtain the desired inequality.
\end{proof}

\section{Properties of equation \eqref{sol_dual} and linear problem}\label{sect3}
In the following sections, we will focus on linear equations of type \eqref{spm}. We begin with standing assumptions concerning model functions. Recall that we will study solutions to equation \eqref{spm} with $a(x) = a^h(x) = a_0(x) + ha_p(x)$, $b(x) = b^h(x) = b^0(x) + hb_p(x)$, $c(x) = c^h(x) = c^0(x) + hc_p(x)$. Since we want to exploit the well-posedness theory and the setting described above we assume:
\begin{description}
\item[(A1)] $a^0, a_p, b^0, b_p, c^0, c_p \in C^{1 + \alpha}(\mathbb{R}^+)$,
\item[(A2)] $a^h = a^0 + a_p h \geq 0$ for any $h \in [-\frac{1}{2}, \frac{1}{2}]$,
\item[(A3)] $b^h(0) = b^0(0) + b_p(0) h > 0$ for any $h \in [-\frac{1}{2}, \frac{1}{2}]$.
\end{description}
We are now ready to state the main result:
\begin{thm}\label{lin_per}
Suppose assumptions (A1)--(A3) hold. Consider measure solution $\mu_t^h$ of \eqref{spm} with  $a(x) := a^h(x) = a^0(x) + ha_p(x)$, $b(x) := b^h(x) = b^0(x) + hb_p(x)$, $c(x) := c^h(x) = c^0(x) + hc_p(x)$ and $h \in [-\frac{1}{2}, \frac{1}{2}]$. Then, mapping $h \mapsto \mu_t^h$ is Fr\'echet differentiable in $C([0,T], Z)$ where $Z = \overline{\mathcal{M}(\mathbb{R}^+)}^{(C^{1+\alpha})^*}$. Moreover, Fr\'echet derivative $H \mapsto {\partial_h} \mu_t^h |_{h = H}$ is H\"older continuous with exponent $\alpha$.
\end{thm}
Actually, it will be useful to view Theorem \ref{lin_per} as a special case of a slightly more general result. We consider equation \eqref{spm} with functions that actually depends on the perturbation parameter $h$ (but we do not specify exactly how they depend on $h$):
\begin{equation}
a^h(x) := a(h,x), \;\;\;\; b^h(x) := b(h,x), \; \; \;\; c^h(x) = c(h,x)
\end{equation}
and assume:
\begin{description}
\item[(B1)] Assumptions (A2) and (A3) holds for the functions $a$ and $b$ respectively,
\item[(B2)] Functions $a(h,x)$, $b(h,x)$ and $c(h,x)$ are $C^{1 + \alpha}([-\frac{1}{2}, \frac{1}{2}] \times \mathbb{R}^+)$ in both variables (with uniform constants in second variables).
\end{description}
With this assumptions, we prove the following version of Theorem \ref{lin_per}:
\begin{thm}\label{lin_per_1}
Suppose assumptions (B1)--(B2) hold for functions $a(h,x)$, $b(h,x)$ and $c(h,x)$. Consider the measure solution $\mu_t^h$ of \eqref{spm} with  $a(x) := a(h,x)$, $b(x) := b(h,x)$, $c(x) := c(h,x)$ and $h \in [-\frac{1}{2}, \frac{1}{2}]$. Then, the mapping $h \mapsto \mu_t^h$ is Fr\'echet differentiable in $C([0,T], Z)$ where $Z = \overline{\mathcal{M}(\mathbb{R}^+)}^{(C^{1+\alpha})^*}$. Moreover, Fr\'echet derivative $H \mapsto {\partial_h} \mu_t^h |_{h = H}$ is H\"older continuous with exponent $\alpha$ and constant $G_LT$ where meaning of this constant is explained below.
\end{thm}

\begin{defn}{\bf (Constants $C_L, H_L, G_L$)}\label{constants_meaning} To make estimating procedure clear, we will always incorporate terms that are not useful anymore to the constant $C_L$ so that $C_L$ should be always understood as a value that depends continuously on length of considered time interval $T$ as well as $W^{1,\infty}$ norms of functions $a$, $b$ and $c$. For example, it is allowed to write:
$$
\norm{b}_{W^{1,\infty}} \norm{c}_{L^{\infty}}T \leq C_L \norm{c}_{L^{\infty}}T \leq C_L T \leq C_L.
$$
However, it is forbidden to write:
$$
\frac{1}{T} \leq C_L
$$
as then $C_L$ blows up when $T \to 0$. Rigorously, it can be realized by taking maximal value from two constants. Similarly, we introduce constant $H_L$ that can depend on $C_L$ and additionally on norms $\norm{f_x}_{\alpha,x}$, $\norm{f_x}_{\alpha,h}$ and $\norm{f_h}_{\alpha,x}$ where $f = a, b, c$ is a model function. Finally, constant $G_L$ should be alwasy linear combination of $1$ and norms $\norm{f_h}_{\alpha,h}$ where $f = a,b,c$ with coefficients estimated by $C_L$ and $H_L$.
\end{defn}
The reason for introducing three different constant is that during the treatment of nonlinear problem, $C_L$ will be easily controlled, $H_L$ slightly harder and the main difficulty will be to control $G_L$. This way, we avoid writing too many different terms that actually have the same effect. It will be also useful to apply the following estimating convention.
\begin{rem}\label{estimate_convention}
In this paper, we will have to estimate differences of two functions evaluated at different points using some classical bounds (Lipschitz or H\"older estimates). For instance, given functions $f$ and $\overline{f}$ defined on $\mathbb{R}^2$ we can estimate using triangle inequality:
\begin{multline*}
\abs{f(x_1, y_1) - \overline{f}(x_2,y_2)} \leq \abs{f(x_1, y_1) - \overline{f}(x_1, y_1)} + 
\abs{\overline{f}(x_1, y_1) - \overline{f}(x_2,y_1)} +
\abs{\overline{f}(x_2,y_1) - \overline{f}(x_2,y_2)} \leq \\ \leq
\norm[1]{f-\overline{f}}_{\infty} + \norm[1]{\overline{f}_x}_{\infty}\abs{x_1-x_2} +\norm[1]{\overline{f}_y}_{\infty} \abs{y_1-y_2},
\end{multline*}
assuming sufficient regularity. This amounts to considering all possible differences between $f(x_1,x_2)$ and  $\overline{f}(y_1,y_2)$. Therefore, when estimation is trivial, we will usually write only final result of estimation, always comparing terms from left to right (so in the example above, we first consider difference between $f$ and $\overline{f}$, then difference on the first variable and finally on the second). 
\end{rem}
To prove Theorem \ref{lin_per_1}, we will demonstrate that $\frac{\mu^{h+\Delta h}_t - \mu_t^h}{\Delta h}$ is a Cauchy sequence in $Z$. To this end, we start with small $\Delta h_1, \Delta h_2$ and write: 
\begin{multline}\label{Cauchy_seq_analysis}
\norm{\frac{\mu^{h + \Delta h_1}_t - \mu_t^h}{\Delta h_1} - \frac{\mu^{h + \Delta h_2}_t - \mu_t}{\Delta h_2}}_{Z} =
\sup_{\xi: \norm{\xi}_{C^{1+\alpha}} \leq 1} \int_{\mathbb{R}^+} \xi \Bigg(\frac{d\mu^{h + \Delta h_1}_t - d\mu_t^h}{\Delta h_1} - \frac{d\mu^{h +\Delta h_2}_t - d\mu_t^h}{\Delta h_2} \Bigg) =\\ = \sup_{\xi: \norm{\xi}_{C^{1+\alpha}} \leq 1} \int_{\mathbb{R}^+} \bigg( \frac{\varphi^{h + \Delta h_1}_{\xi,t}(0,x) - \varphi^{h}_{\xi,t}(0,x)}{\Delta h_1} - \frac{\varphi^{h+\Delta h_2}_{\xi,t}(0,x) - \varphi^{h}_{\xi,t}(0,x)}{\Delta h_2} \bigg) d\mu_0,
\end{multline}
where we applied semigroup property \eqref{semigroup} and adopted notation: $\varphi^h_{\xi,t} = \varphi_{\xi,t}^{a(h,\cdot), b(h,\cdot), c(h,\cdot)}$. In view of Lemma \ref{l1}, it would be sufficient to know that map $[-\frac{1}{2}, \frac{1}{2}] \ni h \mapsto \varphi^h_{\xi,t}(0,x)$ is $C^{1+\alpha}$, independently of $\xi$ and $x$. Recall, $\varphi^h_{\xi,t}$ is the unique solution of implicit equation:
\begin{equation}\label{sol_dual_dep_onh}
\varphi^h_{\xi,t}(s,x) = \xi(X_{b(h, \cdot)}(t-s,x))e^{\int_0^{t-s }c(h, X_{b(h,\cdot)}(u,x)) du} + \int_0^{t-s}a(h,X_{b(h,\cdot)}(u,x))\varphi^h_{\xi,t}(u+s,0)e^{\int_0^u c(h,X_{b(h, \cdot)}(v,x)) dv} du
\end{equation}
and our target is to obtain desired regularity of $\varphi^h_{\xi,t}(s,x)$ from this equation. Therefore, in Section \ref{study_imp_eqn}, we will study general functional identities of the form:
\begin{equation}\label{general_implicit_eqn}
f(h,s,x) = p(h,s,x) + \int_0^{t-s} q(h,u,x) f(h,s+u,0) du
\end{equation}
Then, we will check that equation for $\varphi^h_{\xi,t}(s,x)$ satisfies assumptions of the developed theory for \eqref{general_implicit_eqn} and this will lead to the proof of Theorem \ref{lin_per}.
\begin{ex}\label{flat_metric_is_not_enough}Before we start, it is instructive to see why one has to work in the space $Z$. Natural strategy, as well-posedness theory suggests, would be to study that problem in some linear space being extension of $(\mathcal{M}^+(\mathbb{R}^+), p_F)$ on the whole set of bounded Radon measures. Unfortunately, a very simple example shows that, in general, one cannot establish such results in flat metric setting. More preciesly, consider perturbed transport equation in one dimension ($a^0 = a_p = c^0 = c_p = 0$ and $b^0 = b_p = 1$ in our setting):
\begin{equation}\label{example_diff}
\partial_t \mu_t^h + \partial_x((1+h)\mu_t^h) = 0 \:\:\:\:\:\;\; \mu_0^h = \delta_0.
\end{equation}
One easily checks that $\mu_t^h = \delta_{(1+h)t}$ is a measure solution to \eqref{example_diff} (boundary condition is satisfied for a.e. $t \in [0,T]$). However, sequence $\frac{\mu_t^h - \mu_t^0}{h}$ cannot be a Cauchy sequence with respect to flat metric as for $\xi \in W^{1,\infty}(\mathbb{R}^+)$:
$$
\int_{\mathbb{R}^+}\xi(x) \Bigg( \frac{d\mu_t^{h_1}(x) - d\mu_t^0}{h_1} - \frac{d\mu_t^{h_2}(x) - d\mu_t^0}{h_2} \Bigg)  =  \frac{\xi((1+h_1)t) - \xi(t)}{h_1} - \frac{\xi((1+h_2)t) - \xi(t)}{h_2}
$$
so if we choose:
\begin{equation}\label{example_fcn}
\xi(x) =
\begin{cases} 1-|x-t| &\mbox{if } |x-t| \leq 1,  \\ 
0 & \mbox{if } |x-t| > 1 \end{cases} \in W^{1,\infty}(\mathbb{R}^+)
\end{equation}
for some $t > 1 $, we see that
$$
\lim_{h_1 \to 0^-,~ h_2 \to 0^+} \sup_{\xi \in W^{1,\infty}:~ \norm{\xi} \leq 1} \int_{\mathbb{R}^+}\xi(x) \Bigg( \frac{d\mu_t^{h_1}(x) - d\mu_t^0}{h_1} - \frac{d\mu_t^{h_2}(x) - d\mu_t^0}{h_2} \Bigg) \geq 2t > 0
$$
raising contradiction. It is important to note that it will not help to replace assumption $\xi \in W^{1,\infty}(\mathbb{R}^+)$ with $\xi \in C^1(\mathbb{R}^+)$ as one can uniformly approximate function \eqref{example_fcn} with sequence $(\xi_n) \subset C^1(\mathbb{R}^+)$ keeping condition $\norm{\xi_n}_{W^{1,\infty}} \leq 1$ satisfied. This example suggests that one should consider set of functions with uniformly continuous derivatives. 
\end{ex}
\subsection{Regularity of solutions to \eqref{general_implicit_eqn} with respect to perturbation parameter $h$}\label{study_imp_eqn}
%f(h,s,x) = p(h,s,x) + \int_0^{t-s} q(h,u,x) f(h,u+s,0) du
In this section, we study continuous solutions $f: [-\frac{1}{2}, \frac{1}{2}]\times [0,t] \times \mathbb{R}^+ \to \mathbb{R}$ to the equation \eqref{general_implicit_eqn}. We remark here that some general results on existence and uniqueness for this type of equations were obtained by Karoui \cite{karoui2005existence, karoui2005existence2} using Schauder fixed point argument. Nevertheless, our approach is simpler and aimed also at  H\"older regularity of solutions. 

Let $\mathcal{C}_t$ be the space of bounded and continuous functions on $[-\frac{1}{2}, \frac{1}{2}]\times [0,t] \times \mathbb{R}^+  \to \mathbb{R}$. In $\mathcal{C}_t$ we define subsets:
\begin{itemize}
\item $\mathcal{C}_t^h$ of functions differentiable with respect to the first variable with bounded (independently of other variables) derivative,
\item $\mathcal{C}_t^{h,\alpha}$ of functions differentiable with respect to the first variable with bounded (independently of other variables) and H\"older contrinuous derivative with respect to $h$ (continuous in $h$ and $x$).
\end{itemize}
Similarly, we introduce the subsets $\mathcal{C}_t^{x}$ and $\mathcal{C}_t^{x,\alpha}$. For functions in $\mathcal{C}_t$ we write $h$, $s$ and $x$ to denote first, second and third variable respectively. Partial derivatives with respect to $h$, $s$ and $x$ will be denoted with lower indices, for instance when $g \in \mathcal{C}_t$ we write $g_h$, $g_s$ and $g_x$. We will use norms:
\begin{itemize}
\item for $f \in \mathcal{C}_t$ we write $\norm[1]{f}_{\infty}$ (standard $L^{\infty}$ norm with respect to all variables),
\item for $f \in \mathcal{C}_t^{x}$ we write $\norm[1]{f}_{W^{1,\infty}, x} = \max{\big(\norm[1]{f}_{\infty}, \norm[1]{f_x}_{\infty}\big)}$,
\item for $f \in \mathcal{C}_t^{h}$ we write $\norm[1]{f}_{W^{1,\infty}, h} = \max{\big(\norm[1]{f}_{\infty}, \norm[1]{f_h}_{\infty}\big)}$,
\item for $f \in \mathcal{C}_t^{h} \cap \mathcal{C}_t^{x}$ we write $\norm[1]{f}_{W^{1,\infty}} = \max{\big(\norm[1]{f}_{\infty}, \norm[1]{f_x}_{\infty}, \norm[1]{f_h}_{\infty}\big)}$,
\item for $f \in \mathcal{C}_t^{h,\alpha}$ we write $\norm{f_h}_{\alpha,h}$ and $\norm{f_h}_{\alpha,x}$ as described in \eqref{Holder_partial_norm}, 
\item for $f \in \mathcal{C}_t^{x,\alpha}$ we write $\norm{f_x}_{\alpha,x}$ and $\norm{f_x}_{\alpha,h}$ as described in \eqref{Holder_partial_norm}.
\end{itemize}
We will also adopt the following notation (cf. Definition \ref{constants_meaning}):
\begin{defn}{\bf (Constants $C_q$, $H_q$, $G_q$)}
For $q \in \mathcal{C}_t^{h} \cap \mathcal{C}_t^{x}$, the constant $C_q$ is any constant depending continuously on length of time interval $t$ and $\norm{q}_{W^{1,\infty}}$. Moreover, if $q \in \mathcal{C}_t^{x,\alpha} \cap \mathcal{C}_t^{h,\alpha}$, then $H_q$ is any constant allowed to depend on $C_q$ and additionally on $\norm{q_x}_{\alpha,x}$, $\norm{q_x}_{\alpha,h}$ and $\norm{q_h}_{\alpha,x}$. Finally, $G_q$ is any constant depending on $C_q$, $H_q$ and linear combinations of the norm $\norm{q_h}_{\alpha,h}$.
\end{defn}
We start with the simple existence result:
\begin{lem}\label{lem_imp_exi}
Equation \eqref{general_implicit_eqn} with $p, q \in \mathcal{C}_t$ is uniquely solvable in $\mathcal{C}_t$. Moreover, 
\begin{equation}\label{basic_est_for_sol}
\norm{f}_{\infty} \leq \norm{p}_{\infty} e^{2\norm{q}_{\infty} t}.
\end{equation}
\end{lem}
\begin{proof}
Actually, this lemma follows from the Banach Fixed Point Theorem applied on some small interval of time together with classical extension. However, it will be instructive to see an elegant approach, exploited also in the next Lemma. For $\lambda > 0$ and $f \in \mathcal{C}_t$, we define {\it Bielecki norm}:
\begin{equation}\label{bielecki}
\norm{f}_{\lambda} = \sup_{h \in [-\frac{1}{2}, \frac{1}{2}], s\in [0,t], x \in \mathbb{R}^+} e^{-\lambda(t-s)} \abs{f(h,s,x)}.
\end{equation}
It is easy to check that this norm is equivalent to $\norm{f}_{\infty}$ on $\mathcal{C}_t$ with $\norm{f}_{\lambda} \leq \norm{f}_{\infty} \leq e^{\lambda t} \norm{f}_{\lambda}$. Let $T f$ be defined by (RHS) of \eqref{general_implicit_eqn} so that $T$ is a well-defined and bounded operator from $\mathcal{C}_t$ to $\mathcal{C}_t$. We check the contraction condition: for $f_1, f_2 \in \mathcal{C}_t$ we have
\begin{multline*}
\norm{Tf_1 - Tf_2}_{\lambda} = \sup_{h \in [-\frac{1}{2}, \frac{1}{2}], s\in [0,t], x \in \mathbb{R}^+} e^{-\lambda(t-s)}
\abs{\int_0^{t-s} q(h,u,x) \big(f_1(h,u+s,0) - f_2(h,u+s,0)\big) du} \leq \\  \leq
 \sup_{h \in [-\frac{1}{2}, \frac{1}{2}], s\in [0,t], x \in \mathbb{R}^+} e^{-\lambda(t-s)}
\abs{\int_0^{t-s} e^{\lambda\big(t-(s+u)\big)}q(h,u,x) e^{-\lambda\big(t-(s+u)\big)}\big(f_1(h,u+s,0) - f_2(h,u+s,0)\big) du} \leq \\
\leq \norm{q}_{\infty} \norm{f_1 - f_2}_{\lambda} \sup_{s\in [0,t]}
e^{-\lambda(t-s)} \int_0^{t-s}e^{\lambda\big(t-(s+u)\big)} du =
\norm{q}_{\infty} \norm{f_1 - f_2}_{\lambda} \sup_{s\in [0,t]} \int_0^{t-s}e^{-\lambda u} du =\\= \frac{\norm{q}_{\infty}}{\lambda}
(1 - e^{-\lambda t}) \norm{f_1 - f_2}_{\lambda}.
\end{multline*}
so that choosing $\lambda = \norm{q}_{\infty}$, we conclude the proof of existence and uniqueness. To see that norm of the solution $f$ in $\mathcal{C}_{t}$ depends only on $\mathcal{C}_{t}$ norms of $p$ and $q$, we compute exactly like above $\norm{Tf}_{\lambda}$ to obtain:
\begin{equation*}
\norm{Tf}_{\lambda} \leq \norm{p}_{\infty} + (1 - e^{-\lambda t})\frac{\norm{q}_{\infty}}{\lambda} \norm{f}_{\lambda}.
\end{equation*}
Since $f$ is the solution, $Tf = f$ and so, for $\lambda = \norm{q}_{\infty}$ we end up with $\norm{f}_{\infty} \leq \norm{p}_{\infty} e^{2\norm{q}_{\infty} t}$.
\end{proof}
We shall also discuss how the solution behaves when the first variable (potentially perturbation parameter) is changed. For function $g \in \mathcal{C}_{t}$, we define its pointwise variation in the first variable:
$$
\Delta_h g (h_1,h_2) = \sup_{{s\in [0,t], x\in \mathbb{R}^+}} \abs{g(h_1,s,x) - g(h_2,s,x)}
$$
and similarly we define variation in the second variable:
$$
\Delta_x g (x_1,x_2) = \sup_{{h \in [-\frac{1}{2}, \frac{1}{2}], s\in [0,t]}} \abs{g(h,s,x_1) - g(h,s,x_2)}.
$$
\begin{lem}
Let $f$ be the solution of \eqref{general_implicit_eqn} with $p, q \in \mathcal{C}_t$. Then, 
\begin{equation}\label{variation_f}
\Delta_h f(h_1, h_2) \leq e^{C_qt} \Delta_h p(h_1, h_2) + t \norm{p}_{\infty} e^{C_q t}\Delta_h q(h_1, h_2)
\end{equation}
\end{lem}
\begin{proof}
This can be deduced from Gronwall's inequality (first for $x=0$ and then for any $x \in \mathbb{R}^+$) or again, using Bielecki norm exactly like above:
\begin{equation*}
\norm{\Delta_h f(h_1, h_2)}_{\lambda} \leq
\Delta_h p(h_1, h_2) + t \norm{f}_{\infty}\Delta_h q(h_1, h_2) + (1 - e^{-\lambda t})\frac{\norm{q}_{\infty}}{\lambda} \norm{\Delta_h f(h_1, h_2)}_{\lambda}  
\end{equation*}
(we applied convention from Remark \ref{estimate_convention}). Choosing $\lambda = \norm{q}_{\infty}$ and applying $\norm{f}_{\infty} \leq e^{C_q t} \norm{f}_{\lambda}$ we obtain \eqref{variation_f}.
\end{proof}
We then focus on properties of solutions to \eqref{general_implicit_eqn} implied by Implicit Function Theorem:
\begin{lem}\label{l4}
Consider the solution $f$ of equation \eqref{general_implicit_eqn} with $p, q \in \mathcal{C}_t^h$. Then, the function $h \mapsto f(h,s,0)$ is Fr\'echet differentiable in $C[0,t]$. In particular, there exists ${\partial_h} f(h,s,0) \big|_{h = h_0} \in C[0,t]$ such that:
\begin{equation*}
\norm{\frac{f(h_0+\Delta h,s,0) - f(h_0,s,0)}{\Delta h} - {\partial_h} f(h,s,0) \big|_{h=h_0} }_{C[0,t]} \to 0 \text{     as  } \Delta h \to 0.
\end{equation*}
\end{lem}
\begin{proof}
We want to apply Implicit Function Theorem in Banach spaces (cf. Theorem 5.1.29 in \cite{denkowski2013introduction}). To this end, we write \eqref{general_implicit_eqn} in $x = 0$ as an equation to be solved in $C[0,t]$:
\begin{equation}
0 = - \phi(s) + p(h,s,0) + \int_0^{t-s} q(h,u,0) \phi(s+u) du =: F(h, \phi) 
\end{equation}
where we denoted $\phi(s) = f(h,s,0)$ and $F: [-\frac{1}{2}, \frac{1}{2}] \times C[0,T] \to C[0,T]$. Actually, in view of Lemma \ref{lem_imp_exi}, $F: [-\frac{1}{2}, \frac{1}{2}] \times B_M  \to C[0,t]$ where $B_M$ is some closed ball in $C[0,t]$. 
Moreover, it is easy to check that $F$ is Fr\'echet differentiable:
\begin{equation*}
F(h + \Delta h, \phi+\Delta \phi, ) - F(h, \phi) = DF(h,\phi)[\Delta h, \Delta \phi] + o(\Delta h, \Delta \phi)
\end{equation*}
where $DF(h,\phi): [-\frac{1}{2}, \frac{1}{2}] \times C[0,t] \to C[0,t]$ is a bounded, linear operator defined with:
\begin{multline}\label{Frechet_derivative}
DF(h,\phi)[{\bf \Delta h}, {\bf \Delta \phi}](s) = - {\bf \Delta \phi(s)} +\int_0^{t-s} q(h,0,u) {\bf \Delta \phi(s+u)} du \\ + \int_0^{t-s} q_h(h,0,u) {\phi(s+u)} du {\bf \Delta h} + p_h(h,0,u) 
{\bf \Delta h}
\end{multline}
Now, fix some $h_0 \in (-\frac{1}{2}, \frac{1}{2})$. We want to express $\phi(s)$ as a function of $h$ in some neighbourhood of $h_0$. In view of Implicit Function Theorem, we should check that the operator $DF(h_0,\phi)[0, \Delta \phi]$ is invertible. Define operator $R : C[0,t] \to C[0,t]$ with
\begin{equation*}
(R f)(s)= - f(s)+ \int_0^{t-s} q(h_0,0,u) f(u+s) du
\end{equation*}
so that $DF(h_0,\phi)[0, { \Delta \phi}] = R \Delta \phi$. In view of the Inverse Mapping Theorem, it is sufficient to check that $R$ is injective and surjective which is equivalent to the existence and uniqueness of continuous solutions to the equation $(Rf)(s) = g(s)$ for any $g \in C[0,t]$ (to see injectivity, take $f = 0$). This is implied again by Lemma \ref{lem_imp_exi}.

Therefore, Implicit Function Theorem applies: in some neighbourhood of $h = h_0$ one can find Fr\'echet differentiable function $\phi(h)(s)$ so that $F(h, \phi(h)(s)) = 0$. However, this equation is uniquely solvable with $f(h,0,s)$ solving \eqref{general_implicit_eqn} (Lemma \ref{lem_imp_exi}). We conclude that map $h \mapsto f(h,0,s)$ is Fr\'echet differentiable in $C[0,t]$.
\end{proof}
Using differentiability, we easily deduce:
\begin{lem}\label{l6}
Consider solution $f$ of the equation \eqref{general_implicit_eqn} with $p, q \in \mathcal{C}_t^{h,\alpha}$. Then, map $[-\frac{1}{2}, \frac{1}{2}] \ni H \mapsto {\partial_h} f(h,s,x) \big|_{h=H}$ is  H\"older continuous with exponent $\alpha$. Moreover, we have:
\begin{equation}\label{final_estimate_on_holder_norm}
\norm{f_h}_{\alpha,h} \leq e^{C_qt}\norm{p_h}_{\alpha,h}
+  C_q t \norm{p}_{W^{1,\infty},h} + G_q t \norm{p}_{\infty},
\end{equation}
\begin{equation}\label{final_estimate_on_holder_norm_2}
\norm{f_h}_{\alpha,x} \leq \max\big(2 e^{C_q t} \norm{p_h}_{\infty} , \norm{p_h}_{\alpha,x} + H_q t \norm{p}_{W^{1,\infty},h} \big).
\end{equation}
%\norm{p_h}_{\alpha,x} + T \norm{p}_{\mathcal{C}_t} \norm{q_h}_{\alpha,x} e^{C_q} + T \norm{p_h}_{\mathcal{C}_t}  \norm{q_x}_{\mathcal{C}_t}e^{C_q}
\end{lem}
\begin{proof}
Since we already know that $h \mapsto f(h,0,s)$ is differentiable, we obtain that the map $h \mapsto f(h,s,x)$ is differentiable (directly using \eqref{general_implicit_eqn}). We can differentiate equation \eqref{general_implicit_eqn} to obtain equation of the same type for $f_h$:
\begin{equation}\label{eq_for_f_h}
f_h(h,s,x) = \underbrace{p_h(h,s,x)+ \int_0^{t-s} q_h(h,u,x) f(h,u+s,0) du}_{:=\bar{p}(h,s,x)} 
 + \int_0^{t-s} q(h,u,x) f_h(h,u+s,0) du
\end{equation}
%\Delta_h f(h_1, h_2) \leq e^{2 \norm{q}_{\mathcal{C}_t}t} \Delta_h p(h_1, h_2) + t \norm{p}_{\mathcal{C}_t} e^{4 \norm{q}_{\mathcal{C}_t} t}\Delta_h q(h_1, h_2)
As this is the same type of equation as \eqref{general_implicit_eqn} we have by Lemma \ref{lem_imp_exi} applied to $f_h$ and $f$:
\begin{multline*}
 \norm{f_h}_{\infty} \leq \norm{\bar{p}}_{\infty} e^{C_qt} \leq 
\Big(\norm{p_h}_{\infty} + t \norm{q_h}_{\infty} \norm{f}_{\infty} \Big) e^{C_q t} \leq
\Big(\norm{p_h}_{\infty} + t \norm{q_h}_{\infty} \norm{p}_{\infty} e^{C_q t} \Big) e^{C_q t} \leq \\ \leq
\norm{p}_{W^{1,\infty},h} \Big(1 + t \norm{q_h}_{\infty}  e^{C_q t} \Big) e^{C_q t} 
\leq \norm{p}_{W^{1,\infty},h}  e^{C_qt}.
\end{multline*}
Then, we apply \eqref{variation_f} to obtain a bound for $\Delta_h f_h(h_1, h_2)$. 
\begin{multline*}
\Delta_h f_h(h_1, h_2) \leq e^{C_qt} \Delta_h \bar{p}(h_1, h_2) + t \norm{\bar{p}}_{\infty} e^{C_q t}\Delta_h q(h_1, h_2) 
\leq e^{C_q t}  \Big(\Delta_h \bar{p}(h_1, h_2) + t \norm{\bar{p}}_{\infty} \Delta_h q(h_1, h_2) \Big)  \leq \\
e^{C_q t}  \Big(\underbrace{ \Delta_h {p_h}(h_1, h_2) +  t \norm{q_h}_{\infty} \Delta_h f(h_1, h_2) +
 t \norm{f}_{\infty}  \Delta_h q_h(h_1, h_2) }_{\text{estimate for } \Delta_h \bar{p}(h_1, h_2)}
+ t \underbrace{(\norm{p_h}_{\infty} + t \norm{q_h}_{\infty} \norm{f}_{\infty})}_{\text{estimate for }\norm{\bar{p}}_{\infty}}\Delta_h q(h_1, h_2) \Big)\leq \\ 
e^{C_qt} \Big( \Delta_h {p_h}(h_1, h_2) + 
 t C_q \Delta_h f(h_1, h_2) +
 t \norm{p}_{\infty}  \Delta_h q_h(h_1, h_2) 
+ t (\norm{p_h}_{\infty} + C_q \norm{p}_{\infty}) \Delta_h q(h_1, h_2)\Big).
\end{multline*}
where we used the fact that $\norm{f}_{\infty} \leq \norm{p}_{\infty} e^{C_qt}$. Due to \eqref{variation_f} again, we have:
\begin{multline*}
\Delta_h f(h_1, h_2) \leq e^{C_qt} \Big(\Delta_h p(h_1, h_2) + t \norm{p}_{\infty} \Delta_h q(h_1, h_2) \Big) \leq 
e^{C_qt} \Big(\norm{p_h}_{\infty} + t\norm{p}_{\infty} \norm{q_h}_{\infty} \Big)\abs{h_1 - h_2} \leq \\ \leq
e^{C_qt} \Big( \norm{p_h}_{\infty} + C_q t\norm{p}_{\infty} \Big) \abs{h_1 - h_2} \leq e^{C_qt} \norm{p}_{W^{1,\infty},h} 
\end{multline*}
so that we end up with:
$$
\Delta_h f_h(h_1, h_2) \leq e^{C_qt} \Big(\norm{p_h}_{\alpha,h}
+ t C_q \norm{p}_{W^{1,\infty}, h} + tC_q\norm{p}_{\infty}\norm{q_h}_{\alpha,h}
 \Big)\abs{h_1 - h_2}^{\alpha}.
$$
To establish the inequality \eqref{final_estimate_on_holder_norm_2}, note that \eqref{eq_for_f_h} implies:
\begin{multline*}
\Delta_x f_h(x_1, x_2) \leq \Delta_x p_h(x_1, x_2) + t \norm{f}_{\infty} \Delta_x q_h(x_1, x_2) + t \norm{f_h}_{\infty} \Delta_x q(x_1,x_2) \leq \\ \leq
\norm{p_h}_{\alpha,x}\abs{x_1-x_2}^{\alpha} + t \norm{f}_{\infty} \norm{q_h}_{\alpha,x} \abs{x_1-x_2}^{\alpha}+ t \norm{f_h}_{\infty} \norm{q_x}_{\infty} \abs{x_1 - x_2}.
\end{multline*}
Since $\norm{f_h}_{\infty} < \infty$, we obtain \eqref{final_estimate_on_holder_norm_2} using bounds for $\norm{f}_{\infty}$ and $\norm{f_h}_{\infty}$. Note that we use here that bounded and Lipschitz maps are H\"older continuous. In general, this does not follow from Lipschitz continuity -- just consider linear functions.
\end{proof}
We move on to study properties of map $x \mapsto f(h,s,x)$. This will be easier to establish as there is no implicit relationship in variable $x$ in \eqref{general_implicit_eqn}.
\begin{lem}
Consider solution $f$ of equation \eqref{general_implicit_eqn} with $p, q \in \mathcal{C}_t^{x,\alpha} \cap \mathcal{C}_t^{h}$. Then:
\begin{equation}\label{est_f_x}
\norm{f_x}_{\infty} \leq \norm{p_x}_{\infty} + C_q t \norm{p}_{\infty}, \quad \quad \quad
\norm{f_x}_{\alpha,x} \leq \norm{p_x}_{\alpha,x} + H_q t \norm{p}_{\infty} ,
\end{equation}
\begin{equation}\label{est_f_x_h}
\norm{f_x}_{\alpha,h} \leq \norm{p_x}_{\alpha,h} +H_q t \norm{p}_{W^{1,\infty},h}
\end{equation}
\end{lem}
\begin{proof}
Estimates \eqref{est_f_x} follow directly from bound \eqref{basic_est_for_sol} on $f$ and the equation satisfied by $f_x$:
$$
f_x(h,s,x) = p_x(h,s,x) + \int_0^{t-s} q_x(h,u,x) f(h,u+s,0) du.
$$
To see \eqref{est_f_x_h}, we use triangle inequality:
$$
\Delta_h f_x(h_1, h_2) \leq \Delta_h p_x(h_1, h_2) + t \Delta_h q_x(h_1, h_2) \norm{f}_{\infty} + 
t \norm{q_x}_{\infty} \Delta_h f(h_1, h_2)
$$
together with estimates \eqref{basic_est_for_sol} and \eqref{variation_f}.
\end{proof}
We conclude this section with the summary of obtained results:
\begin{cor}\label{estimates_for_solutions}
Consider the solution $f$ of equation \eqref{general_implicit_eqn} with $p, q \in \mathcal{C}_t^{h,\alpha} \cap \mathcal{C}_t^{x,\alpha}$. Then, map $[-\frac{1}{2}, \frac{1}{2}] \ni H \mapsto f(h,s,x)$ satisfies: 
\begin{itemize}
\item $\norm{f_h}_{\infty} \leq e^{C_qt}\norm{p_h}_{\infty} + C_qt \norm{p}_{\infty}$
\item $\norm{f}_{W^{1,\infty}} \leq e^{C_q t}\norm{p}_{W^{1,\infty}} $,
\item $\norm{f_h}_{\alpha,x} \leq \max\Big(2e^{C_q t} \norm{p}_{W^{1,\infty}} , \norm{p_h}_{\alpha,x} + H_qt \norm{p}_{W^{1,\infty}} \Big),$
\item $\norm{f_h}_{\alpha, h} \leq e^{C_qt}\norm{p_h}_{\alpha,h}
+ G_q t\norm{p}_{W^{1,\infty}}$,
\item $\norm{f_x}_{\alpha,x} \leq \norm{p_x}_{\alpha,x} + H_q t \norm{p}_{{\infty}}  $,
\item $\norm{f_x}_{\alpha,h} \leq \norm{p_x}_{\alpha,h} + H_q t  \norm{p}_{W^{1,\infty}}$.
\end{itemize}
\end{cor}
\subsection{Properties of flow assosciated to model function $b$}
In this subsection, we study properties of map $(h,x) \mapsto X_{b(h,\cdot)}(s,x)$, i.e. solution of ODE \eqref{ode} with velocity $b(h,\cdot)$. We begin with: 
\begin{lem}\label{l0}
Let $X_{b(h, \cdot)}(s,x)$ be the solution of ODE \eqref{ode} with $b:=b(h,x)$. Then, for any fixed $T>0$, map $[-\frac{1}{2}, \frac{1}{2}] \ni H \mapsto \partial_h X_{b(h, \cdot)}(s,y)\big|_{h=H}$ is bounded by $ C_LT $ and H\"older continuous with constant $G_LT$, uniformly for $y\in \mathbb{R}^+$ and $s \in [0,T]$.
\end{lem}
\begin{proof}
This is actually standard computation already perfomed in \cite{1806.00357} but only for a particular case so we will write the whole argument. First, in view of the ODE satisfied by function $H \mapsto  {\partial_h}X_{b(h, \cdot)}(s,y)\big|_{h=H}$, or by Gronwall inequality, for any $t \in [0,T]$:
\begin{equation}\label{lip_apriori}
\abs{X_{b(h_1,\cdot)}(t,y) - X_{b(h_2,\cdot)}(t,y)} \leq T \norm{b}_{W^{1,\infty}} e^{T \norm{b}_{W^{1,\infty}} }\abs{h_1 - h_2} = C_L T \abs{h_1 - h_2} .
\end{equation}
Moreover,
\begin{multline*}
{\partial_s}{\partial_h} X_{b(h,\cdot)}(s,y) \big|_{h=h_1}
- {\partial_s}{\partial_h} X_{b(h,\cdot)}(s,y) \big|_{h=h_2} =
{\partial_h} b(h,X_{b(h,\cdot)}(s,y)) \big|_{h=h_1} -
{\partial_h} b(h,X_{b(h,\cdot)}(s,y))  \big|_{h=h_2}  \\= 
\underbrace{
b_x(h_1, X_{b(h_1, \cdot)}(s,y)) ~{\partial_h} X_{b(h, \cdot)}(s,y) \big|_{h=h_1} -
 b_x(h_2, X_{b(h_2, \cdot)}(s,y))  ~{\partial_h} X_{b(h, \cdot)}(s,y) \big|_{h=h_2}}
_{~\leq ~\norm{b_x}_{\alpha,h}~ \abs{h_1 - h_2}^{\alpha}~+~\norm{b_x}_{\alpha,x}\big(T~C_L~\abs{h_1 - h_2}\big)^{\alpha}  +\norm{b}_{W^{1,\infty}} ~\abs{{\partial_h} X_{b(h,\cdot)}(s,y) |_{h=h_1} - 
 {\partial_h} X_{b(h,\cdot)}(s,y) |_{h=h_2}~}} + \\
+ \underbrace{ b_h (h_1, X_{b(h_1,\cdot)}(s,y)) - 
b_h (h_2, X_{b(h_2,\cdot)}(s,y))}_
{\leq~\norm{b_h}_{\alpha,h}~\abs{h_1 - h_2}^{\alpha}+~ \norm{b_h}_{\alpha,x}~\big(T~C_L~\abs{h_1 - h_2}\big)^{\alpha}  } 
\end{multline*}
After integrating in time \big(note that ${\partial_h} X_{b(h,\cdot)}(0,y) \big|_{h=h_1} = {\partial_h} X_{b(h,\cdot)}(0,y) \big|_{h=h_2}$\big) and applying Gronwall inequality, we conclude the proof.
\end{proof}
We can formulate a similar result for dependence in the spatial variable $x$:
\begin{lem}\label{l0_wx}
Let $X_{b(h, \cdot)}(s,x)$ be the solution of ODE \eqref{ode} with $b:=b(h,x)$. Then, for any fixed $T>0$, map $[-\frac{1}{2}, \frac{1}{2}] \ni x \mapsto {\partial_y} X_{b(h, \cdot)}(s,y)\big|_{y=x}$ is bounded by $e^{C_L T}$ and H\"older continuous (with respect to $h$) with constant $H_{L} T$, uniformly for $h \in [-\frac{1}{2}, \frac{1}{2}]$ and $s \in [0,T]$.
%e^{T \norm{b}_{W^{1,\infty}}}, C_{L} T\norm{D b}_{\alpha}e^{C_{L} T}
\end{lem}
\begin{proof}
Derivative ${\partial_y} X_{b(h, \cdot)}(s,y)|_{y = x}$ satisfies the ODE:
\begin{equation}\label{expression_for_der}
{\partial_s}{\partial_y} X_{b(h, \cdot)}(s,y)\Big|_{y = x}
  = {\partial_y}b(h, X_{b(h, \cdot)}(s,y)) \Big|_{y = x} =
  {b_x}(h, X_{b(h, \cdot)}(s,y)) {\partial_y} X_{b(h, \cdot)}(s,y)|_{y=x} 
\end{equation}
with initial condition ${\partial_y} X_{b(h, \cdot)}(0,x)\Big|_{y = x} = 1$. In view of Gronwall differential inequality, $\abs{{\partial_y} X_{b(h, \cdot)}(s,y)|_{y = x}} \leq e^{T \norm{b}_{W^{1,\infty}}} = e^{C_LT }$.  To see H\"older continuity we write:
\begin{multline*}
{\partial_s}{\partial_y} X_{b(h_1,\cdot)}(s,y) \big|_{y=x}
- {\partial_s}{\partial_y} X_{b(h_2,\cdot)}(s,y) \big|_{y=x} =
{\partial_y} b(h_1,X_{{b(h_1,\cdot)}}(s,y)) \big|_{y=x} -
{\partial_y} b(h_2,X_{{b(h_2,\cdot)}}(s,y))  \big|_{y=x} = \\= 
\underbrace{
b_x (h_1, X_{b(h_1, \cdot)}(s,y))  ~{\partial_y} X_{b(h_1, \cdot)}(s,y) \big|_{y=x} -
b_x(h_2, X_{b(h_2, \cdot)}(s,y))  ~{\partial_y} X_{b(h_2, \cdot)}(s,y) \big|_{y=x}}
_{\leq \norm{b_x}_{\alpha,h}~ \abs{h_1 - h_2}^{\alpha} e^{C_LT} + 
\norm{b_x}_{\alpha,x}~\big( \abs{h_1 - h_2}~ T~ C_L \big)^{\alpha}e^{C_LT}
+ \norm{b}_{W^{1,\infty}} ~\abs{~{\partial_y} X_{b(h_1,\cdot)}(s,y) |_{y=x} - 
 {\partial_y} X_{b(h_2, \cdot)}(s,y) |_{y=x}~}}
\end{multline*}
using Lemma \ref{l0} (map $h \mapsto X_{b^h}(s,y)$ is Lipschitz with constant $C_LT $). After integrating in time and applying Gronwall inequality we conclude the proof.
\end{proof}
We finally study continuity of derivatives with respect to spatial variable:
\begin{lem}
Let $X_{b(h, \cdot)}(s,x)$ be the solution of ODE \eqref{ode} with $b:=b(h,x)$. Then, for any $s \in [0,T]$:
\begin{equation}\label{hol_est_y_y}
\abs[1]{{\partial_y} X_{b(h, \cdot)}(s,y)\big|_{y = x_1} - {\partial_y} X_{b(h, \cdot)}(s,y)\big|_{y = x_2}} \leq H_L T \abs[1]{x_1 - x_2}^{\alpha}
\end{equation}
\begin{equation}
\abs[1]{{\partial_h} X_{b(h, \cdot)}(s,x_1) - {\partial_h} X_{b(h, \cdot)}(s,x_2)} \leq H_L T  \abs[1]{x_1-x_2}^{\alpha}
\end{equation}
\end{lem}
\begin{proof}
Using \eqref{expression_for_der}, we obtain:
\begin{multline*}
{\partial_s}{\partial_y} X_{b(h, \cdot)}(s,y)\big|_{y = x_1} - 
  {\partial_s}{\partial_y} X_{b(h, \cdot)}(s,y)\big|_{y = x_2} = \\ = 
  \underbrace{b_x(h, X_{b(h, \cdot)}(s,x_1)) ~{\partial_y} X_{b(h, \cdot)}(s,y)|_{y=x_1}  - 
  b_x(h, X_{b(h, \cdot)}(s,x_2)) ~{\partial_y} X_{b(h, \cdot)}(s,y)|_{y=x_2}}_{
  \norm{b_x}_{\alpha,x}~\big(e^{C_L T}~\abs{x_1 - x_2} \big)^{\alpha}~e^{C_L T}+~\norm{b}_{W^{1,\infty}}~\abs{{\partial_y} X_{b(h, \cdot)}(s,y)|_{y=x_1}  - 
  {\partial_y} X_{b(h, \cdot)}(s,y)|_{y=x_2} }
  }
\end{multline*}
where we used $\abs[1]{{\partial_y} X_{b(h, \cdot)}(s,y)} \leq e^{C_LT}$. Application of Gronwall inequality yields \eqref{hol_est_y_y}. Similarly,
\begin{multline*}
{\partial_s}{\partial_h} X_{b(h,\cdot)}(s,y_1) 
- {\partial_s}{\partial_h} X_{b(h,\cdot)}(s,y_2) =
{\partial_h} b(h,X_{b(h,\cdot)}(s,y_1)) -
{\partial_h} b(h,X_{b(h,\cdot)}(s,y_2)) = \\= 
\underbrace{
b_x (h, X_{b(h, \cdot)}(s,y_1))  ~{\partial_h} X_{b(h, \cdot)}(s,y_1)  -
b_x (h, X_{b(h, \cdot)}(s,y_2))  ~{\partial h} X_{b(h, \cdot)}(s,y_2)}
_{~\leq \norm{b_x}_{\alpha,x} e^{\alpha T \norm{b}_{W^{1,\infty}}}~\abs{y_1 - y_2}^{\alpha} +\norm{b}_{W^{1,\infty}} ~\abs{~{\partial_h} X_{b(h,\cdot)}(s,y_1) - {\partial_h} X_{b(h,\cdot)}(s,y_2)}}+ \\
+ \underbrace{ b_h (h, X_{b(h,\cdot)}(s,y_1)) - 
b_h (h,X_{b(h,\cdot)}(s,y_2))}_
{\leq~\norm{b_h}_{\alpha,x}~e^{\alpha \norm{b}_{W^{1,\infty}}T}~\abs{y_1 - y_2}^{\alpha}},
\end{multline*}
concluding the proof.
\end{proof}
We summarize this section with obtained results:
\begin{cor}\label{summary_flow}
Let $X_{b(h, \cdot)}(s,x)$ be the solution of ODE \eqref{ode} with $b:=b(h,x)$. Then, for any $s \in [0,T]$:
\begin{itemize}
\item $\abs{X_{b(h_1,\cdot)}(t,y) - X_{b(h_2,\cdot)}(t,y)} \leq  C_LT \abs{h_1 - h_2}.$
\item $\abs{X_{b(h,\cdot)}(t,y_1) - X_{b(h,\cdot)}(t,y_2)} \leq e^{C_L T}\abs{y_1 - y_2}.$
\item $ \abs{{\partial_h}X_{b(h, \cdot)}(s,y)\big|_{h=h_1} - {\partial_h}X_{b(h, \cdot)}(s,y)\big|_{h=h_2}} \leq  G_LT \abs{h_1 - h_2}^{\alpha}$
\item $\abs{{\partial_y} X_{b(h_1, \cdot)}(s,y) - {\partial_y} X_{b(h_2, \cdot)}(s,y)} \leq H_L  T \abs{h_1 - h_2}^{\alpha}$
\item $\abs{{\partial_y} X_{b(h, \cdot)}(s,y)\big|_{y = x_1} - {\partial_y} X_{b(h, \cdot)}(s,y)\big|_{y = x_2}} \leq H_L   T \abs{x_1 - x_2}^{\alpha}$
\item $\abs{{\partial_h} X_{b(h, \cdot)}(s,x_1) - {\partial_h} X_{b(h, \cdot)}(s,x_2)} \leq  H_L T \abs{x_1-x_2}^{\alpha}$
\end{itemize}
\end{cor}
\subsection{Application of theory for \eqref{general_implicit_eqn} to structured population models}\label{taking_particular_function}
In this subsection, we demonstrate how the theory developed for implicit equation \eqref{general_implicit_eqn} can be applied to the setting of Theorem \ref{lin_per_1}. This is slightly technical as we finally arrive at specific functions $p$ and $q$. On the other hand, computations in this Section are fairly standard, and as such, they are mostly presented in Appendix \ref{taking_particular_function_comp}. Here, we list main results to present the flow of ideas. 

To finally characterize solutions to \eqref{sol_dual_dep_onh} we would like to use Corollary \ref{estimates_for_solutions}. To this end, let
\begin{equation}\label{def_of_P}
P^{a,b,c}(h,s,x) =  \xi(h, X_{b(h, \cdot)}(t-s,x))e^{\int_0^{t-s }c(h, X_{b(h,\cdot)}(u,x)) du}, 
\end{equation}
\begin{equation}\label{def_of_Q}
Q^{a,b,c}(h,s,x) = a(h,X_{b(h,\cdot)}(s,x))e^{\int_0^s c(h,X_{b(h, \cdot)}(v,x)) dv},
\end{equation}
and we need to establish bounds for $P^{a,b,c}$ and $Q^{a,b,c}$ used in Corollary \ref{estimates_for_solutions}. Note that the form of the function $P$ is slightly more general (function $\xi$ in equation \eqref{sol_dual_dep_onh} does not depend on $h$) however this generalized setting will be needed later. The following results are proven in Appendix \ref{taking_particular_function_comp}.
\begin{cor}\label{Q_est}
The function $Q^{a,b,c}$ satisfies $\norm[1]{Q^{a,b,c}}_{\infty}, \norm[1]{Q^{a,b,c}_x}_{\infty}, \norm[1]{Q^{a,b,c}_h}_{\infty} \leq C_L$; $\norm[1]{Q^{a,b,c}_h}_{\alpha,x}$, $\norm[1]{Q^{a,b,c}_x}_{\alpha,x}$, $\norm[1]{Q^{a,b,c}_x}_{\alpha,h} \leq H_L$ and $\norm[1]{Q^{a,b,c}_h}_{\alpha,h} \leq G_L$.
\end{cor}
\begin{cor}\label{P_est}
The function $P^{a,b,c}$ satisfies:
\begin{itemize}
\item $\norm[1]{P^{a,b,c}_h}_{\infty} \leq \norm{\xi_h} + C_LT \big(\norm{\xi_x}_{\infty} + \norm{\xi}_{\infty} \big)$,
\item $\norm[1]{P^{a,b,c}}_{W^{1,\infty}} \leq e^{C_LT}\norm[1]{\xi}_{W^{1,\infty}} $,
\item $\norm[1]{P^{a,b,c}_h}_{\alpha,h} \leq G_L T\norm[1]{\xi}_{W^{1,\infty}} +e^{C_LT} \norm[1]{\xi_h}_{\alpha,h} +  C_{L}T^{\alpha} \norm[1]{\xi_h}_{\alpha,x} + C_LT\norm[1]{\xi_x}_{\alpha,h}+ C_LT\norm[1]{\xi_x}_{\alpha,x}$,
\item $\norm[1]{P^{a,b,c}_h}_{\alpha,x} \leq \max{\big(e^{C_LT}\norm[1]{\xi}_{W^{1,\infty}}  , ~e^{C_LT}\norm[1]{\xi_h}_{\alpha,x}  + C_L T \norm[1]{\xi_x}_{\alpha,x} + H_L T \norm[1]{\xi}_{W^{1,\infty}}\big)}$ ,
\item $\norm[1]{P^{a,b,c}_x}_{\alpha,h} \leq
e^{C_LT}\norm[1]{\xi_x}_{\alpha,h} +C_L T^{\alpha}\norm[1]{\xi_x}_{\alpha,x}  + H_LT\norm[1]{\xi}_{W^{1,\infty}}$,
\item $\norm[1]{P^{a,b,c}_x}_{\alpha,x} \leq \max{\big( e^{C_LT}\norm[1]{\xi}_{W^{1,\infty}} , ~e^{C_LT} \norm[1]{\xi_x}_{\alpha,x}  + H_LT \norm[1]{\xi}_{W^{1,\infty}}\big)}$.
\end{itemize}
\end{cor}
We apply directly Corollary \ref{estimates_for_solutions} together with the estimates for $P^{a,b,c}$ provided by Corollary \ref{P_est} and the estimates for $Q^{a,b,c}$ provided by Corollary \ref{Q_est}:
\begin{lem}\label{est_dual_sol_fin_fcns}
Let $\varphi_{\xi}$ be the solution to \eqref{sol_dual_dep_onh} with $p:=P^{a,b,c}$ and $q =  Q^{a,b,c}$. Then $\varphi_{\xi}$ satisfies:
\begin{itemize}
\item $\norm{\varphi_{\xi, h}}_{\infty} \leq \norm{\xi_h}_{\infty} + C_LT \big(\norm{\xi_x}_{\infty} + \norm{\xi}_{\infty}\big), $
\item $\norm{\varphi_{\xi}}_{W^{1,\infty}} \leq e^{C_LT}\norm{\xi}_{W^{1,\infty}} ,$
\item $\norm{\varphi_{\xi, h}}_{\alpha, h} \leq
e^{C_LT}\norm{\xi_h}_{\alpha,h} +G_L T\norm{\xi}_{W^{1,\infty}} + C_{L} T^{\alpha} \norm{\xi_h}_{\alpha,x} + C_LT\norm{\xi_x}_{\alpha,h}+ C_LT\norm{\xi_x}_{\alpha,x},$
\item $\norm{\varphi_{\xi,x}}_{\alpha,x}\leq
e^{C_LT} \max{\big(\norm{\xi}_{W^{1,\infty}} , ~\norm{\xi_x}_{\alpha,x} \big)}  +H_L T  \norm{\xi}_{W^{1,\infty}},$
\item $\norm{\varphi_{\xi,x}}_{\alpha,h}\leq  e^{C_LT}\norm{\xi_x}_{\alpha,h} + C_L T^{\alpha}\norm{\xi_x}_{\alpha,x}  +H_L T\norm{\xi}_{W^{1,\infty}},$
\item $\norm{\varphi_{\xi,h}}_{\alpha,x} \leq
e^{C_L T} \max\big(2 \norm{\xi}_{W^{1,\infty}} , ~\norm{\xi_h}_{\alpha,x}\big) +  C_L T \norm{\xi_x}_{\alpha,x} + H_L T \norm{\xi}_{W^{1,\infty}}.$
\end{itemize}
\end{lem}
An easy implication of Lemma \ref{est_dual_sol_fin_fcns} is
\begin{cor}\label{Holder_cont_phi}
Let $\varphi_{\xi}$ be the solution to \eqref{sol_dual_dep_onh} with 
$$
p(h,s,x) =  \xi(X_{b(h, \cdot)}(t-s,x))e^{\int_0^{t-s }c(h, X_{b(h,\cdot)}(u,x)) du}
$$ and $q =  Q^{a,b,c}$ where $\xi \in C^{1+\alpha}(\mathbb{R}^+)$, $\norm{\xi}_{C^{1+\alpha}(\mathbb{R}^+)} \leq 1$ (note that $\xi$ does not depend on $h$ here). Then, $\norm{\varphi_{\xi,h}}_{\infty} \leq C_LT$, $\norm{\varphi_{\xi}}_{W^{1,\infty}} \leq  e^{C_LT}$, $\norm{\varphi_{\xi, h}}_{\alpha, h} \leq G_LT$, $\norm{\varphi_{\xi,x}}_{\alpha,x}\leq e^{C_LT + H_LT}$ and $\norm{\varphi_{\xi,x}}_{\alpha,h}\leq H_L T^{\alpha}$.
\end{cor}
\begin{proof}
All assertions follow directly from Lemma \ref{est_dual_sol_fin_fcns} by noting that $\xi$ does not depend on $h$ and $\norm{\xi}_{\infty}$, $\norm{\xi_x}_{\infty}$, $\norm{\xi_x}_{\alpha,x} \leq 1$.
\end{proof}
A direct consequence of this estimate is the proof of Theorem \ref{lin_per_1}:
\begin{proof}[Proof of Theorem \ref{lin_per_1}]
As already suggested, we want to show that for some fixed $h$, $\frac{\mu^{h+\Delta h}_t - \mu^h_t}{\Delta h}$ is a Cauchy sequence in $Z$ when $\Delta h \to 0$. To this end,
\begin{multline}\label{Cauchy_seq_analysis_2}
\norm{\frac{\mu^{h + \Delta h_1}_t - \mu_t^h}{\Delta h_1} - \frac{\mu^{h+\Delta h_2}_t - \mu_t^h}{\Delta h_2}}_{Z} =
\sup_{\xi: \norm{\xi}_{C^{1+\alpha}} \leq 1} \int_{\mathbb{R}^+} \xi \bigg(\frac{d\mu^{h + \Delta h_1}_t - d\mu_t^h}{\Delta h_1} - \frac{d\mu^{h + \Delta h_2}_t - d\mu_t^h}{\Delta h_2} \bigg) = \\
= \sup_{\xi: \norm{\xi}_{C^{1+\alpha}} \leq 1} \int_{\mathbb{R}^+} \bigg( \frac{\varphi^{h + \Delta h_1}_{\xi,t}(0,x) - \varphi^{h}_{\xi,t}(0,x)}{\Delta h_1} - \frac{\varphi^{h+\Delta h_2}_{\xi,t}(0,x) - \varphi^{h}_{\xi,t}(0,x)}{\Delta h_2} \bigg) d\mu_0.
\end{multline}
where we applied semigroup property \eqref{semigroup} and adopted notation $\varphi^h_{\xi,t} = \varphi_{\xi,t}^{a(h, \cdot), b(h, \cdot), c(h, \cdot)}$. Now, in view of Corollary \ref{Holder_cont_phi}, map $h \mapsto\varphi^{h}_{\xi,t}(0,x)$ is continuously differentiable with H\"older continuous derivative, independently of $x$ and $\xi$ such that $\norm{\xi}_{C^{1+\alpha}} \leq 1$. Therefore, by Lemma \ref{l1}:
\begin{equation*}
\norm{\frac{\mu^{h + \Delta h_1}_t - \mu_t^h}{\Delta h_1} - \frac{\mu^{h+\Delta h_2}_t - \mu_t^h}{\Delta h_2}}_{Z} \leq G_L T\abs{\Delta h_1}^{\alpha} + G_L T\abs{\Delta h_2}^{\alpha}  + G_L T \abs{\Delta h_1 - \Delta h_2} ^{\alpha} \to 0
\end{equation*}
as $\Delta h_1 \to 0 \text{ and } \Delta h_2 \to 0$. This implies, by completeness of $Z$, existence of Fr\'echet derivative ${\partial_h} \mu_t^h$ for $t \in [0,T]$. 

To see that map $H \mapsto {\partial_h} \mu_t^h|_{h=H}$ is H\"older continuous, we write:
\begin{multline*}
\norm{{\partial_h} \mu_t^h|_{h=h_1} -  {\partial_h} \mu_t^h|_{h=h_2}}_Z = \lim_{\Delta h \to 0} \norm{\frac{\mu_t^{h_1 + \Delta h} - \mu_t^{h_1}}{\Delta h} - \frac{\mu_t^{h_2 + \Delta h} - \mu_t^{h_2}}{\Delta h}}_{Z}= \\= 
\lim_{\Delta h \to 0} \sup_{\xi: \norm{\xi}_{C^{1+\alpha}} \leq 1} \int_{\mathbb{R}^+} \bigg( \frac{\varphi^{h_1 + \Delta h}_{\xi,t}(0,x) - \varphi^{h_1}_{\xi,t}(0,x)}{\Delta h} - \frac{\varphi^{h_2+\Delta h}_{\xi,t}(0,x) - \varphi^{h_2}_{\xi,t}(0,x)}{\Delta h} \bigg) d\mu_0 \leq \\ \leq
\lim_{\Delta h \to 0} \sup_{\xi: \norm{\xi}_{C^{1+\alpha}} \leq 1} C\abs{\Delta h}^{\alpha} + C \abs{h_1 - h_2}^{\alpha}
\end{multline*}
using Lemma \ref{l1} with constant $C$ estimated by $G_LT$ due to Corollary \ref{Holder_cont_phi}.
\end{proof}
We conclude this section with a technical estimate that will be useful later. It is proven in Appendix \ref{taking_particular_function_comp}.
\begin{lem}\label{kicked_continuity}
Let $\varphi^h_{\xi,t}(s,x)$ be the solution to \eqref{sol_dual_dep_onh} with $\xi \in C^{1+\alpha}(\mathbb{R}^+)$ (note carefully that $\xi$ does not depend on $h$ again). Then $$\norm[1]{\varphi^h_{\xi,t}(s,x) - \xi(x)}_{W^{1,\infty}(\mathbb{R}^+), x} \leq C\big(C_L, \norm[1]{\xi}_{C^{1+\alpha}(\mathbb{R}^+)}\big) \abs{t}^{\alpha}.$$
\end{lem}
%\begin{equation}\label{sol_dual_dep_onh}
% = \xi(X_{b(h, \cdot)}(t-s,x))e^{\int_0^{t-s }c(h, X_{b(h,\cdot)}(u,x)) du} + \int_0^{t-s}a(h,X_{b(h,\cdot)}(u,x))\varphi^h(0,u+s)e^{\int_0^u c(h,X_{b(h, \cdot)}(v,x)) dv} du
%\end{equation}
%sol_dual_dep_onh
\subsection{Effect of perturbation in model functions}\label{perturbation_mdlfcns_section}
In this section, we study what happens to the solutions of \eqref{general_implicit_eqn} with $p$ and $q$ given by \eqref{def_of_P} and \eqref{def_of_Q}, when functions $\xi$, $a$, $b$ and $c$ are perturbed. More precisely, we consider two triples of model functions $(a, b, c)$, $(\bar{\vphantom{b}a}, \bar{b}, \bar{\vphantom{b}c})$ as well as two functions $\xi$ and $\bar{\xi}$. 
Given two pairs of times $(s_1, t_1)$ and $(s_2, t_2)$ such that $\abs{\Delta t} := t_2 - s_2 = t_1 - s_1$, we consider solutions ${\varphi}^h_{\xi,t_1}$ and $\bar{\varphi}^h_{\xi,t_2}$ to the following equations:
\begin{multline}\label{f1_do_szacow}
\varphi^h_{\xi,t_1}(s_1+w,x) = \xi(h,X_{b(h, \cdot)}(t_1-s_1-w,x))e^{\int_0^{t_1-s_1-w }c(h, X_{b(h,\cdot)}(u,x)) du} +\\+ \int_0^{t_1-s_1-w}a(h,X_{b(h,\cdot)}(u,x))\varphi^h_{\xi,t_1}(u+s_1+w,0)e^{\int_0^u c(h,X_{b(h, \cdot)}(v,x)) dv} du,
\end{multline}
\begin{multline}\label{f2_do_szacow}
\bar{\varphi}^h_{\bar{\xi},t_2}(s_2+w,x) = \bar{\xi}(h,X_{\bar{b}(h, \cdot)}(t_2-s_2-w,x))e^{\int_0^{t_2-s_2-w}\bar{c}(h, X_{\bar{b}(h,\cdot)}(u,x)) du} + \\+ \int_0^{t_2-s_2-w}\bar{a}(h,X_{\bar{b}(h,\cdot)}(u,x))\bar{\varphi}^h_{\bar{\xi},t_2}(u+s_2+w,0)e^{\int_0^u \bar{c}(h,X_{\bar{b}(h, \cdot)}(v,x)) dv} du,
\end{multline}
where $w \in [0,\abs{\Delta t}]$. Let $\mathcal{D} = [-\frac{1}{2}, \frac{1}{2}] \times \mathbb{R}^+\times  [0,\abs{\Delta t}]$. For $u \in [0, \abs{\Delta t}]$ we define:
$$
p(h,u,x) = \xi(h,X_{b(h, \cdot)}(\abs{\Delta t}-u,x))e^{\int_0^{\abs{\Delta t}-u }c(h, X_{b(h,\cdot)}(u,x)) du},
$$
$$
{q}(h,u,x) ={a}(h,X_{{b}(h,\cdot)}(u,x))e^{\int_0^u {c}(h,X_{{b}(h, \cdot)}(v,x)) dv}
$$
as well as $\bar{p}$ and $\bar{q}$ with $a, b, c, \xi$ replaced by $\bar{\vphantom{b}a}, \bar{b}, \bar{\vphantom{b}c}, \bar{\vphantom{b}\xi}$ respectively. The aim of this chapter is to estimate differences of functions:
$$
\sup_{(h,x,w) \in \mathcal{D}}\abs{\varphi^h_{\xi,t_1}(s_1+w,x) - \bar{\varphi}^h_{\bar{\xi},t_2}(s_2+w,x)}
$$
as well as their derivatives in terms of corresponding values for $\xi - \bar{\xi}$. 

\begin{defn}\label{def_of_Delta_F}
We write $\abs{\Delta f} = \max \big(\norm{a-\bar{a}}_{\infty}, \norm{b-\bar{b}}_{\infty}, \norm{c-\bar{c}}_{\infty} \big)$. Similarly, we will use ${\abs{\Delta f_x}}$ and ${\abs{\Delta f_h}}$ to denote the differences of derivatives. 
\end{defn}
Moreover, we maintain our notation concerning constants $C_L$, $H_L$ and $G_L$ using it simultaneously for the two triples of functions $(a,b,c)$ and $(\bar{\vphantom{b}a}, \bar{b}, \bar{\vphantom{b}c})$ (cf. Definition \ref{constants_meaning}).
Throughout this chapter, we assume:
\begin{description}
\item[(C1)] functions $(a,b,c)$ and $(\bar{\vphantom{b}a}, \bar{b}, \bar{\vphantom{b}c})$ satisfy assumptions (B1)--(B2) of Theorem \ref{lin_per_1},
\item[(C2)] $\abs{\Delta f} \leq C_L \abs{\Delta t}$,
\item[(C3)] $\abs{\Delta f_x} \leq C_L \abs{\Delta t}^{\alpha}$.
\end{description}
The main result of this Section reads:
\begin{thm}\label{thm_main_diff}
Suppose (C1) -- (C3) hold true. Then
\begin{equation}\label{direct_estimate_on_fcns_dif_eqn}
\sup_{(h,x,w) \in \mathcal{D}}\abs{\varphi^h_{\xi,t_1}(s_1+w,x) - \bar{\varphi}^h_{\bar{\xi},t_2}(s_2+w,x)} \leq e^{C_L\abs{\Delta t}} \norm{\xi - \bar{\xi}}_{\infty} + e^{C_L\abs{\Delta t}} \abs{\Delta t}^2\big(2 + \norm{\bar{\xi}}_{\infty}\big),
\end{equation}
\begin{multline}\label{precise_estimate_diff_dx_eq}
\sup_{(h,x,w) \in \mathcal{D}}\abs{{\partial_x}\varphi^h_{\xi,t_1}(s_1+w,x) - {\partial_x}\bar{\varphi}^h_{\bar{\xi},t_2}(s_2+w,x)} \leq e^{C_L\abs{\Delta t}} \norm{\xi_x - \bar{\xi}_x}_{\infty} + C_L  \abs{\Delta t}^{2\alpha} \norm{\bar{\xi}_{x}}_{\alpha,x}  + \\ +H_L \abs[1]{\Delta t} \big( \abs[1]{\Delta t}^{\alpha} + \norm[1]{\xi - \bar{\xi}}_{\infty} \big) 
\max{\big(1, \norm[1]{\xi}_{W^{1,\infty}}, \norm[1]{\bar{\xi}}_{W^{1,\infty}} \big)} ,
\end{multline} 
\begin{multline}\label{precise_estimate_diff_dh_eq}
\sup_{(h,x,w) \in \mathcal{D}}\abs{{\partial_h}\varphi^h_{\xi,t_1}(s_1+w,x) - {\partial_h}\bar{\varphi}^h_{\bar{\xi},t_2}(s_2+w,x)} \leq e^{C_L\abs{\Delta t}} \norm{\xi_h - \bar{\xi}_h}_{\infty} + C_L \abs{\Delta t}\norm{\xi_x - \bar{\xi}_x}_{\infty} + \\ +C_L \abs{\Delta t} \norm{\xi - \bar{\xi}}_{\infty} +  
C_L \abs{\Delta t}^{2\alpha}  \norm{\bar{\xi}_h}_{\alpha,x} +
C_L  \abs{\Delta t}^{1+2\alpha} \norm{\bar{\xi}_x}_{\alpha,x} +\\ +
C_L \abs{\Delta t} \abs{\Delta f_h}  \max(\norm[1]{\xi}_{W^{1,\infty}}, \norm[1]{\bar{\xi}}_{W^{1,\infty}}) +
H_L \abs{\Delta t}^{1+2\alpha}\max{(1, \norm[1]{\xi}_{W^{1,\infty}}, \norm[1]{\bar{\xi}}_{W^{1,\infty}})}.
\end{multline}
\end{thm}
Proof of Theorem \ref{thm_main_diff} is a standard computation performed in Appendix \ref{perturbation_mdlfcns_app}. Observe that this result allows to control distance between two iterations of solutions to implicit equations \eqref{f1_do_szacow} and \eqref{f2_do_szacow}. It will be applied in Section \ref{uniform_convergence_section}. 
\section{Nonlinear problem}\label{sect4}
In this section, we finally move to the nonlinear equation \eqref{spm_non}. The first issue to be dealt with is a dependence on measure in functions $a$, $b$ and $c$ in \eqref{spm_non}. Since measures cannot be evaluated pointwise, we consider quite general class of such models where dependence on measure is actually realized by testing against some nice kernel. More precisely, for a function $f(x,\mu)$, we consider representation:
\begin{equation}\label{repr_nem_op}
f(x,\mu) = F\Bigg(x, \int_0^{\infty} K_F(x,y) d\mu(y) \Bigg).
\end{equation}
We will always use lowercase and uppercase letters to distinguish between these two representations.

Similarly as for linear perturbation case (see Theorem \ref{lin_per}), we define perturbed function $f^h(x,\mu)$ as
\begin{equation}\label{pert_fun_def}
f^h(x,\mu) = f^0(x,\mu) + hf_p(x,\mu) = F^0\Bigg(x, \int_0^{\infty} K_{F^0}(x,y) d\mu(y) \Bigg) + h F_P\Bigg(x, \int_0^{\infty} K_{F_P}(x,y) d\mu(y) \Bigg) .
\end{equation}
Note that, in general, function $f^h(x,\mu)$ does not have a kernel representation.

Now, we introduce the setting for analysis of nonlinear equations, similar to the one for linear case (Theorem \ref{lin_per}). We consider equation:
\begin{equation}\label{spm_non_per}
\left\{ \begin{array}{lll}
\partial_t \mu_t^h + \partial_x(b^h(x, \mu_t) \mu_t) & = c^h(x, \mu_t)\mu_t & \mathbb{R}^{+} \times [0,T],\\
b^h(0, \mu_t) D_{\lambda}\mu_t^h(0) &= \int_{\mathbb{R}^+} a^h(x,\mu_t) d\mu_t^h(x) & [0,T], \\
\mu_0 &= \nu & \mathbb{R}^+.
\end{array} \right.
\end{equation}
where function $a^h$, $b^h$ and $c^h$ are defined as in \eqref{pert_fun_def}. We assume:
\begin{description}
\item[(N1)] $A^0, A_P, B^0, B_P, C^0, C_P \in C^{1+\alpha}(\mathbb{R}^+ \times \mathbb{R}^+)$,
\item[(N2)] $K_{A^0}, K_{A_P}, K_{B^0}, K_{B_P}, K_{C^0}, K_{C_P} \in C^{2+\alpha}(\mathbb{R}^+ \times \mathbb{R}^+)$,
\item[(N3)] $b^h(0, \mu) > 0$ for any $\mu \in \mathcal{M}^+(\mathbb{R}^+)$,
\item[(N4)] $a^h(x,\mu) \geq 0$ for any $\mu \in \mathcal{M}^+(\mathbb{R}^+)$.
\end{description}
\begin{defn}\label{const_C_N}{\bf (Contant $C_N$)}
Similarly to the constants $C_L$, $H_L$ and $G_L$ (cf. Definition \ref{constants_meaning}), we introduce $C_N$ as any constant depending continuously on $C^{1+\alpha}$ norms of model functions or $C^{2+\alpha}$ norms of kernels.
\end{defn}
Assumptions above guarantee existence and uniqueness of the measure solution $\mu_t^h$ to the system \eqref{spm_non_per} due to Theorem \ref{well_posedness_2_nonlinear}:
\begin{proof}
We have to check that assumptions above imply (W1a), (W1b), (W2) and (W3) in Theorem \ref{well_posedness_2_nonlinear}. Clearly, (W1a) and (W3) are satisfied. To verify (W2) we compute for some $f(x,\cdot)$:
\begin{multline*}
\norm{f(x,\mu) - f(x,\nu)}_{\infty} = \norm{F\Bigg(x, \int_0^{\infty} K_F(x,y) d\mu(y) \Bigg) - F\Bigg(x, \int_0^{\infty} K_F(x,y) d\nu(y) \Bigg)}_{\infty}\leq \\ \leq
\norm{F_y}_{\infty} \norm{K_F}_{W^{1,\infty}} p_F(\mu,\nu),
\end{multline*}
as desired. Moreover, since
\begin{equation}\label{analysis_f0}
{\partial_x} f(x,\mu) = F_x\Bigg(x, \int_0^{\infty} K_F(x,y) d\mu(y) \Bigg)
+ F_y\Bigg(x, \int_0^{\infty} K_F(x,y) d\mu(y) \Bigg)\int_0^{\infty} K_{F,x}(x,y) d\mu(y),
\end{equation}
(W1b) is obviously satisfied. The proof is concluded.
\end{proof}
We are ready to formulate the main Theorem:
\begin{thm}\label{nonlin_per}
Suppose assumptions (N1) -- (N4) hold true and $\alpha > \frac{1}{2}$. Let $\mu_t^h$ be the solution to \eqref{spm_non_per}. Then, map $[-\frac{1}{2}, \frac{1}{2}] \ni h \to \mu_t^h$ is Fr\'echet differentiable in space $Z$. 
\end{thm}

The general idea to prove Theorem \ref{nonlin_per} is to exploit iteration scheme \eqref{spm_non_approx} to approximate $\mu^h_t$. Namely, we fix $k\in \mathbb{N}$ and divide interval $[0,T]$ for $2^k$ subintervals of equal length. Then, in interval $[m\frac{T}{2^k}, (m+1)\frac{T}{2^k}]$, approximation $\mu^{h,k}_t$ is defined inductively as the solution to: 
\begin{equation}\label{spm_non_approx}
\left\{ \begin{array}{lll}
\partial_t \mu_t + \partial_x(b^h(x, \mu_{m\frac{T}{2^k}}) \mu_t) & = c^h(x, \mu_{m\frac{T}{2^k}})\mu_t & \mathbb{R}^{+} \times  [m\frac{T}{2^k}, (m+1)\frac{T}{2^k}],\\
b^h(0, \mu_{m\frac{T}{2^k}}) D_{\lambda}\mu_t(0) &= \int_{\mathbb{R}^+} a^h(x,\mu_{m\frac{T}{2^k}}) d\mu_t(x) &  [m\frac{T}{2^k}, (m+1)\frac{T}{2^k}], \\
\mu_{m\frac{T}{2^k}} =  \mu^{h,k}_{m\frac{T}{2^k}}
\end{array} \right.
\end{equation}
(the initial value is obtained from solving similar problem on the interval $[(m-1)\frac{T}{2^k}, m\frac{T}{2^k}]$. Note that since iteration sequence converges in flat metric space $C([0,T], (\mathcal{M}^+(\mathbb{R}^+), p_F))$ to the solution of nonlinear problem (cf. Section \ref{lspm_sect} and \cite{GwMa2010}), it also converges in $C([0,T], Z)$ to the same limit.

This iteration scheme will be used in combination with the following classical lemma stating that sequence of continuous functions converging uniformly has a continuous limit \cite{rudin1976principles}.
\begin{lem}\label{unif_conv_cont_limit}
Let $f_n \to f$ uniformly on a set $E$ in some metric space $(X,d)$. Let $x$ be a limit point of $E$ and suppose that
$$
\lim_{t \to x} f_n(t) = A_n.
$$
Then, $A_n$ converges and $\lim_{t\to x} f(t) = \lim_{n\to\infty} A_n$. In particular,
$$
\lim_{t\to x} \lim_{n\to \infty} f_n(t) = \lim_{n \to \infty} \lim_{t \to x} f_n(t).
$$
\end{lem}
Note that $x$ does not have to belong to the set $E$. As we are interested in existence of the limit $\frac{\mu_t^{h+\Delta h} - \mu_t^h}{\Delta h}$, we should take $E = [-\frac{1}{2}, \frac{1}{2}] \setminus \{0\}$ and prove that:
\begin{itemize}
%unif_conv_cont_limit
\item{the sequence $\frac{\mu_t^{h+\Delta h,k} - \mu_t^{h,k}}{\Delta h}$ converges uniformly for all $\Delta h \in E$ as $k \to \infty$ - this is proven in Theorem \ref{crucial_theorem_1},}
\item{$\frac{\mu_t^{h+\Delta h,k} - \mu_t^{h,k}}{\Delta h}$ converges as $\Delta h \to 0$ (this fact can be interpreted as a differentiability of the approximating sequence) - this is proven in Theorem \ref{diff_k_fixed}.}
\end{itemize}
To prove the first result, it is sufficient to demonstrate that
\begin{equation}\label{target_first_point}
\Delta^{k,t}:= \sup_{\Delta h \in (-\frac{1}{2}, \frac{1}{2})}\norm{\frac{\mu_t^{h+\Delta h,k+1} - \mu_t^{h,k+1}}{\Delta h} - 
\frac{\mu_t^{h+\Delta h,k} - \mu_t^{h,k}}{\Delta h}}_Z  \leq C 2^{-k\beta}
\end{equation}
for some independent constants $C >0$ and $\beta >0$. This estimate is difficult to obtain as it simultaneously captures two effects: when $k \to \infty$ and when $\Delta h \to 0$. In particular, one cannot apply triangle inequality to prove \eqref{target_first_point} -- then we would lose one of these effects.

The main idea to obtain the second result is to propagate differentiability along intervals $[0, \frac{T}{2^k}], [\frac{T}{2^k}, 2\frac{T}{2^k}]$ and so on. More precisely, on the first interval, assertion follows directly from linear theory (Theorem \ref{lin_per_1} or even Theorem \ref{lin_per}). Then, when considering second interval, we should use the fact that nonlinearities are evaluated at measure $\mu_{\frac{T}{2^k}}^{h,k}$ where time $\frac{T}{2^k}$ belongs to the first interval. This gives us a lot of regularity in map $h \mapsto f(x, \mu_{\frac{T}{2^k}}^{h,k})$. We study this effect more carefully in the following Section. However, there is a price to be paid for this type of argument as constants that will accumulate during this inductive procedure have to be properly controlled. This is the reason why we estimated H\"older constants quite carefully in the treatment of linear problem. 

\subsection{Regularity of map $h \mapsto f^h(x,\mu^h)$ when $h \mapsto \mu^h$ is differentiable in $Z$}\label{regularity_estimates_41}
The main target of this subsection is to prove that if the map $h \mapsto \mu^h$ is differentiable in $Z$, then $(x,h) \mapsto f(x,\mu^h)$ satisfies assumption (B2) of Theorem \ref{lin_per_1}, namely $(x,h) \mapsto f^h(x,\mu^h)$ is $C^{1+\alpha}(\mathbb{R}^+ \times [-\frac{1}{2}, \frac{1}{2}])$. This result is motivated by the propagation of differentiability mentioned above.

It will be sufficient to consider $(x,h) \mapsto f(x,\mu^h)$, where $f$ is of the form \eqref{repr_nem_op}, since composition with linear map with arguments on bounded domain will not change the result. We assume:
\begin{description}
\item[(F1)] $F \in C^{1+\alpha}(\mathbb{R}^+ \times \mathbb{R}^+)$ and $K_F \in C^{2+\alpha}(\mathbb{R}^+ \times \mathbb{R}^+)$ with norm $C_F$,
\item[(F2)] Map $ [-\frac{1}{2}, \frac{1}{2}] \ni h \mapsto \mu^h \in \mathcal{M}^+(\mathbb{R}^+)$ is Lipschitz in flat metric with constant $C_L$ (in particular its Fr\'echet derivative in $Z$ is bounded by $C_L$, see for instance Corollary \ref{a_priori_bound_derivative} and its proof),
\item[(F3)] Map $ [-\frac{1}{2}, \frac{1}{2}] \ni h \mapsto \mu^h \in \mathcal{M}^+(\mathbb{R}^+)$ is bounded in total variation norm by a constant $C_T$ and is in $C^{1+\alpha}$ (with Fr\'echet derivative in space $Z$) with some constant $C_H$.
\end{description}
Again, we introduce so many constants to trace different effects as during iteration mentioned above some of them will be easier to estimate than others. For the functions and kernels like above, $x$ and $y$ stands for the first and second derivative respectively. Similarly, we denote partial derivatives with $f_x$, $K_{F, x}$ and so on. We remark here that a map $[-\frac{1}{2}, \frac{1}{2}] \ni h \mapsto \mu^h \in \mathcal{M}^+(\mathbb{R}^+)$ is $C^{1+\alpha}$ if $\norm{\mu^h}_{C^{1+\alpha}([-\frac{1}{2}, \frac{1}{2}])}$ with the norm defined by an obvious generalization of formula \eqref{1plusalpha_norm}. We will prove:
\begin{thm}\label{Nem_op}
Under assumptions (F1) -- (F3), the map $\mathbb{R}^+ \times [-\frac{1}{2}, \frac{1}{2}] \ni (x,h) \to f(x,\mu^h)$ satisfies assumption (B2) of Theorem \ref{lin_per_1}. Moreover, 
$$
\norm{{\partial_x} f(x,\mu^h)}_{\alpha,x}, \norm{{\partial_h} f(x,\mu^h)}_{\alpha,x}, \norm{{\partial_x} f(x,\mu^h)}_{\alpha,h} \leq C(C_L, C_T, C_F),
$$
$$
\norm{{\partial_h} f(x,\mu^h)}_{\alpha,h} \leq C(C_L, C_T, C_F)(1 + C_H).
$$
In particular, if functions $F^0, F_P$ and kernels $K_{F^0}, K_{F_P}$ satisfy assumptions (F1) -- (F3), then map $\mathbb{R}^+ \times [-\frac{1}{2}, \frac{1}{2}] \ni (x,h) \to f^h(x,\mu^h)$ meets (B2) in Theorem \ref{lin_per_1} with:
$$
\norm{{\partial_x} f^h(x,\mu^h)}_{\alpha,x}, \norm{{\partial_h} f^h(x,\mu^h)}_{\alpha,x}, \norm{{\partial_x} f(x,\mu^h)}_{\alpha,h} \leq C(C_L, C_T, C_F),
$$
$$
\norm{{\partial_h} f^h(x,\mu^h)}_{\alpha,h} \leq C(C_L, C_T, C_F)(1 + C_H).
$$
\end{thm}
The proof will be divided for three Lemmas:
\begin{lem}\label{Nem_op_1}
Suppose (F1) holds. For each $\mu \in \mathcal{M}^+(\mathbb{R}^+)$ with $\norm{\mu}_{TV} \leq C_T$, map $x \mapsto f(x,\mu)$ is $C^{1+\alpha}$ with norm depending on $C(C_L, C_T, C_F)$. More preciesly, $\norm{f(x,\mu)}_{W^{1,\infty}} \leq C(C_T, C_F)$ and $\norm{{\partial_x}f(x, \mu)}_{\alpha,x} \leq C(C_L, C_T, C_F)$. Moreover, for any $h \in [-\frac{1}{2}, \frac{1}{2}]$, $\norm{{\partial_x}f(x, \mu^h)}_{\alpha,h} \leq C(C_L, C_T, C_F)$. 
\end{lem}
\begin{proof}
It is obvious that $\norm{ f(x,\mu)}_{\infty} \leq \norm{F}_{\infty}$ and it follows from \eqref{analysis_f0} that $\norm{f_x(x,\mu)}_{\infty} \leq C_F + C_F^2 C_T \leq C(C_F, C_T)$. Then, using \eqref{analysis_f0} again, we write:
\begin{multline*}
{\partial_x}  f(x,\mu) \Big|_{x=z_1} - {\partial_x}  f(x,\mu)\Big|_{x=z_2} = 
 \underbrace{F_x\bigg(z_1, \int_0^{\infty} K_{F}(z_1,y) d\mu(y) \bigg)
 -F_x\bigg(z_2, \int_0^{\infty} K_{F}(z_2,y) d\mu(y) \bigg)}_
 {\leq~ \norm{F_x}_{\alpha,x}~\abs{z_1-z_2}^{\alpha}~+~ 
 \norm{F_x}_{\alpha,y}~\norm{K_{F,x}}_{\infty}^{\alpha}~\norm{\mu}_{TV}^{\alpha}~\abs{z_1-z_2}^{\alpha}}+\\
 + \underbrace{ F_y\bigg(z_1, \int_0^{\infty} K_{F}(z_1,y) d\mu(y) \bigg) \int_0^{\infty}  K_{F,x}(z_1,y) d\mu(y)
 -  F_y\bigg(z_2, \int_0^{\infty} K_{F}(z_2,y) d\mu(y) \bigg) \int_0^{\infty} K_{F,x}(z_2,y) d\mu(y)}_{\leq 
 \Big(\norm{F_y}_{\alpha,x}~\abs{z_1 - z_2}^{\alpha}~+~ \norm{F_y}_{\alpha,y}~\norm{K_{F,x}}_{\infty}^{\alpha}~\norm{\mu}_{TV}^{\alpha}~\abs{z_1-z_2}^{\alpha}\Big)~\norm{K_{F,x}}_{\infty}~\norm{\mu}_{TV}~+~ 
 \norm{F_y}_{\infty}~\norm{K_{F,x}}_{\alpha}~\abs{z_1 - z_2}^{\alpha} \norm{\mu}_{TV}
 }\leq \\
 \leq C(C_F, C_T) \abs{z_1 - z_2}^{\alpha}.
\end{multline*}
Similarly,
\begin{multline*}
\abs{{\partial_x}  f(x,\mu^{h_1})  - {\partial_x}  f(x,\mu^{h_2})} = 
 \underbrace{F_x\bigg(x, \int_0^{\infty} K_{F}(x,y) d\mu^{h_1}(y) \bigg)
 -F_x\bigg(x, \int_0^{\infty} K_{F}(x,y) d\mu^{h_2}(y) \bigg)}_
 {\leq~ \norm{F_x}_{\alpha,y}~\Big(\norm{K_{F}}_{W^{1,\infty}}~p_F(\mu^{h_1}, \mu^{h_2})\Big)^{\alpha} \leq C(C_F, C_L) \abs{h_1 - h_2}^{\alpha}}+\\
 + \underbrace{ F_y\bigg(x, \int_0^{\infty} K_{F}(x,y) d\mu^{h_1}(y) \bigg) \int_0^{\infty}  K_{F,x}(x,y) d\mu^{h_1}(y)
 -  F_y\bigg(x, \int_0^{\infty} K_{F}(x,y) d\mu^{h_2}(y) \bigg) \int_0^{\infty} K_{F,x}(x,y) d\mu^{h_2}(y)}_{\leq ~ \norm{F_y}_{\alpha,y}~\Big(\norm{K_{F}}_{W^{1,\infty}}~p_F(\mu^{h_1}, \mu^{h_2})\Big)^{\alpha}~\norm{K_{F,x}}_{\infty}~\norm{\mu^{h_1}}_{TV}~+~ 
 \norm{F_y}_{\infty}~\norm{K_{F,x}}_{W^{1,\infty}}~p_F(\mu^{h_1}, \mu^{h_2}) ~\leq~C(C_F, C_T, C_L) \abs{h_1 - h_2}^{\alpha}
 }\leq \\
 \leq C(C_F, C_T, C_L)  \abs{h_1 - h_2}^{\alpha}.
 \end{multline*}
\end{proof}
\begin{lem}\label{Nem_op_2}
Suppose (F1)--(F3) hold true. Then, the map $[-\frac{1}{2}, \frac{1}{2}] \ni h \mapsto f(x,\mu^h)$ is $C^{1+\alpha}$. More precisely, $\norm[2]{f(x,\mu^h)}_{\infty} \leq C_F$, $\norm[2]{{\partial_h} f(x,\mu^h)}_{\infty} \leq C(C_F, C_L)$ and $\norm[2]{{\partial_h} f(x,\mu^h)}_{\alpha,h} \leq C(C_F,C_L,C_T)(1+C_H)$.
\end{lem}
\begin{proof}
This is very similar to the previous Lemma. Obviously, $\norm{f(x,\mu^h)}_{\infty} \leq C_F$. Now, we compute derivative of $h \mapsto f(x,\mu^h)$:
\begin{multline*}
\lim_{\Delta h \to 0} \frac{ f(x,\mu^{h+\Delta h}) - f(x,\mu^{h})}{\Delta h} =
\lim_{\Delta h \to 0} \frac{F\Big(x, \int_0^{\infty} K_{F}(x,y) d\mu^{h+\Delta h}(y) \Big) - 
F\Big(x, \int_0^{\infty} K_{F}(x,y) d\mu^{h}(y) \Big)}{\Delta h}= \\
= \lim_{\Delta h \to 0} \underbrace{\frac{F\Big(x, \int_0^{\infty} K_{F}(x,y) d\mu^{h+\Delta h}(y)\Big) - 
F\Big(x, \int_0^{\infty} K_{F}(x,y) d\mu^{h}(y) \Big)}{\int_0^{\infty} K_{F}(x,y) d(\mu^{h+\Delta h}- \mu^h)(y)}}_{:= ~A(\Delta h)} \underbrace{\frac{\int_0^{\infty} K_{F}(x,y) d(\mu^{h+\Delta h}- \mu^h)(y)}{\Delta h}}_{:=~B(\Delta h)}.
\end{multline*}
Since $ K_{F}(x,y) \in C^{1+\alpha}$, by assumption on differentiability:
 $$
 \abs[2]{\int_0^{\infty} K_{F}(x,y) d(\mu^{h+\Delta h}- \mu^h)(y)} \leq \norm[2]{\mu^{h+\Delta h} - \mu^h }_{Z} \to 0 \text{  as  } \Delta h \to 0.
 $$
 Therefore,
$$
\lim_{\Delta h \to 0} A(\Delta h) = F_y \bigg(x, \int_0^{\infty} K_{F}(x,y) d\mu^{h}(y) \bigg).
$$
Now, note that,
\begin{multline*}
\abs{\frac{\int_0^{\infty} K_{F}(x,y) d(\mu^{h+\Delta h}- \mu^h)(y)}{\Delta h} - \bigg({\partial_H}\mu^H \Big|_{H=h}, K_{F}(x,y)\bigg)_{(Z, C^{1+\alpha})} }\leq \\ \leq \norm{K_{F}(x,y)}_{C^{1+\alpha}} 
\norm{\frac{\mu^{h+\Delta h}- \mu^h}{\Delta h} - {\partial_H}\mu^H |_{H=h}}_Z \to 0 \text{ as } \Delta h \to 0,
\end{multline*}
where $(\cdot, \cdot)_{(Z, C^{1+\alpha})} $ denotes the dual pairing (we had to be careful here as ${\partial_H}\mu^H \Big|_{H=h}$ does not have to be a measure). We conclude:
\begin{equation}\label{derf/h}
{\partial_h} f(x,\mu^h) = F_y \bigg(x, \int_0^{\infty} K_{F}(x,y) d\mu^{h}(y) \bigg) \cdot  \bigg({\partial_H}\mu^H \Big|_{H=h}, K_{F}(x,y)\bigg)_{(Z, C^{1+\alpha})}
\end{equation}
so that $\norm{{\partial_h} f(x,\mu^h)}_{\infty} \leq C(C_F, C_L)$ (we remark here that $\norm{{\partial_H}\mu^H}_{Z} \leq C_L$ -- see assumption (F2) or the proof of Corollary \ref{a_priori_bound_derivative}). Finally, we check H\"older continuity:
\begin{multline*}
{\partial_h} f(x,\mu^h) \big|_{h=h_1} - {\partial_h} f(x,\mu^h) \big|_{h=h_2}=\\ = 
 \underbrace{
 \Bigg[F_y \bigg(x, \int_0^{\infty} K_{F}(x,y) d\mu^{h_1}(y) \bigg) -
  F_y \bigg(x, \int_0^{\infty} K_{F}(x,y) d\mu^{h_2}(y) \bigg)\Bigg]
  \cdot  \bigg({\partial_H}\mu^H \Big|_{H=h_1}, K_{F}(x,y)\bigg)_{(Z, C^{1+\alpha})}}_
  {\leq~\norm{F_y}_{\alpha,y}~\Big(\norm{K_{F}}_{W^{1,\infty}}~
  p_F(\mu^{h_1}, \mu^{h_2})\Big)^{\alpha} C_L \norm{K_{F}}_{C^{1+\alpha}}}+
   \\ +
  \underbrace{F_y \bigg(x, \int_0^{\infty} K_{F}(x,y) d\mu^{h_2}(y) \bigg) \cdot  \bigg[\bigg({\partial_H}\mu^H \Big|_{H=h_1}, K_{F}(x,y)\bigg)_{(Z, C^{1+\alpha})} -
  \Bigg({\partial_H}\mu^H \Big|_{H=h_2}, K_{F}(x,y)\Bigg)_{(Z, C^{1+\alpha})} \bigg]}_
  {\leq~\norm{F_y}_{\infty}~C_H~\abs{h_1 - h_2}^{\alpha}~\norm{K_{F}}_{C^{1+\alpha}}} \leq \\
  \leq C(C_L, C_F, C_T)(1+ C_H) \abs{h_1 - h_2}^{\alpha}.
\end{multline*}
\end{proof}
\begin{lem}
Suppose (F1)--(F3) hold true. Then, the map $\mathbb{R}^+ \ni x \mapsto {\partial_h} f(x,\mu^h)$ is H\"older continuous with constant $C(C_F,C_L,C_T)$.
\end{lem}
\begin{proof}
To perform estimates like above we need a bound of the form $\norm{y \mapsto K_F(x_1,y) - K_F(x_2,y)}_{C^{1+\alpha}} \leq C \abs{x_1-x_2}^{\alpha}$ so that we can conclude
\begin{equation}\label{crucial_estimate_forholder}
 \bigg({\partial_H}\mu^H \Big|_{H=h}, K_{F}(x_1,y) - K_{F}(x_2,y)\bigg)_{(Z, C^{1+\alpha})} \leq \norm[2]{{\partial_H}\mu^H}_Z \norm[2]{K_{F}(x_1,y) - K_{F}(x_2,y)}_{C^{1+\alpha}} \leq C \abs{x_1 - x_2}^{\alpha}
\end{equation}
for some constant $C$ to be determined. Using only $C^{1+\alpha}$ regularity we can easily deduce:
$$
\norm{y \mapsto K_{F}(x_1,y) - K_{F}(x_2,y)}_{\infty} \leq C_F \abs{x_1 - x_2}, \quad \quad \quad
\norm{y \mapsto {\partial_y}\big(K_{F}(x_1,y) - K_{F}(x_2,y)\big)}_{\infty} \leq C_F \abs{x_1 - x_2}^{\alpha}.
$$
To obtain desired $C^{1+\alpha}$ regularity of the map $y \mapsto K_{F}(x_1,y) - K_{F}(x_2,y)$, we will prove
\begin{equation}\label{Holest_2}
\sup_{x_1, x_2, x_1 \neq x_2 } \sup_{ y_1, y_2, y_2 \neq y_2} \abs{\frac{\big(K_{F,y}(x_1,y_1) - K_{F,y}(x_2,y_1)\big) - \big(K_{F,y}(x_1,y_2) - K_{F,y}(x_2,y_2)\big)}{\abs{x_1 - x_2}^{\alpha}\abs{y_1-y_2}^{\alpha}}} \leq C_F.
\end{equation}
Indeed, if $\abs{x_1 - x_2} \geq 1$ and $\abs{y_1-y_2} \geq 1$, then \eqref{Holest_2} is trivial (we use directly $L^{\infty}$ norms). Moreover, if $\abs{x_1 - x_2} \geq 1, \abs{y_1-y_2} \leq 1$ we use Taylor's expansion together with $C^{2+\alpha}$ regularity of kernel $K_F$:
$$
\abs{\big(K_{F,y}(x_1,y_1) - K_{F,y}(x_2,y_1)\big) - \big(K_{F,y}(x_1,y_2) - K_{F,y}(x_2,y_2)\big)} \leq C_F \abs{y_1-y_2}\abs{x_1-x_2}^{\alpha}+ C_F\abs{y_1-y_2}^2,
$$
so that \eqref{Holest_2} follows. Similarly, if $\abs{x_1 - x_2} \leq 1$ and $\abs{y_1-y_2} \geq 1$, we use analogous expansion:
$$
\abs{\big(K_{F,y}(x_1,y_1) - K_{F,y}(x_2,y_1)\big) - \big(K_{F,y}(x_1,y_2) - K_{F,y}(x_2,y_2)\big)} \leq C_M \abs{x_1-x_2}\abs{y_1-y_2}^{\alpha}+ C_M\abs{x_1-x_2}^2.
$$
Finally, if $\abs{x_1 - x_2} \leq 1$ and $\abs{y_1-y_2} \leq 1$ we use both bounds together:
\begin{multline*}
\abs{\frac{\big(K_{F,y}(x_1,y_1) - K_{F,y}(x_2,y_1)\big) - \big(K_{F,y}(x_1,y_2) - K_{F,y}(x_2,y_2)\big)}{\abs{x_1 - x_2}^{\alpha}\abs{y_1-y_2}^{\alpha}}} \leq \\
 \leq C_M\frac{\text{min}\Big(\abs{x_1-x_2}\abs{y_1-y_2}^{\alpha}+ \abs{x_1-x_2}^2,
 \abs{y_1-y_2}\abs{x_1-x_2}^{\alpha}+ \abs{y_1-y_2}^2\Big)}
 {\abs{x_1 - x_2}^{\alpha}\abs{y_1-y_2}^{\alpha}} = \\ = C_M
 \text{min}\Big(\abs{x_1-x_2}^{1-\alpha}+ \abs{x_1-x_2}^{2-\alpha}\abs{y_1-y_2}^{-\alpha},
 \abs{y_1-y_2}^{1-\alpha}+ \abs{y_1-y_2}^{2-\alpha}\abs{x_1-x_2}^{-\alpha}\Big) \leq \\ \leq
 C_M + C_M~\text{min}\Big(\abs{x_1-x_2}^{2-\alpha}\abs{y_1-y_2}^{-\alpha},
 \abs{y_1-y_2}^{2-\alpha}\abs{x_1-x_2}^{-\alpha}\Big)
\end{multline*}
since $\abs{x_1-x_2} \leq 1$ and $\abs{y_1-y_2} \leq 1$. Finally, observe that the minimum above has to be bounded since product of the two terms is bounded:
$$
\big(\abs{x_1-x_2}^{2-\alpha}\abs{y_1-y_2}^{-\alpha}\big) \cdot \big( \abs{y_1-y_2}^{2-\alpha}\abs{x_1-x_2}^{-\alpha} \big) =  \abs{x_1-x_2}^{2-2\alpha}\abs{y_1-y_2}^{2-2\alpha} \leq 1
$$
as $\alpha \in (0,1)$. Therefore, \eqref{crucial_estimate_forholder} follows for $C = C(C_F, C_L)$. Finally, using \eqref{derf/h}, we conclude:
\begin{multline*}
\abs{ {\partial_h} f(x_1,\mu^h) - {\partial_h} f(x_2,\mu^h)} \leq  \\ 
\leq  \abs{\Bigg(F_y \bigg(x_1, \int_0^{\infty} K_{F}(x_1,y) d\mu^{h}(y) \bigg) -  F_y \bigg(x_2, \int_0^{\infty} K_{F}(x_2,y) d\mu^{h}(y) \bigg)\Bigg)\cdot  \bigg({\partial_H}\mu^H \Big|_{H=h}, K_{F}(x_1,y)\bigg)_{(Z, C^{1+\alpha})}} + \\ 
+ F_y \bigg(x_2, \int_0^{\infty} K_{F}(x_2,y) d\mu^{h}(y) \bigg) \cdot  \bigg({\partial_H}\mu^H \Big|_{H=h}, K_{F}(x_1,y) - K_{F}(x_2,y)\bigg)_{(Z, C^{1+\alpha})} \leq C(C_T, C_F, C_L)\abs{x_1-x_2}^{\alpha}.
\end{multline*}
\end{proof}
We conclude this section with a technical estimate that will be useful later:
\begin{lem}\label{bound_for_measure_der_h}
Let $\{\mu^h\}_{h\in [-\frac{1}{2}, \frac{1}{2}]}$ and $\{\nu^h\}_{h\in [-\frac{1}{2}, \frac{1}{2}]}$ be two families of measures satisfying (F2) and (F3). Then
$$
\norm[2]{{\partial_h}f(x,\mu^h) - {\partial_h}f(x,\nu^h)}_{\infty} \leq C(C_F, C_L) \big(p_F(\mu^h, \nu^h)\big)^{\alpha} +  C_F \norm[2]{{\partial_H}\mu^H \Big|_{H=h} - {\partial_H}\nu^H \Big|_{H=h}}_{Z}.
$$
\end{lem}
\begin{proof}
From \eqref{derf/h} we easily compute:
\begin{multline*}
\norm{{\partial_h}f(x,\mu^h) - {\partial_h}f(x,\nu^h)}_{\infty} \leq \\ \leq
\abs[4]{\underbrace{\Bigg(F_y \bigg(x, \int_0^{\infty} K_{F}(x,y) d\mu^{h}(y) \bigg) - F_y \bigg(x, \int_0^{\infty} K_{F}(x,y) d\mu^{h}(y) \bigg)\Bigg)}_{\leq~
\norm{F_y}_{\alpha,y} ~\big(\norm{K_{F}}_{W^{1,\infty}} p_F(\mu^{h}, \nu^{h})\big)^{\alpha}
}
\underbrace{\bigg({\partial_H}\mu^H \Big|_{H=h}, K_{F}(x,y)\bigg)_{(Z, C^{1+\alpha})}}_{\leq~C_{L}~\norm{K_{F}}_{C^{1+\alpha}}}} +  \\ +
\abs[4]{\underbrace{F_y \bigg(x, \int_0^{\infty} K_{F}(x,y) d\mu^{h}(y) \bigg) }_{\leq~\norm{F_y}_{\infty}}
\bigg({\partial_H}\mu^H \Big|_{H=h} - {\partial_H}\nu^H \Big|_{H=h}, K_{F}(x,y)\bigg)_{(Z, C^{1+\alpha})} }.
\end{multline*}
\end{proof}
\subsection{Stability estimates for sequence $\mu_t^{h,k}$}
In this subsection, we will establish some bounds on the sequence $\mu_t^{h,k}$ that will allow to perform the iterations. We begin with the simple, yet powerful lemma that was already established in \cite{GwMa2010} in a simplified version. It is proven by simple induction.
\begin{lem}\label{discrete_Gronwall}
Suppose a sequence $u_k$ satisfies $\abs{u_k} \leq a \max{\big(c,\abs{u_{k-1}}\big)} + b$ for some nonnegative constants $a \geq 1$, $b$ and $c$. Then,
$$
\abs{u_k} \leq a^k \max{\big(\abs{u_0},c)} + 
\begin{cases} 
\frac{a^k - 1}{a - 1} b &\mbox{if } a > 1,  \\ 
k b & \mbox{if } a = 1.
\end{cases}
$$
\end{lem}
\begin{proof}
Let $a > 1$. Clearly, lemma holds for $k=1$. Suppose it holds for $k = m$. Then,
\begin{multline*}
\abs{u_{m+1}} \leq a \max{\big(c,\abs{u_{m}}\big)} + b \leq 
 a \max{\Big(c,a^m \max{\big(\abs{u_0},c)} + \frac{a^m - 1}{a - 1} b\Big)} + b \leq \\ \leq
a \max{\Big(c,a^m \max{\big(\abs{u_0},c)}\Big)} + a\frac{a^m - 1}{a - 1} b + b \leq
a^{m+1} \max{(\abs{u_0},c))} + a\frac{a^m - 1}{a - 1} b + b =\\=
a^{m+1} \max{(\abs{u_0},c))} +b \frac{a^{m+1} - 1}{a - 1}
\end{multline*}
and similarly (or even slightly easier) when $a=1$. The proof is concluded.
\end{proof}
Our first observation is that the sequence $\mu_t^{h,k}$ is uniformly bounded in total variation norm, independently of $k \in \mathbb{N}$, $h \in [-\frac{1}{2}, \frac{1}{2}]$ and $t \in [0,T]$ so it will not blow up while iterating:
\begin{lem}\label{total_stability}
There is a constant $C_N$ such that
$$
\norm[2]{\mu_t^{h,k}}_{TV} \leq \norm[2]{\mu_0}_{TV}e^{C_N T}.
$$
\end{lem} 
\begin{proof}
Fix $k \in \mathbb{N}$ and observe that by \eqref{stability}, for $ t \in [m \frac{T}{2^k}, (m+1) \frac{T}{2^k}]$ we have:
\begin{equation*}
\norm[2]{\mu^{h,k}_{t}}_{TV} \leq \norm[2]{\mu_{m \frac{T}{2^k}}^{h,k}}_{TV} e^{2\big(\norm{a^h}_{\infty} + \norm{c^h}_{\infty}\big) \frac{T}{2^k}} \leq 
\norm[2]{\mu_{m \frac{T}{2^k}}^{h,k}}_{TV} e^{C_N \frac{T}{2^k}}.
\end{equation*}
Using Lemma \ref{discrete_Gronwall} with $c = 0$ and $m \leq 2^k$ we conclude the proof.
\end{proof}
We move on to study the variation of iteration sequence $\mu_t^{h,k}$ as $h$ and $k$ are changed. 
\begin{lem}\label{iteration_in_h}
For some constant $C_N$, independent of $k$, we have:
$$
p_F(\mu_t^{h+\Delta h, k}, \mu_t^{h,k}) \leq C_N \abs{\Delta h}.
$$
\end{lem}
\begin{proof}
Let $t \in [m \frac{T}{2^k}, (m+1) \frac{T}{2^k}]$ and $t^* = m \frac{T}{2^k}$. Let $\bar{\mu}_t^{h+\Delta h, k}$ be the solution to \eqref{spm_non_approx} with perturbed initial condition $\mu_{t^*}^{h+\Delta h, k}$ but not perturbed nonlinearities $a^h(x,\mu_{t^*}^{h, k}), b^h(x,\mu_{t^*}^{h, k}), c^h(x,\mu_{t^*}^{h, k})$ (one can think of $\bar{\mu}_t^{h+\Delta h, k} $ as some measure between ${\mu}_t^{h+\Delta h, k}$ and $\mu_t^{h, k})$. Then, using \eqref{cont_model} and \eqref{cont_initial}, we obtain:
\begin{multline*}
p_F(\mu_t^{h+\Delta h, k}, \mu_t^{h,k}) \leq p_F(\mu_t^{h+\Delta h, k}, \bar{\mu}_t^{h+\Delta h, k}) +
p_F( \bar{\mu}_t^{h+\Delta h, k}, \mu_t^{h,k}) \leq \\
\leq \norm{\mu_{t^*}^{h+\Delta h, k}}_{TV}\frac{T}{2^k} e^{\bar{C} \frac{T}{2^k}} \sum_{f=a,b,c}\norm{f^{h+\Delta h}(x,\mu_{t^*}^{h+\Delta h, k}) - f^h(x,\mu_{t^*}^{h, k})}_{\infty} + 
e^{\bar{C} \frac{T}{2^k}} p_F({\mu}_{t^*}^{h+\Delta h, k}, \mu_{t^*}^{h,k}).
\end{multline*}
where the constant $\bar{C}$ depends on $W^{1,\infty}$ norms of model functions $x \mapsto a^h(x,\mu_{t^*}^{h, k})$, $x \mapsto b^h(x,\mu_{t^*}^{h, k})$ and $x \mapsto c^h(x,\mu_{t^*}^{h, k})$. In view of Lemma \ref{Nem_op_1}, these norms depend on $C_N$ and total variation norm of $\mu_{t^*}^{h, k}$ which, in view of Lemma \ref{total_stability}, is also uniformly bounded by $C_N$. This implies $\bar{C} = C_N$. Then, for $f=a, b, c$ we compute:
\begin{multline*}
\norm{f^{h+\Delta h}(x,\mu_{t^*}^{h+\Delta h, k}) - f^h(x,\mu_{t^*}^{h, k})}_{\infty} = 
\norm{\Delta h f_p(x,\mu_{t^*}^{h+\Delta h, k}) + f^h(x,\mu_{t^*}^{h+\Delta h, k}) - f^h(x,\mu_{t^*}^{h, k})}_{\infty} \leq \\
\leq C_N\abs{\Delta h} + C_N \sup_x \int_{\mathbb{R}^+} \big(K_{F^0}(x,y) + K_{F_P}(x,y)\big) d(\mu_{t^*}^{h+\Delta h, k} -  \mu_{t^*}^{h, k}) \leq C_N\abs{\Delta h} + C_N p_F( {\mu}_{t^*}^{h+\Delta h, k}, \mu_{t^*}^{h,k}).
\end{multline*}
so that coming back to the starting point:
\begin{multline*}
p_F(\mu_t^{h+\Delta h, k}, \mu_t^{h,k})  \leq C_N \frac{T}{2^k} e^{C_N \frac{T}{2^k}} \abs{\Delta h} + C_N \frac{T}{2^k} e^{C_N \frac{T}{2^k}} p_F( {\mu}_{t^*}^{h+\Delta h, k}, \mu_{t^*}^{h,k}) + 
e^{C_N \frac{T}{2^k}} p_F({\mu}_{t^*}^{h+\Delta h, k}, \mu_{t^*}^{h,k}) \leq \\
\leq C_N \frac{T}{2^k} e^{C_N \frac{T}{2^k}} \abs{\Delta h} + e^{C_N \frac{T}{2^k}} p_F({\mu}_{t^*}^{h+\Delta h, k}, \mu_{t^*}^{h,k}),
\end{multline*}
using inequality $1+x \leq e^x$. Now, for $m = 0,1,...,2^k$, let $a_m = p_F(\mu_{m \frac{T}{2^k}}^{h+\Delta h, k}, \mu_{m \frac{T}{2^k}}^{h,k})$ so that $a_0 = 0$. In view of Lemma \ref{discrete_Gronwall}:
\begin{equation}\label{example_for_iterations}
a_m \leq C_N \frac{e^{C_N \frac{T}{2^k} m} - 1}{e^{C_N \frac{T}{2^k}} - 1}\frac{T}{2^k} \abs{\Delta h} \leq 
C_N \frac{e^{C_N T} - 1}{e^{C_N \frac{T}{2^k}} - 1}\frac{T}{2^k} \abs{\Delta h} \leq C_N \abs{\Delta h}
\end{equation}
since the sequence $\frac{{e^{C_N \frac{T}{2^k}} - 1}}{\frac{T}{2^k}}$ is convergent as $k \to \infty$. The proof is concluded.
\end{proof}
\begin{cor}\label{a_priori_bound_derivative}
Suppose we know that the map $h \mapsto \mu_t^{h,k}$ is Fr\'echet differentiable in $Z$. Then $\norm{\frac{\partial}{\partial h} \mu_t^{h,k}}_Z \leq C_N$. Indeed, 
\begin{multline*}
\norm{\frac{\partial}{\partial h} \mu_t^{h,k}}_Z = \lim_{\Delta h \to 0} \frac{1}{\Delta h} \sup_{\xi \in C^{1+\alpha},~\norm{\xi}\leq 1} \int_{\mathbb{R}^+} \xi d(\mu^{h+\Delta h,k}_t - \mu^{h,k}_t) \leq \\ \leq
\lim_{\Delta h \to 0} \frac{1}{\Delta h} \sup_{\xi \in W^{1,\infty},~\norm{\xi}\leq 1} \int_{\mathbb{R}^+} \xi d(\mu^{h+\Delta h,k}_t - \mu^{h,k}_t) \leq \lim_{\Delta h \to 0} \frac{1}{\Delta h} p_F(\mu^{h+\Delta h,k}_t, \mu^{h,k}_t) \leq C_N.
\end{multline*}
This provides a priori estimate on the derivatives that does not change with $k$.
\end{cor}
Similarly, we can study stability of the sequence $\mu_t^{h,k}$ as $h$ is fixed and $k$ is changed:
\begin{lem}\label{iteration_in_k}
For some constant $C_N$, independent of $k$, we have:
$$
p_F(\mu_t^{h, k+1}, \mu_t^{h,k}) \leq C_N {2^{-k}}.
$$
\end{lem}
\begin{proof}
Let $t \in [l\frac{T}{2^k}, (l+1) \frac{T}{2^k}] := A_l$. Denote by $A_l^+$ and $A_l^-$ the right and left part of $A_l$ respectively. Moreover, put $t_* = l\frac{T}{2^k}, t_m = t_* + \frac{T}{2^{k+1}}$ so that $ A_l^-=[t_*, t_m]$. First, if $t\in A_l^-$ we have by \eqref{cont_model} and \eqref{cont_initial}:
\begin{equation*}
p_F(\mu_t^{h, k+1}, \mu_t^{h,k}) \leq \norm{\mu_{t_*}^{h,k+1}}_{TV}\frac{T}{2^{k+1}} e^{C \frac{T}{2^{k+1}}} \sum_{f = a,b,c} \norm{f(x,\mu_{t_*}^{h,k+1}) - f(x,\mu_{t_*}^{h,k})}_{\infty} + e^{C \frac{T}{2^{k+1}}} p_F(\mu_{t_*}^{h,k+1}, \mu_{t_*}^{h,k}),
\end{equation*}
where $C$ depends on $W^{1,\infty}$ norms of maps $x \mapsto f(x,\mu_{t_*}^{h,k+1})$ and  $x \mapsto f(x,\mu_{t_*}^{h,k})$ for $f = a,b,c$. Similarly as above, $C = C_N$ due to Lemma \ref{Nem_op_1} and Lemma \ref{total_stability}. Moreover, by kernel representation of $f$: 
$$\norm{f(x,\mu_{t_*}^{h,k+1}) - f(x,\mu_{t_*}^{h,k})}_{\infty} \leq C_N p_F(\mu_{t_*}^{h,k+1}, \mu_{t_*}^{h,k}).
$$
Therefore, using $1+x \leq e^x$, we conclude:
\begin{equation*}
p_F(\mu_t^{h, k+1}, \mu_t^{h,k}) \leq e^{C_N \frac{T}{2^{k+1}}} p_F(\mu_{t_*}^{h,k+1}, \mu_{t_*}^{h,k}).
\end{equation*}
Now, let $t \in A_l^+$. Again, using \eqref{cont_model} and \eqref{cont_initial}:
\begin{multline*}
p_F(\mu_t^{h, k+1}, \mu_t^{h,k}) \leq \norm{\mu_{t_m}^{h,k+1}}\frac{T}{2^{k+1}} e^{C_N \frac{T}{2^{k+1}}} \sum_{f = a,b,c} \norm{f(x,\mu_{t_m}^{h,k+1}) - f(x,\mu_{t_*}^{h,k})}_{\infty} + e^{C_N \frac{T}{2^{k+1}}} p_F(\mu_{t_m}^{h,k+1}, \mu_{t_m}^{h,k}) \leq \\ 
\leq \frac{C_NT}{2^{k+1}} e^{C_N \frac{T}{2^{k+1}}}p_F(\mu_{t_m}^{h,k+1},\mu_{t_*}^{h,k}) + e^{C_N \frac{T}{2^{k+1}}} p_F(\mu_{t_m}^{h,k+1}, \mu_{t_m}^{h,k}).
\end{multline*}
Using triangle inequality, continuity in time \eqref{cont_time} and the fact that $t_m \in A_l^-$:
\begin{multline*}
p_F(\mu_{t_m}^{h,k+1},\mu_{t_*}^{h,k}) \leq 
p_F(\mu_{t_m}^{h,k+1},\mu_{t_m}^{h,k})+ p_F(\mu_{t_m}^{h,k},\mu_{t_*}^{h,k}) \leq p_F(\mu_{t_m}^{h,k+1},\mu_{t_m}^{h,k}) + C_N \frac{T}{2^{k+1}} \leq \\ 
\leq e^{C_N \frac{T}{2^{k+1}}}p_F(\mu_{t_*}^{h,k+1},\mu_{t_*}^{h,k}) + C_N \frac{T}{2^{k+1}}.
\end{multline*}
Therefore, we conclude:
\begin{equation*}
p_F(\mu_t^{h, k+1}, \mu_t^{h,k}) \leq \Big(\frac{C_N}{2^{k+1}}\Big)^2 + p_F(\mu_{t_*}^{h,k+1},\mu_{t_*}^{h,k})e^{C_N \frac{T}{2^{k+1}}}.
\end{equation*}
Setting iterations over intervals $[0,\frac{T}{2^k}], [\frac{T}{2^k}, 2\frac{T}{2^k}], ...$, we conclude the proof due to Lemma \ref{discrete_Gronwall}.
\end{proof}
\begin{cor}\label{Deltaf_estimate_cor}
Consider the interval of time $[l \frac{T}{2^k}, (l+1) \frac{T}{2^k}]$ with $t_* = l \frac{T}{2^k}$ and $t_m = t_* + \frac{T}{2^{k+1}}$ (the middle of the interval). Then,
\begin{equation}\label{Deltaf_estimate}
\norm{f(x,\mu^{h,k+1}_{t_m}) - f(x,\mu^{h,k}_{t_*})}_{\infty}, 
\norm{f(x,\mu^{h,k+1}_{t_*}) - f(x,\mu^{h,k}_{t_*})}_{\infty}
 \leq C_N {2^{-k}},
\end{equation}
\begin{equation}\label{Deltaf_x_estimate}
\norm{{\partial_x}f(x,\mu^{h,k+1}_{t_m}) - {\partial_x}f(x,\mu^{h,k}_{t_*})}_{\infty}, 
\norm{{\partial_x}f(x,\mu^{h,k+1}_{t_*}) - {\partial_x}f(x,\mu^{h,k}_{t_*})}_{\infty}
 \leq C_N 2^{-k \alpha}.
\end{equation}
\end{cor}
\begin{proof}
Since $\norm{f(x,\mu) - f(x,\nu)}_{\infty} \leq C_N~ p_F(\mu, \nu)$, assertion \eqref{Deltaf_estimate} follows directly from Lemma \ref{iteration_in_k}. Moreover, by \eqref{analysis_f0},
\begin{multline*}
\abs{{\partial_x} f(x,\mu) - {\partial_x} f(x,\nu)} = \underbrace{\abs{F_x\Bigg(x, \int_0^{\infty} K_F(x,y) d\mu(y) \Bigg) - F_x\Bigg(x, \int_0^{\infty} K_F(x,y) d\nu(y) \Bigg)}}_{\leq~\norm{F_x}_{\alpha,y} ~\norm{K_F}_{W^{1,\infty}}^{\alpha} ~\big(p_F(\mu,\nu)\big)^{\alpha}} + \\
+ \underbrace{\abs{F_y\Bigg(x, \int_0^{\infty} K_F(x,y) d\mu(y) \Bigg)\int_0^{\infty} K_{F,x}(x,y) d\mu(y) - F_y\Bigg(x, \int_0^{\infty} K_F(x,y) d\nu(y) \Bigg)\int_0^{\infty} K_{F,x}(x,y) d\nu(y)}}_{\leq~\norm{F_y}_{\alpha,y}~\norm{K_F}_{W^{1,\infty}} ~\big(p_F(\mu,\nu)\big)^{\alpha}~\norm{K_{F,x}}_{\infty}~\norm{\mu}_{TV}~ +~ \norm{F_y}_{\infty}~\norm{K_{F,x}}_{W^{1,\infty}}~\big(p_F(\mu,\nu)\big)^{\alpha}}. 
\end{multline*}
so \eqref{Deltaf_x_estimate} follows again from Lemma \ref{iteration_in_k}.
\end{proof}
\begin{rem}\label{practical_iteration_lemma}
After two examples (Lemmas \ref{iteration_in_h} and \ref{iteration_in_k}), Reader should have some intuition in how iteration lemma (Lemma \ref{discrete_Gronwall}) applies to our approximating procedure. Briefly speaking, if $a_m$ is an iterated quantity ($m = 0, 1, ..., 2^k$) with $a_0 = 0$ we have implication:
$$
a_m \leq a_{m-1}e^{C2^{-k}} + \Big(\frac{C}{2^k}\Big)^{\gamma} \implies a_m \leq \bar{C}\Big(\frac{1}{2^k}\Big)^{\gamma-1}.
$$
for a possibly different constant $\bar{C}$. Therefore, we are interested in making $\gamma$ as big as possible. On the other hand, if we obtain a bound of the form:
\begin{equation}\label{general_exp_estimates}
a_m \leq a_{m-1} + \Big(\frac{C}{2^k}\Big)^{\gamma_1} +  \Big(\frac{C}{2^k}\Big)^{\gamma_2} + \Big(\frac{C}{2^k}\Big)^{\gamma_3} + ...,
\end{equation}
there is no point in recording $\gamma_1$, $\gamma_2$, $\gamma_3$ and so on. It is sufficient to know what is the smallest value of $\gamma_k$ appearing on (RHS) of \eqref{general_exp_estimates} as this is the only thing that matters.
\end{rem}
\subsection{Differentiability of map $h \mapsto \mu^{h,k}_t$}
In this Section, we prove that the approximating sequence $\mu^{h,k}_t$ is Fr\'echet differentiable with respect to $h$. This is one of the two results needed for application of Lemma \ref{unif_conv_cont_limit}. As already suggested in the introduction to Section \ref{sect4}, the main idea is to propagate differentiability from interval $[m \frac{T}{2^k}, (m+1) \frac{T}{2^k}]$ to $[(m+1) \frac{T}{2^k}, (m+2) \frac{T}{2^k}]$ while in the first interval $[0,\frac{T}{2^k}]$ differentiability follows directly from linear theory.

The main obstacle is that for $m > 0$, at interval $[m \frac{T}{2^k}, (m+1) \frac{T}{2^k}]$ we solve problem with perturbed nonlinearities (this can be handled by Theorem \ref{lin_per_1}) and with perturbed initial condition. To address this issue, we introduce intermediate measure that evolves with perturbed flow but starts from non-perturbed initial condition, similarly as in the proof of Lemma \ref{iteration_in_h}. 

\begin{thm}\label{diff_k_fixed}
Fix $k \in \mathbb{N}$. Let $\mu^{h,k}_t$ be the solution to \eqref{spm_non_approx}. Then, map $[-\frac{1}{2}, \frac{1}{2}] \ni h \mapsto \mu^{h,k}_t$ is Fr\'echet differentiable in $C([0,T], Z)$. Moreover, the derivative is H\"older continuous with constant bounded by$C_N 2^{k(1-\alpha)}$. 
\end{thm}
\begin{proof}
The proof goes by induction on time intervals. Let $t \in [0, \frac{T}{2^k}]$. Then, $\mu_t^{h,k}$ solves the equation:
\begin{equation*}\label{spm_non_approx_m=0}
\left\{ \begin{array}{lll}
\partial_t \mu_t + \partial_x(b^h(x, \mu_{0}) \mu_t) & = c(x, \mu_{0})\mu_t & \mathbb{R}^{+} \times  [0, \frac{T}{2^k}],\\
b^h(0, \mu_{0}) D_{\lambda}\mu_t(0) &= \int_{\mathbb{R}^+} a^h(x,\mu_{0}) d\mu_t(x) &  [0, \frac{T}{2^k}], \\
\end{array} \right.
\end{equation*}
with initial measure $\mu_{0}$. In this case, differentiability and H\"older continuity with some constant $C^{(0)}:= C_N \frac{T}{2^k}$ follows directly from Theorem \ref{lin_per}. However, substantial problems arise in next intervals.

Suppose that the result holds for some interval $[m \frac{T}{2^k}, (m+1) \frac{T}{2^k}]$ with H\"older constant of the derivative $C^{(m)}$. Let $t \in [(m+1) \frac{T}{2^k}, (m+2) \frac{T}{2^k}]$, $t^* =  (m+2) \frac{T}{2^k}$ and $t_*:= (m+1) \frac{T}{2^k}$. We introduce an intermediate measure $\bar{\mu}_t^{h+\Delta h, k}$ starting at time $(m+1) \frac{T}{2^k}$ from the nonperturbed initial measure $\mu_{(m+1) \frac{T}{2^k}}^{h,k}$ but evolving with perturbed flow:
\begin{equation*}
\left\{ \begin{array}{lll}
\partial_t \mu_t + \partial_x(b^{h+\Delta h}(x, \mu^{h+\Delta h}_{m\frac{T}{2^k}}) \mu_t) & = c^{h+\Delta h}(x, \mu^{h+\Delta h}_{m\frac{T}{2^k}})\mu_t & \mathbb{R}^{+} \times  [(m+1)\frac{T}{2^k}, (m+2)\frac{T}{2^k}],\\
b^{h+\Delta h}(0, \mu^{h+\Delta h}_{m\frac{T}{2^k}}) D_{\lambda}\mu_t(0) &= \int_{\mathbb{R}^+} a^{h+\Delta h}(x,\mu^{h+\Delta h}_{m\frac{T}{2^k}}) d\mu_t(x) &  [(m+1)\frac{T}{2^k}, (m+2)\frac{T}{2^k}], \\
\end{array} \right.
\end{equation*}
Then, we write:
\begin{equation}\label{splitting}
\frac{\mu_t^{h + \Delta h, k} - \mu_t^{h,k}}{\Delta h} = \frac{\mu_t^{h + \Delta h, k} - \bar{\mu}_t^{h+\Delta h,k}}{\Delta h} + \frac{\bar{\mu}_t^{h + \Delta h, k} - {\mu}_t^{h,k}}{\Delta h} = A_1 + A_2,
\end{equation}
and the plan is to show that both terms have limits in $Z$ when $\Delta h \to 0$. For term $A_1$, we will demonstrate that it is a Cauchy sequence. To this end, for some small $\Delta h_1 $ and $\Delta h_2$ we write:
\begin{multline}\label{estimate_begin}
\norm{\frac{\mu_t^{h + \Delta h_1, k} - \bar{\mu}_t^{h+\Delta h_1,k}}{\Delta h_1} -  \frac{\mu_t^{h + \Delta h_2, k} - \bar{\mu}_t^{h+\Delta h_2,k}}{\Delta h_2}}_Z \\
= \sup_{\xi \in C^{1+\alpha}, \norm{\xi} \leq 1} \int_{\mathbb{R}^+} \xi \frac{ d\mu_t^{h + \Delta h_1, k} - d\bar{\mu}_t^{h+\Delta h_1,k}}{\Delta h_1} -  \int_{\mathbb{R}^+} \xi \frac{d\mu_t^{h + \Delta h_2, k} - d\bar{\mu}_t^{h+\Delta h_2,k}}{\Delta h_2} \\ =
 \sup_{\xi \in C^{1+\alpha}, \norm{\xi} \leq 1} \int_{\mathbb{R}^+}  \varphi_{\xi, t}^{h+\Delta h_1}(0,x) \frac{ d\mu_{t_*}^{h + \Delta h_1, k} - d{\mu}_{t_*}^{h,k}}{\Delta h_1} -  \int_{\mathbb{R}^+} \varphi_{\xi, t}^{h+\Delta h_2}(0,x) \frac{d\mu_{t_*}^{h + \Delta h_2, k} - d{\mu}_{t_*}^{h,k}}{\Delta h_2} 
\end{multline}
where we applied semigroup property \eqref{semigroup}. Here, for any $H \in [-\frac{1}{2}, \frac{1}{2}]$, function $\varphi_{\xi, t}^{H}(s,x)$ (with $s \in [0,t-t_*]$) solves the implicit equation:
\begin{multline}\label{impl_eqn_in_sect4}
\varphi_{\xi,t}^H(s,x) = \xi(X_{\bar{b}(H,\cdot)}(t-t_*-s,x))e^{\int_0^{t-t_*-s}\bar{c}(H, X_{\bar{b}(H,\cdot)}(u,x)) du} +\\ + \int_0^{t-t_*-s}\bar{a}(H,X_{\bar{b}(H,\cdot)}(u,x))\varphi_{\xi, t}^H(u+s,0)e^{\int_0^u \bar{c}(H, X_{b(H,\cdot)}(v,x)) dv} du
\end{multline}
%\norm{x \mapsto \varphi^h(x,s)}_{C^{1+\alpha}} \leq C_{L,M} (1 + \norm{Da}_{\alpha} + \norm{Db}_{\alpha} + \norm{Dc}_{\alpha}) e^{C_{L,M} T} 
where $\bar{a}(h,x) = a(x, \mu_{t_*}^{h,k})$, $\bar{b}(h,x) = b(x, \mu_{t_*}^{h,k})$ and $\bar{c}(h,x) = c(x, \mu_{t_*}^{h,k})$. Now, by induction hypothesis, the map $h \mapsto \mu_{t_*}^{h,k}$ is $C^{1+\alpha}$ so we can use Theorem \ref{Nem_op} to obtain bounds on the function $\bar{f}(h,x)$, where $f = a,b,c$. Recall that these estimates were formulated in terms of four constants $C_T$, $C_L$, $C_F$ and $C_H$ defined in Section \ref{regularity_estimates_41}. However, 
\begin{itemize}
\item $C_F \leq C_N$ by the definition of $C_N$,
\item $C_T \leq C_N$ by Lemma \ref{total_stability},
\item $C_L \leq C_N$ by Lemma \ref{iteration_in_h},
\item $C_H \leq C^{(m)}$ by the definition of $C_H$.
\end{itemize}
Therefore, in view of Theorem \ref{Nem_op}, $\norm{f_x}_{\alpha,x}, \norm{f_x}_{\alpha,h}, \norm{f_h}_{\alpha,x} \leq C_N$ and $\norm{f_h}_{\alpha,h} \leq C_N(1+ C^{(m)})$. This implies that the constants $C_L$, $H_L$ and $G_L$ in estimates for a linear equation (Corollary \ref{Holder_cont_phi}) are bounded by $C_N$, $C_N$ and $C_N(1+ C^{(m)})$ respectively. Therefore, using Corollary \ref{Holder_cont_phi}, we conclude that $\norm{{\partial_x} \varphi_{\xi,t}^h}_{\alpha,h} \leq C_N 2^{-k\alpha}$. In particular, using also \eqref{lip_pert}, 
$$
\norm[1]{\varphi_{\xi, t}^{h_1} - \varphi_{\xi, t}^{h_2}}_{W^{1,\infty},x} \leq C_N \max \big(2^{-k}, 2^{-k\alpha} \big)  \abs{h_1 - h_2}^{\alpha} \leq {C_N}{2^{-k \alpha}} \abs{h_1 - h_2}^{\alpha}.
$$
Moreover, from Corollary \ref{Holder_cont_phi} we deduce $\norm[1]{\varphi_{\xi, t}^{h}}_{W^{1,\infty}, x} \leq e^{C_N \frac{T}{2^k} }$ and $\norm[1]{\partial_x \varphi_{\xi, t}^{h}}_{\alpha,x} \leq e^{C_N \frac{T}{2^k} }$. Then, we continue estimate in \eqref{estimate_begin}:
\begin{multline*}
\sup_{\xi \in C^{1+\alpha}, \norm{\xi} \leq 1} \int_{\mathbb{R}^+}  \Big(\varphi_{\xi, t}^{h+\Delta h_1}(0,x) -  \varphi_{\xi, t}^{h+\Delta h_2}(0,x) \Big) \frac{ d\mu_{t_*}^{h + \Delta h_1, k}(x) - d{\mu}_{t_*}^{h,k}(x)}{\Delta h_1} +   
\\ + \int_{\mathbb{R}^+} \varphi_{\xi, t}^{h+\Delta h_2}(0,x)\Bigg( \frac{ d\mu_{t_*}^{h + \Delta h_1, k}(x) - d{\mu}_{t_*}^{h,k}(x)}{\Delta h_1} - \frac{d\mu_{t_*}^{h + \Delta h_2, k}(x) - d{\mu}_{t_*}^{h,k}(x)}{\Delta h_2} \Bigg) \leq \\ \leq
\frac{C_N}{2^{k\alpha}}\frac{1}{\Delta h_1}\abs{\Delta h_1 - \Delta h_2}^{\alpha} p_F(\mu_{t_*}^{h + \Delta h_1, k}, d{\mu}_{t_*}^{h,k}) +e^{C_N\frac{T}{2^k} } \norm{\frac{\mu_{t_*}^{h + \Delta h_1, k} - {\mu}_{t_*}^{h,k}}{\Delta h_1} - \frac{\mu_{t_*}^{h + \Delta h_2, k} - {\mu}_{t_*}^{h,k}}{\Delta h_2}}_{Z},
\end{multline*}
where we used $W^{1,\infty}$ and $C^{1+\alpha}$ estimates recalled above. Now, the first term is bounded by $\frac{C_N}{2^{k\alpha}} \abs{\Delta h_1 - \Delta h_2}^{\alpha}$ in view of Lemma \ref{iteration_in_h} while the second is a Cauchy difference converging to zero by induction hypothesis (we have $t_* \in [m \frac{T}{2^k}, (m+1) \frac{T}{2^k}]$).

We move on to study term $A_2$ which is much easier. Since $\bar{\mu}_t^{h + \Delta h, k}$ and ${\mu}_t^{h,k}$ has the same starting point at $t_*$, this is exactly the setting of Theorem \ref{lin_per_1}. Therefore, $A_2$ is convergent. Moreover, H\"older constant of the limit is bounded by $\frac{C_N}{2^k} C^{(m)}$.

Therefore, $\frac{\mu^{h+\Delta h, k}_t - \mu^{h,k}_t}{\Delta h}$ is a Cauchy sequence in $Z$, and so it converges to the limit denoted by $\frac{\partial}{\partial h} \mu^{h,k}_t$. To conclude the proof, we need to check that map $[-\frac{1}{2}, \frac{1}{2}] \ni H \mapsto \frac{\partial}{\partial h} \mu^{h,k}_t|_{h=H}$  is H\"older continuous and estimate its norm. 

We apply the similar splitting as in \eqref{splitting}:
\begin{multline*}
{\partial_h} \mu^{h,k}_t|_{h=H_1} - {\partial_h} \mu^{h,k}_t|_{h=H_2} = \\=
\lim_{\Delta h \to 0} \underbrace{\frac{\mu_t^{H_1 + \Delta h, k} - \bar{\mu}_t^{H_1+\Delta h,k}}{\Delta h} -
 \frac{\mu_t^{H_2 + \Delta h, k} - \bar{\mu}_t^{H_2+\Delta h,k}}{\Delta h}}_{:=B_1} +
\underbrace{\lim_{\Delta h \to 0} \frac{\bar{\mu}_t^{H_1 + \Delta h, k} - {\mu}_t^{H_1,k}}{\Delta h} - \frac{\bar{\mu}_t^{H_2 + \Delta h, k} - {\mu}_t^{H_2,k}}{\Delta h} }_{:=B_2}
\end{multline*}
where the limit is taken in the space $Z$. For $B_2$, we deduce from Theorem \ref{lin_per_1} that $\abs{B_2} \leq \frac{C_N}{2^k} C^{(m)} \abs{H_1 - H_2}^{\alpha}$. For $B_1$, we apply semigroup property \eqref{semigroup} exactly like above:
\begin{multline}
\norm{\frac{\mu_t^{H_1 + \Delta h, k} - \bar{\mu}_t^{H_1+\Delta h,k}}{\Delta h} -  \frac{\mu_t^{H_2 + \Delta h, k} - \bar{\mu}_t^{H_2+\Delta h,k}}{\Delta h}}_Z =\\
= \sup_{\xi \in C^{1+\alpha}, \norm{\xi} \leq 1} \int_{\mathbb{R}^+} \xi \frac{ d\mu_t^{H_1 + \Delta h, k}(x) - d\bar{\mu}_t^{H_1+\Delta h,k}(x)}{\Delta h} -  \int_{\mathbb{R}^+} \xi \frac{d\mu_t^{H_2 + \Delta h, k}(x) - d\bar{\mu}_t^{H_2+\Delta h,k}(x)}{\Delta h} =\\ =
 \sup_{\xi \in C^{1+\alpha}, \norm{\xi} \leq 1} \int_{\mathbb{R}^+}  \varphi_{\xi, t}^{H_1+\Delta h}(0,x) \frac{ d\mu_{t_*}^{H_1 + \Delta h, k}(x) - d{\mu}_{t_*}^{H_1,k}(x)}{\Delta h} -  \int_{\mathbb{R}^+} \varphi_{\xi, t}^{H_2+\Delta h}(0,x) \frac{d\mu_{t_*}^{H_2 + \Delta h, k}(x) - d{\mu}_{t_*}^{H_2,k}(x)}{\Delta h}  =\\ =
  \sup_{\xi \in C^{1+\alpha}, \norm{\xi} \leq 1} \int_{\mathbb{R}^+}  \Big(\varphi_{\xi, t}^{H_1+\Delta h}(0,x) - \varphi_{\xi, t}^{H_2+\Delta h}(0,x)  \Big) \frac{ d\mu_{t_*}^{H_1 + \Delta h, k}(x) - d{\mu}_{t_*}^{H_1,k}(x)}{\Delta h}+ \\ +  \int_{\mathbb{R}^+} \varphi_{\xi, t}^{H_2+\Delta h}(0,x) \Bigg(\frac{ d\mu_{t_*}^{H_1 + \Delta h, k}(x) - d{\mu}_{t_*}^{H_1,k}(x)}{\Delta h} - \frac{d\mu_{t_*}^{H_2 + \Delta h, k}(x) - d{\mu}_{t_*}^{H_2,k}(x)}{\Delta h}\Bigg)\leq \\ 
  \leq \frac{C_N}{2^{k\alpha}} \abs{H_1 - H_2}^{\alpha} + e^{C_N\frac{T}{2^k}} \norm{\frac{\mu_{t_*}^{H_1 + \Delta h, k} - {\mu}_{t_*}^{H_1,k}}{\Delta h} -  \frac{\mu_{t_*}^{H_2 + \Delta h, k} - {\mu}_{t_*}^{H_2,k}}{\Delta h}}_Z.
\end{multline}
Since, as we send $\Delta h \to 0$:
$$
\norm{\frac{\mu_{t_*}^{H_1 + \Delta h, k} - {\mu}_{t_*}^{H_1,k}}{\Delta h} -  \frac{\mu_{t_*}^{H_2 + \Delta h, k} - {\mu}_{t_*}^{H_2,k}}{\Delta h}}_Z \leq 
C^{(m)} \abs{H_1 - H_2}^{\alpha},
$$ 
so combining this with $\abs{B_2} \leq \frac{C_N}{2^k} C^{(m)} \abs{H_1 - H_2}^{\alpha}$, we obtain the following iteration inequality:
\begin{equation*}
C^{(m+1)} \leq \frac{C_N}{2^{k\alpha}} + \frac{C_N}{2^k} C^{(m)} + e^{C_N\frac{T}{2^k}} C^{(m)} \leq \frac{C_N}{2^{k\alpha}} + e^{C_N\frac{T}{2^k}} C^{(m)},
\end{equation*}
where $m = 0, ..., 2^{k-1}$ and $C^0 = C_N 2^{-k}$. Therefore, Lemma \ref{discrete_Gronwall} implies:
$$
C^{(m)} \leq  C_N 2^{-k} e^{C_N \frac{T}{2^k} \cdot 2^k} +  \frac{C_N}{2^{k\alpha}} \frac{2^k}{T} \frac{T}{2^k} \frac{\big(e^{C_N \frac{T}{2^k}}\big)^m - 1}{e^{C_N \frac{T}{2^k}} - 1} \leq C_N 2^{k(1-\alpha)}.
$$
\end{proof}
We conclude this section with some sort of continuity-in-time-result for the obtained derivative:
\begin{lem}\label{cont_der_fixed_time}
Let $f\in C^{1+\alpha}(\mathbb{R}^+)$. Let $t_* = l\frac{T}{2^k}$ and $t_f = (l+1)\frac{T}{2^k}$. Then 
$$
\bigg({\partial_H}\mu^{h,k}_{t_f} \Big|_{H=h} - {\partial_H}\mu^{h,k}_{t_*} \Big|_{H=h}, f(x)\bigg)_{(Z, C^{1+\alpha})} \leq C(C_N, \norm{f}_{C^{1+\alpha}}) 2^{-k \alpha}.
$$
\end{lem}
\begin{proof}
We compute using semigroup property \eqref{semigroup}, adopting notation from \eqref{impl_eqn_in_sect4}:
\begin{multline*}
\bigg({\partial_H}\mu^{h,k}_{t_f} \Big|_{H=h} - {\partial_H}\mu^{h,k}_{t_*} \Big|_{H=h}, f(x)\bigg)_{(Z, C^{1+\alpha})} = \lim_{\Delta h \to 0} \bigg(\frac{\mu^{h+\Delta h,k}_{t_f} - \mu^{h,k}_{t_f}}{\Delta h} - \frac{\mu^{h+\Delta h,k}_{t_*}- \mu^{h,k}_{t_*}}{\Delta h}, f(x)\bigg)_{(Z, C^{1+\alpha})} =
 \\ = \lim_{\Delta h \to 0} \int_{\mathbb{R}^+} f(x) \frac{d\mu^{h+\Delta h,k}_{t_f}(x) - d\mu^{h,k}_{t_f}(x)}{\Delta h} - \frac{d\mu^{h+\Delta h,k}_{t_*}(x)- d\mu^{h,k}_{t_*}(x)}{\Delta h} = \\ =\lim_{\Delta h \to 0}
  \int_{\mathbb{R}^+} \frac{\varphi^{h+\Delta h}_{f, \frac{T}{2^k}}(0,x) d\mu^{h+\Delta h,k}_{t_*}(x) - \varphi^h_{f, \frac{T}{2^k}}(0,x) d\mu^{h,k}_{t_*}(x)}{\Delta h} - \frac{f(x)d\mu^{h+\Delta h,k}_{t_*}(x)- f(x)d\mu^{h,k}_{t_*}(x)}{\Delta h} = \\ =\lim_{\Delta h \to 0}
  \int_{\mathbb{R}^+} \frac{\big(\varphi^{h+\Delta h}_{f, \frac{T}{2^k}}(0,x) - f(x)\big) d\mu^{h+\Delta h,k}_{t_*}(x) - \big(\varphi^h_{f, \frac{T}{2^k}}(0,x)  - f(x)\big) d\mu^{h,k}_{t_*}(x)}{\Delta h} = \\ =
 \lim_{\Delta h \to 0}  \int_{\mathbb{R}^+} \frac{\big(\varphi^{h+\Delta h}_{f, \frac{T}{2^k}}(0,x) - f(x)\big) d\big(\mu^{h,k}_{t_*}(x) - \mu^{h+\Delta h,k}_{t_*}(x)\big)}{\Delta h} + \frac{\big(\varphi^{h+\Delta h}_{f, \frac{T}{2^k}}(0,x)  - \varphi^h_{f, \frac{T}{2^k}}(0,x) \big)}{\Delta h}d\mu^{h,k}_{t_*} (x)\leq \\ \leq \lim_{\Delta h \to 0} \norm{\varphi^{h+\Delta h}_{f, \frac{T}{2^k}}(0,x) - f(x)}_{W^{1,\infty}(\mathbb{R}^+), x}\frac{p_F(\mu^{h,k}_{t_*},\mu^{h+\Delta h,k}_{t_*})}{\Delta h} + \norm{{\partial_h} \varphi^{h}_{f, \frac{T}{2^k}}}_{\infty}\mu^{h,k}_{t_*}(\mathbb{R}^+).
\end{multline*}
Now, the first term is bounded by $C_N 2^{-k \alpha}$ due to Lemmas \ref{kicked_continuity} and \ref{iteration_in_h} while the second term is controlled by $C_N 2^{-k}$ due to Corollary \ref{Holder_cont_phi} and Lemma \ref{total_stability}.
\end{proof}
\subsection{Uniform convergence of difference quotients and proof of Theorem \ref{nonlin_per}}\label{uniform_convergence_section}
In this Section, we will prove, that under assumption $\alpha > \frac{1}{2}$, the sequence $\frac{\mu_t^{h+\Delta h,k} - \mu_t^{h,k}}{\Delta h}$ converges uniformly in $Z$ for all $\Delta h \in (-\frac{1}{2}, \frac{1}{2}) \setminus \{0\}$ as $k \to \infty$. To achieve this, we consider quantity
$$
\Delta^{k,t}:= \sup_{\Delta h \in (-\frac{1}{2}, \frac{1}{2})}\norm{\frac{\mu_t^{h+\Delta h,k+1} - \mu_t^{h,k+1}}{\Delta h} - 
\frac{\mu_t^{h+\Delta h,k} - \mu_t^{h,k}}{\Delta h}}_Z
$$
and the plan is to obtain the bound $\Delta^{k,t} \leq C_N2^{-k\beta}$ for some constant $\beta>0$. Consider interval of time $I_l = [l \frac{T}{2^k}, (l+1)\frac{T}{2^k}]$. Denote $t_* = l \frac{T}{2^k}$, $t^* = (l+1)\frac{T}{2^k}$  and $t_m = l \frac{T}{2^k}  + \frac{T}{2^{k+1}}$ (this is the middle of the interval $I_l$). Suppose $t_c \in (t_m, t^*]$. Using semigroup property, \eqref{semigroup}, we can write:
\begin{multline}\label{analysis_delta_k}
\Delta^{k,t_c} =  \sup_{\Delta h \in (-\frac{1}{2}, \frac{1}{2})} \sup_{\xi \in C^{1+\alpha}: \norm{\xi}\leq 1}\int_{\mathbb{R}^+} \xi \Bigg(\frac{d\mu_{t_c}^{h+\Delta h,k+1} - d\mu_{t_c}^{h,k+1}}{\Delta h}  - \frac{d\mu_{t_c}^{h+\Delta h,k} - d\mu_{t_c}^{h,k}}{\Delta h} \Bigg) = \\
= \sup_{\Delta h \in (-\frac{1}{2}, \frac{1}{2})} \sup_{\xi \in C^{1+\alpha}: \norm{\xi} \leq 1} \int_{\mathbb{R}^+} \frac{\varphi_{\xi}^{h+\Delta h, k+1}(0,x)d\mu_{t_m}^{h+\Delta h,k+1}(x) -
\varphi_{\xi}^{h, k+1}(0,x)d\mu_{t_m}^{h,k+1}(x)}{\Delta h}\\ - \frac{\varphi_{\xi}^{h+\Delta h, k}(\frac{T}{2^{k+1}},x)d\mu_{t_m}^{h+\Delta h,k}(x)  - 
\varphi_{\xi}^{h, k}(\frac{T}{2^{k+1}},x)d\mu_{t_m}^{h,k}(x)}{\Delta h},
\end{multline}
where $\varphi_{\xi}^{h, k}$ and $\varphi_{\xi}^{h, k+1}$  solve appropriate implicit equations like \eqref{impl_eqn_in_sect4}. Analysis of the resulting expression seems to be difficult. Indeed, our target estimate has to capture decay when both $\Delta h \to 0$ and $k \to \infty$. If we used triangle inequality, we would lose one of these effects. However, recall that $\mu_{0}^{h,k} = \mu_{0}^{h+\Delta h,k} = \mu_{0}^{h,k+1} = \mu_{0}^{h+\Delta h,k+1} =\mu_0$ so we can apply semigroup property \eqref{semigroup} $m$ times more, moving down to the time $t = 0$. Then, we will end up with integration with respect to the same measure as initial condition is the same. This corresponds to the iterations of equation \eqref{sol_dual}, motivating the following definition:
\begin{defn}
Let $a(h,x), b(h,x), c(h,x), \xi(h,x): \mathbb{R}^+ \times [-\frac{1}{2}, \frac{1}{2}]\to \mathbb{R}$ where $a$, $b$ and $c$ satisfy assumptions of Theorem \eqref{lin_per_1}. We say that $\xi$ generates function $\varphi^h_{\xi,t}(s,x)$ in time $t$ and with flow $(a,b,c)$ provided $\varphi^h_{\xi,t}$ solves: 
\begin{equation}\label{sol_dual_generator}
\varphi^h_{\xi,t}(s,x) = \xi(h,X_{b(h,\cdot)}(t-s,x))e^{\int_0^{t-s }c(h,X_b(u,x)) du} + \int_0^{t-s}a(h,X_b(u,x))\varphi^h_{\xi,t}(u+s,0)e^{\int_0^u c(h,X_b(v,x)) dv} du.
\end{equation}
\end{defn}
Fix $t_c \in (t_m, t^*]$. When we move down in time from $t_c$ to $t_m$ in integral \eqref{analysis_delta_k}, using semigroup property \eqref{semigroup}, we actually generate a new function:
\begin{itemize}
\item for the $(k+1)$-th approximation with flow $(a(x,\mu_{t_m}^{h,k+1}), b(x,\mu_{t_m}^{h,k+1}), c(x,\mu_{t_m}^{h,k+1}))$ in time $t = t_c - t_m$. It will be denoted by $\varphi^{h, k+1,k+1}_{\xi, t_c, 1}$ (this means: starting from function $\xi$ and time $t_c$ with perturbation parameter $h$, we moved down to the closest meshpoint with flow of the $(k+1)$-th approximation level and with respect to the mesh of diameter $\frac{T}{2^{k+1}}$). The generated function in the new integral expression will be evaluated at $s = 0$ so that $t - s = t_c - t_m$. 
\item for the $k$-th approximation with flow $(a(x,\mu_{t_*}^{h,k}), b(x,\mu_{t_*}^{h,k}), c(x,\mu_{t_m}^{h,k}))$ in time $t = t_c - t_*$. It will be denoted by $\varphi^{h, k,k+1}_{\xi, t_c, 1}$ (this means: starting from function $\xi$ and time $t_c$ with perturbation parameter $h$, we moved down to the closest meshpoint with flow of the $k$-th approximation level and with respect to the mesh of diameter $\frac{T}{2^{k+1}}$). The generated function in the new integral expression will be evaluated at $s = \frac{T}{2^{k+1}}$ so that $t - s = t_c - t_* - \frac{T}{2^{k+1}} = t_c - t_m$.
\end{itemize}
After this step, we can write:
\begin{multline*}
\Delta^{k,t_c} = \sup_{\Delta h \in (-\frac{1}{2}, \frac{1}{2})} \sup_{\xi \in C^{1+\alpha}:~\norm{\xi} \leq 1} \int_{\mathbb{R}^+} \frac{\varphi_{\xi,t_c,1}^{h+\Delta h, k+1, k+1}(0,x)d\mu_{t_m}^{h+\Delta h,k+1}(x) - 
\varphi_{\xi,t_c,1}^{h, k+1,k+1}(0,x)d\mu_{t_m}^{h,k+1}(x)}{\Delta h}  \\- \frac{\varphi_{\xi,t_c,1}^{h+\Delta h, k, k+1}(\frac{T}{2^{k+1}},x)d\mu_{t_m}^{h+\Delta h,k}(x) - 
\varphi_{\xi,t_c,1}^{h, k, k+1}(\frac{T}{2^{k+1}},x)d\mu_{t_m}^{h,k}(x)}{\Delta h}.
\end{multline*}
We apply semigroup property once more, generating new functions:
\begin{itemize}
\item for the $(k+1)$-th approximation with flow $(a(x,\mu_{t_*}^{h,k+1}), b(x,\mu_{t_*}^{h,k+1}), c(x,\mu_{t_*}^{h,k+1}))$ in time $t = \frac{T}{2^k+1}$. It will be denoted by $\varphi^{h, k+1,k+1}_{\xi, t_c, 2}$ (this means: starting from function $\xi$ and time $t_c$ with perturbation parameter $h$, we moved twice to the closest meshpoint with the flow of the $(k+1)$-th approximation level and with respect to the mesh of diameter $\frac{T}{2^{k+1}}$). The generated function in the new integral expression will be evaluated at $s = 0$ so that $t - s = \frac{T}{2^{k+1}}$. 
\item for $k$-th approximation with flow $(a(x,\mu_{t_*}^{h,k}), b(x,\mu_{t_*}^{h,k}), c(x,\mu_{t_m}^{h,k}))$ in time $t = \frac{T}{2^{k+1}}$. It will be denoted by $\varphi^{h, k,k+1}_{\xi, t_c, 2}$ (this means: starting from function $\xi$ and with perturbation parameter $h$, we moved twice to the closest meshpoint with flow of the $k$-th approximation level and with respect to the mesh of diameter $\frac{T}{2^{k+1}}$). The generated function in the new integral expression will be evaluated at $s = 0$ so that $t - s = \frac{T}{2^{k+1}}$.
\end{itemize}
After this step, we have:
\begin{multline*}
\Delta^{k,t_c} = \sup_{\Delta h \in (-\frac{1}{2}, \frac{1}{2})} \sup_{\xi \in C^{1+\alpha}: \norm{\xi} \leq 1} \int_{\mathbb{R}^+} \frac{\varphi_{\xi,t_c,2}^{h+\Delta h, k+1, k+1}(0,x)d\mu_{t_*}^{h+\Delta h,k+1}(x) - 
\varphi_{\xi,t_c,2}^{h, k+1,k+1}(0,x)d\mu_{t_*}^{h,k+1}(x)}{\Delta h}  \\- \frac{\varphi_{\xi,t_c,2}^{h+\Delta h, k, k+1}(0,x)d\mu_{t_*}^{h+\Delta h,k}(x) - 
\varphi_{\xi,t_c,2}^{h, k, k+1}(0,x)d\mu_{t_*}^{h,k}(x)}{\Delta h}.
\end{multline*}
Eventually, after $2 l$ more iterations we end up with:
\begin{multline*}
\Delta^{k,t_c} =  \sup_{\Delta h \in (-\frac{1}{2}, \frac{1}{2})}  \sup_{\xi \in C^{1+\alpha}:~\norm{\xi} \leq 1} \int_{\mathbb{R}^+} \frac{\varphi_{\xi,t_c,2(l+1)}^{h+\Delta h, k+1, k+1}(0,x) - 
\varphi_{\xi,t_c,2(l+1)}^{h, k+1,k+1}(0,x)}{\Delta h} \\ - \frac{\varphi_{\xi,t_c,2(l+1)}^{h+\Delta h, k, k+1}(0,x) - 
\varphi_{\xi,t_c,2(l+1)}^{h, k, k+1}(0,x)}{\Delta h} d\mu_0(x) = \\
=  \sup_{\Delta h \in (-\frac{1}{2}, \frac{1}{2})}  \sup_{\xi \in C^{1+\alpha}: \norm{\xi} \leq 1} \int_{\mathbb{R}^+} \int_0^1 
\Big({\partial_H} \varphi_{\xi,t_c,2(l+1)}^{H, k+1, k+1}\Big|_{H = h+u\Delta h} - 
{\partial_H} \varphi_{\xi,t_c,2(l+1)}^{H, k, k+1}\Big|_{H = h+u\Delta h}\Big) du~d\mu_0(x)
\end{multline*}
In fact, we have proven the following estimate: if $t_c \in (l \frac{T}{2^{k+1}}, (l+1) \frac{T}{2^{k+1}}]$
\begin{equation}\label{the_most_important_bound}
\Delta^{k, t_c} \leq C_N  \sup_{\xi \in C^{1+\alpha}: \norm{\xi} \leq 1} \norm[2]{{\partial_H} \varphi_{\xi,t_c,l+1}^{H, k+1, k+1} - 
{\partial_H} \varphi_{\xi,t_c,l+1}^{H, k, k+1}}_{\infty}.
\end{equation}
Therefore, we can set up iterations for $\Gamma^{k, t_c, h}_v :=  \sup_{\xi \in C^{1+\alpha}: \norm{\xi} \leq 1}\norm{{\partial_h} \varphi_{\xi,t_c,v}^{h, k+1, k+1} - {\partial_h} \varphi_{\xi,t_c,v}^{h, k, k+1}}_{\infty}$ where $v = 0,1,..., l+1$. Analysis of these quantities can be performed with Theorem \ref{thm_main_diff} from Section \ref{perturbation_mdlfcns_section}. It suggests that we also need bounds on $$\Gamma^{k, t_c}_v := \sup_{\xi \in C^{1+\alpha}: \norm{\xi} \leq 1} \norm[2]{ \varphi_{\xi,t_c,v}^{h, k+1, k+1} - \varphi_{\xi,t_c,v}^{h, k, k+1}}_{\infty}, \quad \Gamma^{k, t_c, x}_v := \sup_{\xi \in C^{1+\alpha}: \norm{\xi} \leq 1} \norm[2]{{\partial_x} \varphi_{\xi,t_c,v}^{h, k+1, k+1} - {\partial_x} \varphi_{\xi,t_c,v}^{h, k, k+1}}_{\infty}.$$ We begin our analysis with establishing stability of sequence $\{\varphi_{\xi,t_c,v}^{h, m, k+1}  \}$ as $m \leq k+1$, $v = 0, 1, ..., l+1$ and $t_c \in (l \frac{T}{2^{k+1}}, (l+1) \frac{T}{2^{k+1}}]$ .
\begin{lem}\label{basic_stability_perturbations}
Let $\xi \in C^{1+\alpha}(\mathbb{R}^+)$ with $\norm{\xi}_{C^{1+\alpha}} \leq 1$ and $t_c \in (l \frac{T}{2^{k+1}}, (l+1) \frac{T}{2^{k+1}}]$. There is a constant $C_N$ such that for any $m \leq k+1$ and $v = 0, 1, ..., l+1$ sequence $\{\varphi_{\xi,t_c,v}^{h, m, k+1}  \}$ satisfies:
$$
\norm[2]{\varphi_{\xi,t_c,v}^{h, m, k+1}}_{W^{1,\infty}},\quad \norm[2]{{\partial_x} \varphi_{\xi,t_c,v}^{h, m, k+1}}_{\alpha,x},\quad \norm[2]{{\partial_h} \varphi_{\xi,t_c,v}^{h, m, k+1}}_{\alpha,x}  \leq C_N.
$$
\end{lem}
\begin{proof}
We use bounds provided by Lemma \ref{est_dual_sol_fin_fcns}. As in the proof of Theorem \ref{diff_k_fixed}, we note that constants $C_L$ and $H_L$ in these bounds become $C_N$ in our nonlinear setting. Then,
$$
\norm{\varphi_{\xi,t_c,v+1}^{h, m, k+1}}_{W^{1,\infty}} \leq \norm{\varphi_{\xi,t_c,v}^{h, m, k+1}}_{W^{1,\infty}} e^{ C_N \frac{T}{2^{k+1}}}
$$
so by Lemma \ref{discrete_Gronwall} and $v \leq (l+1) \leq 2^{k+1}$:
$$
\norm[1]{\varphi_{\xi,t_c,v}^{h, m, k+1}}_{W^{1,\infty}} \leq \norm[1]{\xi}_{W^{1,\infty}} e^{C_N v\frac{T}{2^{k+1}} } \leq e^{C_N T } \leq C_N.
$$
Similarly,
$$
\norm[2]{{\partial_x} \varphi_{\xi,t_c,v+1}^{h, m, k+1}}_{\alpha,x} \leq \max{\bigg(\norm[2]{{\partial_x} \varphi_{\xi,t_c,v}^{h, m, k+1}}_{\alpha,x}, \norm[2]{ \varphi_{\xi,t_c,v}^{h, m, k+1}}_{W^{1,\infty}} \bigg)}e^{C_N \frac{T}{2^{k+1}} } + C_N \frac{T}{2^{k+1}} ,
$$
so as we already know that $ \norm[2]{\varphi_{\xi,t_c,v}^{h, m, k+1}}_{W^{1,\infty}} \leq C_N$, Lemma \ref{discrete_Gronwall} again imply:
$$
\norm[2]{{\partial_x} \varphi_{\xi,t_c,v}^{h, m, k+1}}_{\alpha,x} \leq e^{C_N l\frac{T}{2^{k+1}} } \max{(1, C_N)} + \frac{e^{l\frac{T}{2^{k+1}}} - 1}{e^{\frac{T}{2^{k+1}}} - 1} C_N \frac{T}{2^{k+1}}  \leq C_N,
$$
exactly like in \eqref{example_for_iterations}. Finally,
\begin{multline*}
\norm[2]{{\partial_h} \varphi_{\xi,t_c,v+1}^{h, m, k+1}}_{\alpha,x} \leq \max{\bigg(\norm[2]{{\partial_h} \varphi_{\xi,t_c,v}^{h, m, k+1}}_{\alpha,x}, 2\norm[2]{ \varphi_{\xi,t_c,v}^{h, m, k+1}}_{W^{1,\infty}} \bigg)}e^{C_N \frac{T}{2^{k+1}} } + C_N \frac{T}{2^{k+1}}  \norm[2]{{\partial_x}\varphi_{\xi,t_c,v}^{h, m, k+1}}_{\alpha,x} + \\ +C_N \frac{T}{2^{k+1}}  \norm[2]{\varphi_{\xi,t_c,v}^{h, m, k+1}}_{W^{1,\infty}} \leq 
\max{\bigg(\norm[2]{{\partial_h} \varphi_{\xi,t_c,v}^{h, m, k+1}}_{\alpha,x}, 2 C_N \bigg)}e^{C_N \frac{T}{2^{k+1}} } + C_N\frac{T}{2^{k+1}} 
\end{multline*}
so we conclude, exactly like above, that $\norm[2]{{\partial_h} \varphi_{\xi,t_c,v}^{h, m, k}}_{\alpha,x} \leq C_N$.
\end{proof}
\begin{lem}\label{first_estimates_for_Gamma}
Let $\xi \in C^{1+\alpha}(\mathbb{R}^+)$ with $\norm{\xi}_{C^{1+\alpha}} \leq 1$. Let $t_c \in (l \frac{T}{2^{k+1}}, (l+1) \frac{T}{2^{k+1}}]$. There is a constant $C_N$ such that $\Gamma^{k, t_c}_v \leq C_N2^{-k}$ and $\Gamma^{k, t_c,x}_v \leq C_N2^{k(1-2\alpha)}$, independently of $v = 0, 1, ..., l+1$. 
\end{lem}
\begin{proof}
Note that thanks to Corollary \ref{Deltaf_estimate_cor}, assumptions (C1)--(C3) of Theorem \ref{thm_main_diff} are satisfied. Therefore, using the bound \eqref{direct_estimate_on_fcns_dif_eqn} in Theorem \ref{thm_main_diff}, we deduce
$$
\Gamma^{k,t_c}_{v+1} \leq \Gamma^{k,t_c}_{v}e^{C_N \frac{T}{2^{k+1}}} + C_N (2^{-k})^2 e^{C_N \frac{T}{2^{k+1}}} \sup_{\xi \in C^{1+\alpha}: \norm{\xi} \leq 1} \big(2 + \norm[1]{\varphi_{\xi,t_c,v}^{h, k, k+1}}_{\infty}\big) \leq e^{C_N \frac{T}{2^{k+1}}} \Gamma^{k,t_c}_{v} + C_N2^{-2k}
$$
where we also used Lemma \ref{basic_stability_perturbations} (to bound $\norm[1]{\varphi_{\xi,l}^{h, k+1, k}}_{\infty}$ with $C_N$). Using \ref{discrete_Gronwall}, we conclude $\Gamma_{l,k} \leq C_N2^{-k}$ (compare also with Remark \ref{practical_iteration_lemma}). 

To establish inequality $\Gamma^{k,t_c,x}_{v} \leq C_N2^{k(1-2\alpha)}$, we use again Theorem \ref{thm_main_diff}. We note carefully that terms appearing in the estimate \eqref{precise_estimate_diff_dx_eq} can be bounded (below we call terms exactly like they appear in \eqref{precise_estimate_diff_dx_eq}):
\begin{itemize}
\item $\norm[1]{\xi - \bar{\xi}} \leq C_N2^{-k}$ due to estimate for $\Gamma^{k,t_c}_{v}$,
\item $\norm[1]{\bar{\xi}_x}_{\alpha,x}, \norm[1]{\bar{\xi}}_{W^{1,\infty}}, \norm[1]{\xi}_{W^{1,\infty}} \leq C_N$ due to Lemma \ref{basic_stability_perturbations}.
\end{itemize}
Since $1+\alpha \geq 2\alpha$, we obtain bound:
$$
\Gamma^{k,t_c,x}_{v+1} \leq e^{C_N \frac{T}{2^{k+1}} } \Gamma^{k,t_c,x}_{v} + C_N2^{-2k \alpha}.
$$
Using Lemma \ref{discrete_Gronwall}, we deduce $\Gamma^{k,t_c,x}_{v} \leq C_N 2^{k(1-2\alpha)}$ as desired.
\end{proof}
We finally move to study quantity $\Delta^{k,t_c}$. To bound it, we want to use \eqref{the_most_important_bound} and bound $\Gamma^{k,t_c,h}_{v}$ by iterations like above. To write an equation for the iterations, we apply \eqref{precise_estimate_diff_dh_eq} in Theorem \ref{thm_main_diff}:
\begin{lem}\label{first_nice_result_iteration_in_h}
Let $\xi \in C^{1+\alpha}(\mathbb{R}^+)$ with $\norm{\xi}_{C^{1+\alpha}} \leq 1$. Let $t_c \in (l \frac{T}{2^{k+1}}, (l+1) \frac{T}{2^{k+1}}]$. There is a constant $C_N$ such that if $l-v$ is even
$$
\Gamma^{k, t_c,h}_{v+1} \leq e^{C_N2^{-k}} \Gamma^{k, t_c,h}_v + 2^{-k}  \sup_{f=a,b,c} \norm[2]{{\partial_h}f\Big(x,\mu^{h,k+1}_{(l-v)\frac{T}{2^{k+1}}}\Big) - {\partial_h}f\Big(x,\mu^{h,k}_{(l-v)\frac{T}{2^{k+1}}} \Big)}_{\infty} + C_N 2^{-2\alpha k},
$$
while when $l-v$ is odd
$$
\Gamma^{k, t_c,h}_{v+1} \leq e^{C_N2^{-k}} \Gamma^{k, t_c,h}_v + 2^{-k}  \sup_{f=a,b,c} \norm[2]{{\partial_h}f\Big(x,\mu^{h,k+1}_{(l-v)\frac{T}{2^{k+1}}}\Big) - {\partial_h}f\Big(x,\mu^{h,k}_{(l-v-1)\frac{T}{2^{k+1}}} \Big)}_{\infty} +C_N 2^{-2\alpha k}.
$$
\end{lem}
\begin{proof}
Again, due to Corollary \ref{Deltaf_estimate_cor}, assumptions (C1)--(C3) of Theorem \ref{thm_main_diff} are satisfied. We want to use \eqref{precise_estimate_diff_dh_eq} as our iterated inequality. To this end, we observe that terms appearing in estimate \eqref{precise_estimate_diff_dh_eq} can be bounded (below we call terms exactly like they appear in \eqref{precise_estimate_diff_dh_eq}):  
\begin{itemize}
\item $\norm[1]{\xi -\bar{\xi}}_{\infty}$, $\norm[1]{\xi}_{W^{1,\infty}}$, $\norm[1]{\xi_x}_{\alpha,x}$ are controlled like in the proof of Lemma \ref{first_estimates_for_Gamma},
\item $\norm[1]{\xi_h}_{\alpha,x} \leq C_N$ due to Lemma \ref{basic_stability_perturbations}.
\end{itemize}
Term $\abs{\Delta f_h}$ in \eqref{precise_estimate_diff_dh_eq} is more subtle. Note carefully that:
$$
\Gamma^{k, t_c,h}_{v+1} = \sup_{\xi \in C^{1+\alpha}: \norm{\xi} \leq 1} \norm[2]{{\partial_h} \varphi_{\xi,t_c,v+1}^{h, k+1, k+1} - {\partial_h} \varphi_{\xi,t_c,v+1}^{h, k, k+1}}_{\infty},
\quad \quad
\Gamma^{k, t_c,h}_{v} = \sup_{\xi \in C^{1+\alpha}: \norm{\xi} \leq 1} \norm[2]{{\partial_h} \varphi_{\xi,t_c,v}^{h, k+1, k+1} - {\partial_h} \varphi_{\xi,t_c,v}^{h, k, k+1}}_{\infty}.
$$
In view of Definition \ref{def_of_Delta_F}, there are two cases:
\begin{itemize}
\item{If $(l-v)$ is even, $\varphi_{\xi,t_c,v+1}^{h, k+1, k+1}$ is generated from $\varphi_{\xi,t_c,v}^{h, k+1, k+1}$ with the flow 
$$
\Big(a\Big(x,\mu^{h,k+1}_{(l-v)\frac{T}{2^{k+1}}}\Big), b\Big(x,\mu^{h,k+1}_{(l-v)\frac{T}{2^{k+1}}}\Big), c\Big(x,\mu^{h,k+1}_{(l-v)\frac{T}{2^{k+1}}}\Big)\Big)
$$ and $\varphi_{\xi,t_c,v+1}^{h, k, k+1}$ is generated from $\varphi_{\xi,t_c,v}^{h, k, k+1}$ with the flow 
$$\Big(a\Big(x,\mu^{h,k}_{(l-v)\frac{T}{2^{k+1}}}\Big), b\Big(x,\mu^{h,k}_{(l-v)\frac{T}{2^{k+1}}}\Big), c\Big(x,\mu^{h,k}_{(l-v)\frac{T}{2^{k+1}}}\Big)\Big).
$$
Therefore,
\begin{equation*}
\abs{\Delta f_h} = \sup_{f=a,b,c} \norm[2]{{\partial_h}f\Big(x,\mu^{h,k+1}_{(l-v)\frac{T}{2^{k+1}}}\Big) - {\partial_h}f\Big(x,\mu^{h,k}_{(l-v)\frac{T}{2^{k+1}}} \Big)}_{\infty}.
\end{equation*}}
\item{If $(l-v)$ is odd, $\varphi_{\xi,t_c,v+1}^{h, k+1, k+1}$ is generated from $\varphi_{\xi,t_c,v}^{h, k+1, k+1}$ with the flow 
$$
\Big(a\Big(x,\mu^{h,k+1}_{(l-v)\frac{T}{2^{k+1}}}\Big), \Big(x,\mu^{h,k+1}_{(l-v)\frac{T}{2^{k+1}}}\Big), c\Big(x,\mu^{h,k+1}_{(l-v)\frac{T}{2^{k+1}}}\Big)\Big)
$$ 
but $\varphi_{\xi,t_c,v+1}^{h, k, k+1}$ is generated from $\varphi_{\xi,t_c,v}^{h, k, k+1}$ with the flow 
$$
\Big(a\Big(x,\mu^{h,k}_{(l-v-1)\frac{T}{2^{k+1}}}\Big), b\Big(x,\mu^{h,k}_{(l-v-1)\frac{T}{2^{k+1}}}\Big), c\Big(x,\mu^{h,k}_{(l-v-1)\frac{T}{2^{k+1}}}\Big)\Big).
$$ Therefore,
\begin{equation*}
\abs{\Delta f_h} = \sup_{f=a,b,c} \norm[2]{{\partial_h}f\Big(x,\mu^{h,k+1}_{(l-v)\frac{T}{2^{k+1}}}\Big) - {\partial_h}f\Big(x,\mu^{h,k}_{(l-v-1)\frac{T}{2^{k+1}}} \Big)}_{\infty}.
\end{equation*}}
\end{itemize}
If Reader finds it hard to deduce these statements, it may be helpful to consider case $v=0$ first. Finally, note that $2^{-k(1+\alpha)} \leq 2^{-2\alpha k}$ so now, the assertion follows directly from \eqref{precise_estimate_diff_dh_eq}.
\end{proof} 
Unfortunately, we still cannot bound the term $\abs{\Delta f_h}$ that appears above. We address this problem now.
\begin{lem}\label{estimate_for_deltaf_h}
Let $f$ be one of the model functions ($a$, $b$ or $c$) satisfying assumptions (N1)--(N4) of Theorem \ref{nonlin_per}. Let $t_* = j\frac{T}{2^k}$ and $t_m = t_* + \frac{T}{2^{k+1}}$. Then
$$
\norm[2]{{\partial_h}f(x,\mu_{t_*}^{k,h}) - {\partial_h}f(x,\mu_{t_*}^{k+1,h})}_{\infty}, ~\norm[2]{{\partial_h}f(x,\mu_{t_*}^{k,h}) - {\partial_h}f(x,\mu_{t_m}^{k+1,h})}_{\infty} \leq C_N2^{-k\alpha} +C_N\Delta^{k,t_*}.
$$
\end{lem} 
\begin{proof}
The first bound is a direct consequence of Lemmas \ref{bound_for_measure_der_h} and \ref{iteration_in_k}. To see the second one, we use again Lemmas \ref{bound_for_measure_der_h} and \ref{iteration_in_k} to obtain:
$$
 \norm[2]{{\partial_h}f(x,\mu_{t_*}^{k,h}) - {\partial_h}f(x,\mu_{t_m}^{k+1,h})}_{\infty} \leq C_N2^{-k\alpha} + C_N\norm[2]{{\partial_h}\mu_{t_m}^{k+1,h} - {\partial_h}\mu_{t_*}^{k,h}}_{Z}.
$$
However, due to Lemma \ref{cont_der_fixed_time}:
$$
\norm[2]{{\partial_h}\mu_{t_m}^{k+1,h} - {\partial_h}\mu_{t_*}^{k,h}}_{Z} \leq 
\norm[2]{{\partial_h}\mu_{t_m}^{k+1,h} - {\partial_h}\mu_{t_*}^{k+1,h}}_{Z} + 
\norm[2]{{\partial_h}\mu_{t_*}^{k+1,h} - {\partial_h}\mu_{t_*}^{k,h}}_{Z} \leq
C_N2^{-k\alpha} + C_N\Delta^{k,t_*}. 
$$
\end{proof}
Lemmas \ref{first_nice_result_iteration_in_h} and \ref{estimate_for_deltaf_h} obviously lead to:
\begin{cor}\label{bounds_but_still_not_complete}
Let $\xi \in C^{1+\alpha}(\mathbb{R}^+)$ with $\norm{\xi}_{C^{1+\alpha}} \leq 1$. Let $t_c \in (l \frac{T}{2^{k+1}}, (l+1) \frac{T}{2^{k+1}}]$. There is a constant $C_N$ such that if $l-v$ is even:
\begin{equation}\label{est_even}
\Gamma^{k, t_c,h}_{v+1} \leq e^{C_N2^{-k}} \Gamma^{k, t_c,h}_v + C_N 2^{-k} \Delta^{k,(l-v)\frac{T}{2^{k+1}}} + C_N 2^{-2\alpha k},
\end{equation}
while when $l-v$ is odd:
\begin{equation}\label{est_odd}
\Gamma^{k, t_c,h}_{v+1} \leq e^{C_N2^{-k}} \Gamma^{k, t_c,h}_v + C_N 2^{-k} \Delta^{k,(l-v-1)\frac{T}{2^{k+1}}} + C_N 2^{-2\alpha k}.
\end{equation}
\end{cor}
Using \eqref{the_most_important_bound}, we can easily control $\Delta^{k, m\frac{T}{2^{k+1}}}$ for even $m$:
\begin{lem}\label{almost_almost_there}
For even $m$, $\Delta^{k, m\frac{T}{2^{k+1}}} \leq C_N 2^{-k(2\alpha - 1)}$.
\end{lem}
\begin{proof}
We apply previous results with $t_c =  m\frac{T}{2^{k+1}} \in ((m-1)\frac{T}{2^{k+1}},  m\frac{T}{2^{k+1}}]$. Let $n \leq m$ be also an even integer. Using \eqref{est_odd} (with $l = m-1$ and $v = n-1$ so that $l-v = m - n$ is even) and \eqref{est_even} (with $l = m-1$ and $v = n-2$ so that $l-v = m-n+1$ is odd) we deduce:
\begin{multline}\label{first_final_iteration}
\Gamma_{n}^{k, m\frac{T}{2^{k+1}},h} \leq \Gamma^{k,  m\frac{T}{2^{k+1}}, h}_{n-1} e^{C_N2^{-k}} + 2^{-k} C_N\Delta^{k,(m-n)\frac{T}{2^{k+1}}} + 2^{-2\alpha k}C_N \leq \\ \leq
\Big(\Gamma^{k,  m\frac{T}{2^{k+1}}, h}_{n-2} e^{C_N2^{-k}} + 2^{-k} C_N\Delta^{k,(m-n)\frac{T}{2^{k+1}}} + 2^{-2\alpha k}C_N \Big)e^{C_N2^{-k}} + 2^{-k} C_N\Delta^{k,(m-n)\frac{T}{2^{k+1}}} + 2^{-2\alpha k}C_N \leq \\
\leq e^{C_N2^{-k}} \Gamma^{k,  m\frac{T}{2^{k+1}}, h}_{n-2}  + C_N 2^{-k} \Delta^{k,(m-n)\frac{T}{2^{k+1}}} + C_N 2^{-2\alpha k} 
\end{multline}
after standard modification of constant $C_N$. Since $\Gamma^{k,  m\frac{T}{2^{k+1}}, h}_{0} = 0$ and $\Delta^{k,0} = 0$, applying \eqref{first_final_iteration} inductively, we obtain:
\begin{multline*}
\Gamma_{m}^{k, m\frac{T}{2^{k+1}},h} \leq \sum_{n = 2, \text { n is even}}^m e^{C_N(m-n)2^{-k}}C_N 2^{-k} \Delta^{k,(m-n)\frac{T}{2^{k+1}}} + \sum_{n = 2, \text { n is even}}^m e^{C_N (m-n)2^{-k}} 2^{-2\alpha k}C_N \leq \\ \leq \sum_{n = 2, \text { n is even}}^{m-2} C_N 2^{-k} \Delta^{k,(m-n)\frac{T}{2^{k+1}}} + C_N 2^{(1-2\alpha) k}
\end{multline*}
after summation of the geometric series and noting that $m \leq 2^{k+1}$. Finally, recalling \eqref{the_most_important_bound} yields:
\begin{equation}\label{almost_there}
\Delta^{k,m\frac{T}{2^{k+1}}} \leq \sum_{n = 2, \text { n is even}}^{m-2} C_N 2^{-k} \Delta^{k,(m-n)\frac{T}{2^{k+1}}} + C_N 2^{(1-2\alpha) k}.
\end{equation}
Let $g: [0,1] \to \mathbb{R}^+$ be the nonnegative, measurable function defined with
$$
g(s) = \sum_{l=2, l\text{ is even}}^{2^{k+1}} \Delta^{k,l \frac{T}{2^{k+1}}} \mathbbm{1}_{[\frac{l-2}{2^{k+1}},\frac{l}{2^{k+1}}]},
$$
so that \eqref{almost_there} implies a slightly weaker inequality:
$$
g(s) \leq C_N \int_0^s g(u) du + C_N 2^{(1-2\alpha) k}.
$$
Invoking Gronwall integral inequality concludes the proof.
\end{proof}
In view of Lemma \ref{almost_almost_there}, we can refine Corollary \ref{bounds_but_still_not_complete}:
\begin{cor}\label{bounds_now_complete}
Let $\xi \in C^{1+\alpha}(\mathbb{R}^+)$ with $\norm{\xi}_{C^{1+\alpha}} \leq 1$. Let $t_c \in (l \frac{T}{2^{k+1}}, (l+1) \frac{T}{2^{k+1}}]$. There is a constant $C_N$ such that:
\begin{equation}\label{est_even_and_odd}
\Gamma^{k, t_c,h}_{v+1} \leq e^{C_N2^{-k}} \Gamma^{k, t_c,h}_v + C_N2^{-2\alpha k}.
\end{equation}
\end{cor}
From Corollary \ref{bounds_now_complete} and estimate \eqref{the_most_important_bound} we easily deduce:
\begin{thm}\label{crucial_theorem_1}
Let $t \in [0,T]$. There is a constant $C_N$ such that $\Delta^{k,t} \leq C_N 2^{(1-2\alpha) k}$.
\end{thm}
\begin{proof}
Now, this is just a standard application of Lemma \ref{discrete_Gronwall} in the spirit of Lemmas \ref{iteration_in_h} and \ref{iteration_in_k}. Let $l$ be such that $t \in (l \frac{T}{2^{k+1}}, (l+1) \frac{T}{2^{k+1}}]$. In view of \eqref{the_most_important_bound}, we have to estimate $\norm[1]{{\partial_H} \varphi_{\xi,t_c,l+1}^{H, k+1, k+1} - 
{\partial_H} \varphi_{\xi,t_c,l+1}^{H, k, k+1}}_{\infty}$. However, by Corollary \ref{bounds_now_complete}:
$$
\Gamma^{k, t_c,h}_{v+1} \leq e^{C_N2^{-k}} \Gamma^{k, t_c,h}_v + C_N 2^{-2\alpha k},
$$
for all $v \leq l$. Therefore, by Lemma \ref{discrete_Gronwall} (compare also with Remark \ref{practical_iteration_lemma}) $\norm[1]{{\partial_H} \varphi_{\xi,t_c,l+1}^{H, k+1, k+1} - 
{\partial_H} \varphi_{\xi,t_c,l+1}^{H, k, k+1}}_{\infty} \leq C_N2^{(1-2\alpha)k}$. Invoking \eqref{the_most_important_bound} concludes the proof.
\end{proof}
\section{Summary, discussion and future perspectives}\label{sect5}
In this paper, we proved that the map $[-\frac{1}{2}, \frac{1}{2}] \ni h \mapsto \mu_t^h$ is differentiable in $C([0,T], Z)$. Note that for linear problems \eqref{spm}, proof of differentiability is almost trivial. The only required step is to see that the solutions of \eqref{sol_dual_dep_onh} have bounded and H\"older continuous derivative (which follows from the Implicit Function Theorem) and then use \eqref{Cauchy_seq_analysis} together with Taylor's estimate from Lemma \ref{l1}. Actually, most of Section \ref{sect3} should be considered not as the solution of linear case but rather as a preparation for nonlinear problem. 

Nonlinear equation \eqref{spm_non} is analysed using approximating scheme from \cite{GwMa2010}. This approach has proven extremely successful in the development of well-posedness theory for \eqref{spm_non} where Authors were interested in the limit as $k \to \infty$. In our case, we somehow take limits $k \to \infty$ and $\Delta h \to 0$ simultaneously. To solve the problem, we first note that differentiability may follow from the classical result stating that uniform limit of continuous functions is continuous (Lemma \ref{unif_conv_cont_limit}). Therefore, we focus on two targets: differentiability of approximations $\mu_t^{h,k}$ and uniform (in $\Delta h$) convergence of difference quotients  $\frac{\mu_t^{h+\Delta h,k} - \mu_t^{h,k}}{\Delta h}$ when $k \to \infty$. Here we arrive at two crucial observations. Namely,
\begin{itemize}
\item{differentiability of $\mu_t^{h,k}$ can be deduced by induction. On the first interval it follows by small generalization of linear theory (Theorem \ref{lin_per_1}) while on a general interval of the form $[m\frac{T}{2^k}, (m+1)\frac{T}{2^k}]$ -- due to estimates in Section \ref{regularity_estimates_41} -- can be deduced from differentiability on $[(m-1)\frac{T}{2^k}, m\frac{T}{2^k}]$. Briefly speaking, if one knows that $h \mapsto \mu_t^{h,k}$ is differentiable, then the map $(x,h) \mapsto f(x,\mu_t^{h,k})$ has much more regularity than without this information. For the details, see the proof of Theorem \ref{diff_k_fixed}.}
\item{one can apply semigroup property \eqref{semigroup} sufficiently many times to remove the singular term $\frac{1}{\Delta h}$ from the expression $$\norm[3]{\frac{\mu_t^{h+\Delta h,k+1} - \mu_t^{h,k+1}}{\Delta h}-\frac{\mu_t^{h+\Delta h,k} - \mu_t^{h,k}}{\Delta h}}_{Z},$$ leading to estimate \eqref{the_most_important_bound}. There is a price to be paid for that -- to use this estimate we need to control differences of iterated solutions to integral equations like \eqref{sol_dual_dep_onh} with different flows (computations performed in Section \ref{perturbation_mdlfcns_section}). For the details, see the proof of Theorem \ref{crucial_theorem_1}.}
\end{itemize}
All these ideas could not be implemented without the simple, yet extremely powerful iteration lemma (Lemma \ref{discrete_Gronwall}). In particular, thanks to Remark \ref{practical_iteration_lemma}, application of this lemma basically comes down to obtaining a sufficiently large exponent $\beta$ in $T^{\beta}$ (time) in estimates for the linear model.

One can see some room for improvement in the condition $\alpha > \frac{1}{2}$. Indeed, on the one hand, Example \ref{flat_metric_is_not_enough} suggests that test functions defining space $Z$ should have at least uniformly continuous derivatives. On the other hand, if we assume dyadic approximation method as used in this paper and limiting procedure given by Lemma \ref{unif_conv_cont_limit}, condition $\alpha > \frac{1}{2}$ seems to be optimal. To see this, observe that when we write iteration inequality for the difference:
$$
\Gamma^{k, t_c, h}_v =  \sup_{\xi \in C^{1+\alpha}: \norm{\xi} \leq 1}\norm[2]{{\partial_h} \varphi_{\xi,t_c,v}^{h, k+1, k+1} - {\partial_h} \varphi_{\xi,t_c,v}^{h, k, k+1}}_{\infty}
$$ 
as in Section \ref{uniform_convergence_section}, we obtain as one of the summands:
$$
\norm[2]{{\partial_h} \varphi_{\xi,t_c,v-1}^{h, k, k+1}}_{\alpha,x} \norm[2]{X_b(s,\cdot) - X_{\bar{b}}(s,x)}_{\infty}^{\alpha},
$$ 
where $X_b$ and $X_{\bar{b}}$ corresponds to the flows on $k$--th and $(k+1)$--th approximation level respectively. Now, using Lemma \ref{perturbation_flow} from Appendix \ref{perturbation_mdlfcns_app}, also proven by triangle inequality, this difference can be bounded by $C_N 2^{-2\alpha k}$ so in view of Remark \ref{practical_iteration_lemma},  there is no hope for better estimate than $\Gamma^{k, t_c, h}_v \leq C_N 2^{(1-2\alpha) k}$. Actually, this shows that to prove Theorem \ref{crucial_theorem_1}, one has to demonstrate that all the other summands in iteration inequality for $\Gamma^{k, t_c, h}_v$ decay at least like $C_N 2^{-2\alpha k}$. In particular, one cannot argue that our sloppy computations involving many terms being incorporated to constants $C_L$ or $C_N$ leads to the condition $\alpha > \frac{1}{2}$ as in the end, only term $C_N 2^{-2\alpha k}$ matters. That's why, in our personal view, to improve this condition, one has to replace either dyadic approximation scheme or limiting procedure with a better strategy.

As already suggested in the introductory part of this paper, differentiability of map $h \mapsto \mu_t^h$ is the first step towards application of differential tools in optimization concerned with real--world phenomena governed by structured population models, especially in biology. For instance, consider a given family of distributions $\{\nu_t\}_{t \in [0,T]}$ indexed with time as well as model functions $a^0, a_p, b^0, b_p, c^0, c_p$. Can we choose $h$ such that $\mu_t^h = \nu_t$, i.e. can perturbations of model functions result in observed distribution $\nu_t$? Mathematically, this question corresponds to minimization of appropriate functional but biologically, it concerns completeness of theoretical model for some complex real--world systems. In particular, when it comes to application of differential algorithms like steepest descent, note that the standard assumption is $C^1$ regularity of the minimized function (cf. \cite{hinze2008optimization}, Chapter 2). In view of Theorem \ref{diff_k_fixed}, all approximations of derivatives are H\"older continuous in $h$ with constant $C_N 2^{k(1-\alpha)}$ so regularity condition is satisfied when $\alpha = 1$.

Moreover, our work suggests that space $Z$ is an appropriate setting for analysis of control problems mentioned above. This motivates further studies of properties of $Z$ -- in particular, in connection to our existence result. For instance, it is quite natural to ask whether obtained derivative satisfies appropriate PDE and in what sense. Furthermore, the derivative as an element of the space $Z$ has quite abstract nature. Therefore, it seems desirable to work on its interpretations, for instance in terms of changes in measure of a given set $A \subset \mathbb{R}^+$.

It is also worth mentioning that although our work covers great variety of structured population models, a lot of equations in applications are actually systems. Typical examples may be Kermack--McKendrick model in epidemiology, system for cell cycle in cell biology or Mackey--Rey model for hematopoiesis \cite{Pe2007}. Thus, another potential direction in research would be to extend well--posedness theory and differentiability results for this class of equations.

%List few crucial ideas from the proofs \\
%Comment here on estimation (loosing some terms - why they cannnot help) and condition $\alpha > \frac{1}{2}$. \\
%Comment here on fact that problem for linear equation is a simple consequence of IFT and properties of Holder functions (Taylor expansion...) - three lines proof.
\section*{Acknowledgements}
Author thanks Piotr Gwiazda for suggesting this topic as a subject of research, constant encouragment and valuable comments while preparing the final text. Author is supported by National Science Center, Poland through project no. 2017/27/B/ST1/01569. 
\bibliographystyle{abbrv}
\bibliography{measure_solutions_to_perturbed_spm} %, hyper_hd}
%hyper, hyper_hd, 
\begin{appendix}
\section{Extension of well-posedness theory for structured population models}\label{ewptspm}
\begin{proof}[Proof of Theorem \ref{well_posedness_2_nonlinear}]
Briefly speaking, the idea is that solutions to \eqref{spm} are {\it a priori} bounded in total variation norm so one can modify model functions for measures with big norm without changing the model. More preciesly, consider model functions $a$, $b$ and $c$ satisfying (W1a), (W1b), (W2) and (W3). Fix interval $[0,T]$ and observe that any measure solution $\mu_t: [0,T] \to \mathcal{M}^+(\mathbb{R}^+)$ satisfies weak formulation: for any $\varphi \in C^1([0,T] \times \mathbb{R}^+) \cap W^{1,\infty}([0,T] \times \mathbb{R}^+)$
\begin{multline}
\int_{\mathbb{R}^+} \varphi(T,x) d\mu_T - \int_{\mathbb{R}^+} \varphi(0,x) d\mu_0 = 
\int_0^T  \int_{\mathbb{R}^+} {\partial_t} \varphi(t,x) d\mu_t dt + \\ +
\int_0^T  \int_{\mathbb{R}^+} {\partial_x} \varphi(t,x) b(x,\mu_t) d\mu_t dt +  
\int_0^T  \int_{\mathbb{R}^+} \varphi(t,x) c(x,\mu_t) d\mu_t dt +  
\int_0^T  \int_{\mathbb{R}^+} \varphi(t,0) a(x,\mu_t) d\mu_t dt.
\end{multline}
Applying this with $\varphi(t,x) = 1$ we obtain:
$$
\int_{\mathbb{R}^+} d\mu_T - \int_{\mathbb{R}^+} d\mu_0 = \int_0^T  \int_{\mathbb{R}^+} c(x,\mu_t) d\mu_t dt +  
\int_0^T  \int_{\mathbb{R}^+} a(x,\mu_t) d\mu_t dt.
$$
which shows, after using Gronwall inequality, that any measure solution (in particular, a family of nonnegative measures) satisfies {\it a priori} the bound:
\begin{equation}\label{aprioriforsol}
\norm{\mu_t}_{TV} \leq \norm{\mu_0}_{TV} e^{(\norm{a}_{\infty} + \norm{c}_{\infty}) T} := \bar{C}.
\end{equation}
With this in mind, we define the new model functions:
\begin{equation}\label{example_fcn_a}
\bar{a}(x,\mu) =
\begin{cases} a(x,\mu) &\mbox{if } \norm{\mu}_{TV} \leq \bar{C},  \\ 
a(x,\mu)e^{-(\norm{\mu}_{TV} - C)} & \mbox{if } \norm{\mu}_{TV} > \bar{C} 
\end{cases}
\end{equation}
and similarly we define $\bar{b}(x,\mu)$ and $\bar{c}(x,\mu)$. We have to check that these new model functions satisfy assumptions (W1)--(W3) where (W3) is trivially satisfied and (W1) follows from exponential decay in definition \eqref{example_fcn_a}. 

Actually, only (W2) is not trivial. Let $\mu$, $\nu \in \mathcal{M}^+(\mathbb{R}^+)$. We have three cases. If $\norm{\mu}_{TV} \leq \bar{C}$ and $\norm{\nu}_{TV}\leq \bar{C}$ this follows from assumption (W2) for functions $a$, $b$ and $c$. If $\norm{\mu}_{TV} > \bar{C}$ and $\norm{\nu}_{TV} > \bar{C}$ we compute:
\begin{multline*}
\norm{\bar{a}(x,\mu) - \bar{a}(x,\nu)}_{\infty} = \norm{a(x,\mu)e^{-(\norm{\mu}_{TV} - \bar{C})} -
a(x,\nu)e^{-(\norm{\nu}_{TV} - \bar{C})}}_{\infty}\leq \\ \leq 
\norm{a(x,\mu)e^{-(\norm{\mu}_{TV} - \bar{C})} -
a(x,\mu)e^{-(\norm{\nu}_{TV} - \bar{C})}}_{\infty} + 
\norm{a(x,\mu)e^{-(\norm{\nu}_{TV} - \bar{C})} -
a(x,\nu)e^{-(\norm{\nu}_{TV} - \bar{C})}}_{\infty} \leq \\ \leq
C \norm{a}_{\infty}\Big|\norm{\mu}_{TV} - \norm{\nu}_{TV}\Big| + 
e^{C} L_R p_F(\mu,\nu) \leq (C \norm{a(x,\mu)}_{\infty} + e^{\bar{C}} L_R) p_F(\mu,\nu),
\end{multline*}
where we used that the map $(-\infty, \bar{C}] \ni x \mapsto e^x$ is Lipschitz continuous with constant $C$ and that $a$ satisfies assumption (W2). We finally check case $\norm{\mu}_{TV} > \bar{C}$ and $\norm{\nu}_{TV} \leq \bar{C}$:
\begin{multline*}
\norm{\bar{a}(x,\mu) - \bar{a}(x,\nu)}_{\infty} = \norm{a(x,\mu)e^{-(\norm{\mu}_{TV} - \bar{C})} -
a(x,\nu)}_{\infty}\leq \\ \leq 
\norm{a(x,\mu)e^{-(\norm{\mu}_{TV} - \bar{C})} - a(x,\mu)}_{\infty} +\norm{a(x,\mu) - a(x,\nu)}_{\infty} \leq
\norm{a(x,\mu)}_{\infty}\abs{e^{-(\norm{\mu}_{TV} - \bar{C})} - 1} +  L_R p_F(\mu,\nu) \leq \\
\leq \norm{a(x,\mu)}_{\infty} \abs{e^{-(\norm{\mu}_{TV} - \bar{C})} - e^{-(\norm{\nu}_{TV} - \bar{C})}}
+  L_R p_F(\mu,\nu)  \leq (C + L_R)p_F(\mu,\nu)
\end{multline*}
since $\norm{\nu}_{TV} \leq \bar{C}$ implies $e^{-(\norm{\nu}_{TV} - \bar{C})} \geq 1$. 

Therefore, there exists the unique solution to the problem \eqref{spm_non} with the model functions $\bar{a}$, $\bar{b}$, $\bar{c}$. However, this solution satisfies the bound \eqref{aprioriforsol} so it also solves \eqref{spm_non} (functions $a$, $b$ and $c$ were not changed for measures $\mu$ with $\norm{\mu} \leq \bar{C}$. From this, the uniqueness follows too. 
\end{proof}
\section{Technical computations from Section \ref{taking_particular_function}}\label{taking_particular_function_comp}
We begin with studying each term appearing in the equation \eqref{sol_dual_dep_onh} starting slightly more generally:
\begin{lem}\label{pierwsza_funkcja}
Let $f(h,u,x) = \xi(h,X_{b(h, \cdot)}(u,x))$ with $\xi \in C^{1+\alpha}([-\frac{1}{2}, \frac{1}{2}] \times \mathbb{R})$ and $u \in [0,T]$. Then, map $h \mapsto f(h,u,x)$ is $C^{1+\alpha}$ with:
\begin{itemize}
\item $\norm{f}_{\infty}\leq \norm{\xi}_{\infty}$,
\item $\norm{f_{h}}_{\infty} \leq \norm{\xi_h}_{\infty} + C_L T\norm{\xi_x}_{\infty}$,
\item $\norm{f_{x}}_{\infty}  \leq e^{C_LT} \norm{\xi_x}_{\infty} $,
\item $\norm{f}_{W^{1,\infty}}  \leq e^{C_LT}\norm{\xi}_{W^{1,\infty}} $,
\item $\norm{f_{h}}_{\alpha,h} \leq \norm{\xi_h}_{\alpha,h} + C_{L} T^{\alpha} \norm{\xi_h}_{\alpha,x} + C_{L} T \norm{\xi_x}_{\alpha,h} + C_{L} T^{1+\alpha} \norm{\xi_x}_{\alpha,x} + G_{L}T\norm{\xi}_{W^{1,\infty}}$,
\item $\norm{f_{x}}_{\alpha,x} \leq e^{C_LT} \norm{\xi_x}_{\alpha,x}  + H_L T \norm{\xi}_{W^{1,\infty}}  $,
\item $\norm{f_{h}}_{\alpha,x} \leq e^{C_LT} \norm{\xi_h}_{\alpha,x}  + C_L T \norm{\xi_x}_{\alpha,x}  + H_L T \norm{\xi}_{W^{1,\infty}} $,
\item $\norm{f_{x}}_{\alpha,h} \leq \norm{\xi_x}_{\alpha,h} + C_L T^{\alpha}\norm{\xi_x}_{\alpha,x}  + H_L T \norm{\xi}_{W^{1,\infty}} $.
\end{itemize}
\end{lem}
\begin{proof}
It is clear that $\norm{f}_{\infty} \leq \norm{\xi}_{\infty} $. Moreover, explicit computation of the derivatives yields:
$$
f_h(h,u,x) = \xi_h(h,X_{b(h, \cdot)}(u,x)) + \xi_x(h,X_{b(h, \cdot)}(u,x)) {\partial_h} X_{b(h, \cdot)}(u,x),
$$
$$
f_x(h,u,x) = \xi_x(h,X_{b(h, \cdot)}(u,x)) {\partial_x} X_{b(h, \cdot)}(u,x).
$$
Therefore, due to Corollary \ref{summary_flow}, $\norm{f_{h}}_{\infty} \leq \norm{\xi_h}_{\infty} + C_LT\norm{\xi_x}_{\infty} $ and $\norm{f_{x}}_{\infty} \leq e^{C_LT} \norm{\xi_x}_{\infty} $. Then, we compute explicitly recalling convention from Remark \ref{estimate_convention} and Corollary \ref{summary_flow}:
\begin{multline*}
\abs{f_h(h_1,u,x) - f_h(h_2,u,x)} \leq \norm{\xi_h}_{\alpha,h}\abs{h_1-h_2}^{\alpha} + \norm{\xi_h}_{\alpha,x} (C_LT)^{\alpha}\abs{h_1-h_2}^{\alpha}
 +\\+
\big(\norm{\xi_x}_{\alpha,h}\abs{h_1-h_2}^{\alpha} + \norm{\xi_x}_{\alpha,x}~ (C_LT)^{\alpha}\abs{h_1-h_2}^{\alpha}) C_LT
 + \norm{\xi_x}_{\infty}G_LT\abs{h_1-h_2}^{\alpha}
\end{multline*}
and
\begin{equation*}
\abs{f_x(h,x_1,u) - f_x(h,x_2,u)} \leq \big(\norm{\xi_x}_{\alpha,x}e^{\alpha C_LT}\abs{x_1-x_2}^{\alpha}\big)e^{C_LT} + \norm{\xi_x}_{\infty}~TH_L\abs{x_1-x_2}^{\alpha}.
\end{equation*}
Similarly,
\begin{multline*}
\abs{f_h(h,x_1,u) - f_h(h,x_2,u)} \leq \norm{\xi_h}_{\alpha,x} e^{\alpha C_LT}\abs{x_1 - x_2}^{\alpha}+ \norm{\xi_x}_{\alpha,x} e^{\alpha C_LT}\abs{x_1 - x_2}^{\alpha}C_LT+\norm{\xi_x}_{\infty}T H_L ,
\end{multline*}
\begin{equation*}
\abs{f_x(h_1,u,x) - f_x(h_2,u,x)} \leq \norm{\xi_x}_{\alpha,h}\abs{h_1 - h_2}^{\alpha}+\norm{\xi_x}_{\alpha,x}\big( C_LT\abs{h_1 - h_2}\big)^{\alpha}+\norm{\xi_x }_{\infty} T H_L\abs{h_1 - h_2}^{\alpha}.
\end{equation*}
\end{proof}
From Lemma \ref{pierwsza_funkcja}, we obtain bunch of Corollaries:
\begin{cor}\label{druga_funkcja}
Let $f(h,u,x) = \xi(X_{b(h, \cdot)}(u,x))$ with $\norm{\xi}_{C^{1+\alpha}(\mathbb{R}^+)} \leq 1$. Then, using Lemma \ref{pierwsza_funkcja}, we conclude that map $h \mapsto f(h,u,x)$ is $C^{1+\alpha}([-\frac{1}{2}, \frac{1}{2}]\times \mathbb{R}^+)$ with $\norm{f}_{\infty}\leq 1$, $\norm{f_{h}}_{\infty} \leq C_LT$ and $\norm{f_{h}}_{\alpha,h} \leq  G_LT$.
\end{cor}
\begin{cor}\label{trzecia_funkcja}
Let $f(h,u,x) = a(h, X_{b(h, \cdot)}(u,x))$ where $a(\cdot,\cdot)$ is a model function. Then, using Lemma \ref{pierwsza_funkcja} we conclude that $\norm{f}_{\infty}, \norm{f_{h}}_{\infty}, \norm{f_{x}}_{\infty} \leq C_L$, $\norm{f_{h}}_{\alpha,h} \leq G_L$ and $\norm{f_{x}}_{\alpha,h}, \norm{f_{h}}_{\alpha,x}, \norm{f_{x}}_{\alpha,x} \leq H_L$.
\end{cor}
We move to the exponential part of the implicit formula. Recall that the map $[-M, M] \ni x \mapsto \exp(x)$ is Lipschitz continuous with constant $\exp(M)$. 
\begin{lem}\label{czwarta_funkcja}
Let $f(h,u,x) = e^{\int_0^u c(h, X_{b(h, \cdot)}(v,x)) dv}$ with $u\in [0,T]$. Then map $h \mapsto f(h,u,x)$ satisfies $\norm{f}_{\infty}\leq e^{C_LT}$, $\norm{f_{h}}_{\infty}, \norm{f_x}_{\infty} \leq  C_{L}T $, $\norm{f_{h}}_{\alpha,h} \leq G_LT$ and $\norm{f_{x}}_{\alpha,h}, \norm{f_{h}}_{\alpha,x}, \norm{f_{x}}_{\alpha,x} \leq  H_LT$.
\end{lem}
\begin{proof}
Of course, $\norm{f}_{\infty}\leq e^{C_LT}$. By direct computation:
\begin{equation*}
f_h(h,u,x) = e^{\int_0^u c(h, X_{b(h, \cdot)}(v,x)) dv}\int_0^u \Big(c_h(h, X_{b(h, \cdot)}(v,x))+ 
c_x(h, X_{b(h, \cdot)}(v,x)) ~{\partial_h}X_{b(h, \cdot)}(v,x) \Big)dv
,\end{equation*}
so that $\norm{f_{h}}_{\infty}\leq  C_{L}T$. Moreover, using Remark \ref{estimate_convention} and Corollary \ref{summary_flow}:
\begin{equation*}
\abs{f_h(h_1,u,x) - f_h(h_2,u,x)} \leq T^2 C_L\abs{h_1-h_2}  + T(H_L~+~ G_L)\abs{h_1 - h_2}^{\alpha} \leq G_L T  \abs{h_1 - h_2}^{\alpha}.
\end{equation*}
Furthermore, if $\abs{x_1-x_2} \leq 1$:
\begin{equation*}
\abs{f_h(h,x_1,u) - f_h(h,x_2,u)} \leq C_LT~\abs{x_1-x_2} +
 T H_L~\abs{x_1 - x_2}^{\alpha} \leq H_L T  \abs{x_1 - x_2}^{\alpha}
\end{equation*}
and otherwise,
$$
\abs{f_h(h,x_1,u) - f_h(h,x_2,u)} \leq 2 \norm{f_h}_{\infty} \leq 2 \norm{f_h}_{\infty} {\abs{x_1 - x_2}^{\alpha}} \leq  C_L T {\abs{x_1 - x_2}^{\alpha}} \leq H_L T {\abs{x_1 - x_2}^{\alpha}}.
$$
Similarly, using chain rule:
$$
f_x(h,u,x) = e^{\int_0^u c(h, X_{b(h, \cdot)}(v,x)) dv} \int_0^u c_x(h, X_{b(h, \cdot)}(v,x))  ~{\partial_x} X_{b(h, \cdot)}(v,x)  dv, 
$$
so that we compute all bounds exactly like for $f_h$. 
\end{proof}
We move to draw conclusions about functions $P^{a,b,c}(h,s,x)$ given by \eqref{def_of_P} and $Q^{a,b,c}(h,s,x)$ given by \eqref{def_of_Q}. We start more generally:
\begin{lem}\label{product_rule}
Consider function $f(x,h) = k(x,h) l(x,h)$ defined on $\mathbb{R^+}\times [-\frac{1}{2}, \frac{1}{2}]$. Then:
\begin{itemize}
\item $\norm{f}_{\infty} \leq \norm{k}_{\infty} \norm{l}_{\infty}$,
\item $\norm{f_x}_{\infty} \leq \norm{k_x}_{\infty}  \norm{l}_{\infty} +  \norm{k}_{\infty} \norm{l_x}_{\infty}$,
\item $\norm{f_h}_{\infty} \leq \norm{k_h}_{\infty}  \norm{l}_{\infty} +  \norm{k}_{\infty} \norm{l_h}_{\infty}$,
\item $\norm{f_h}_{\alpha,h} \leq  2\norm{k_h}_{\infty} \norm{l_h}_{\infty}+ 
\norm{k}_{\infty}\norm{l_h}_{\alpha,h} + 
\norm{k_h}_{h,\alpha}\norm{l}_{\infty}$,
\item $\norm{f_h}_{\alpha,x} \leq \max{\big(2\norm{f_h}_{\infty},\norm{k_h}_{\alpha,x}\norm{l}_{\infty} + \norm{k_h}_{\infty} \norm{l_x}_{\infty} + \norm{k_x}_{\infty} \norm{l_h}_{\infty}+\norm{l_h}_{\alpha,x} \norm{k}_{\infty}\big)} $,
\item $\norm{f_x}_{\alpha,h} \leq \norm{k_x}_{\alpha,h}\norm{l}_{\infty} + \norm{k_x}_{\infty} \norm{l_h}_{\infty} + \norm{k_h}_{\infty} \norm{l_x}_{\infty}+\norm{l_x}_{\alpha,h} \norm{k}_{\infty}$,
\item $\norm{f_x}_{\alpha,x} \leq \max{\big(2\norm{f_x}_{\infty}, 2\norm{k_x}_{\infty} \norm{l_x}_{\infty}+ 
\norm{k}_{\infty}\norm{l_x}_{\alpha,x} + 
\norm{k_x}_{x,\alpha}\norm{l}_{\infty}\big)}$.
\end{itemize}
\end{lem}
\begin{proof}
First three bounds are obvious. By chain rule, $f_x = k_x l + k l_x$. Therefore,
\begin{multline*}
\abs{f_x(x_1,h) - f_x(x_2,h)} \leq \abs{k_x(x_1,h) l(x_1,h) - k_x(x_2,h)l(x_2,h)} + 
\abs{k(x_1,h) l_x(x_1,h) - k(x_2,h) l_x(x_2,h)} \leq \\
\norm{k_x}_{x,\alpha}\norm{l}_{\infty}\abs{x_1-x_2}^{\alpha} + \norm{l_x}_{\infty}\norm{k_x}_{\infty}\abs{x_1-x_2} + 
\norm{k_x}_{\infty} \norm{l_x}_{\infty}\abs{x_1-x_2} + 
\norm{k}_{\infty}\norm{l_x}_{\alpha,x} \abs{x_1-x_2}^{\alpha}
\end{multline*}
so that $\norm{f_x}_{\alpha,x} \leq \max{\big(2\norm{f_x}_{\infty}, 2\norm{k_x}_{\infty} \norm{l_x}_{\infty}+ 
\norm{k}_{\infty}\norm{l_x}_{\alpha,x} + 
\norm{k_x}_{x,\alpha}\norm{l}_{\infty}\big)}$ and bound on $\norm{f_h}_{\alpha,h}$ can be established by symmetry (there is no maximum as domain is bounded here). Similarly,
\begin{multline*}
%f_h = k_h l + k l_h
\abs{f_h(x_1,h) - f_h(x_2,h)} \leq \abs{k_h(x_1,h)l(x_1,h) - k_h(x_2,h)l(x_2,h)} 
+ \abs{k(x_1,h)l_h(x_1,h) - k(x_1,h)l_h(x_1,h)} \leq \\
\norm{k_h}_{\alpha,x}\norm{l}_{\infty}\abs{x_1-x_2}^{\alpha} +
\norm{k_h}_{\infty} \norm{l_x}_{\infty}\abs{x_1-x_2} + 
\norm{k_x}_{\infty} \norm{l_h}_{\infty}\abs{x_1-x_2} + 
\norm{l_h}_{\alpha,x} \norm{k}_{\infty} \abs{x_1-x_2}^{\alpha}
\end{multline*}
so that $\norm{f_h}_{\alpha,x} \leq \max{\big(2\norm{f_h}_{\infty},\norm{k_h}_{\alpha,x}\norm{l}_{\infty} + \norm{k_h}_{\infty} \norm{l_x}_{\infty} + \norm{k_x}_{\infty} \norm{l_h}_{\infty}+\norm{l_h}_{\alpha,x} \norm{k}_{\infty}\big)}$ and bound for $\norm{f_x}_{\alpha,h}$ can be established again by symmetry. 
\end{proof}
With these estimates in hand, it is easy to prove Corollary \ref{Q_est}, Corollary \ref{P_est} and Lemma \ref{est_dual_sol_fin_fcns}:
\begin{proof}[Proof of Corollary \ref{Q_est}]
We use directly Lemma \ref{product_rule} together with the estimates provided by Corollary \ref{trzecia_funkcja} and Lemma \ref{czwarta_funkcja}. As form of bounds to be used is quite simple, we proceed quickly. Clearly, $\norm{Q^{a,b,c}}_{\infty}$, $\norm{Q^{a,b,c}_x}_{\infty}$, $\norm{Q^{a,b,c}_h}_{\infty} \leq C_L$ is obvious as estimates for these quantities do not contain H\"older constants. To verify next inequalities, observe, that using notation of Lemma \ref{product_rule}, constant $G_L$ can only appear from expressions containing $\norm{k_h}_{\alpha,h}$ or $\norm{l_h}_{\alpha,h}$; all other H\"older constants are bounded by $H_L$.
\end{proof}
\begin{proof}[Proof of Corollary \ref{P_est}]
We use directly Lemma \ref{product_rule} together with estimates provided by Lemma \ref{pierwsza_funkcja} and Lemma \ref{czwarta_funkcja}. We want to trace dependence on norm of function $\xi$ so we proceed more carefully. Let
$$
k(h,s,x) = \xi(h, X_{b(h, \cdot)}(t-s,x)), \quad \quad \quad l(h,s,x) = e^{\int_0^{t-s }c(h, X_{b(h,\cdot)}(u,x)) du},
$$
(we do not worry about dependence on $s \in [0,T]$ as it is always fixed). Clearly,  $\norm[1]{P^{a,b,c}}_{\infty} \leq e^{C_LT} \norm[1]{\xi}_{\infty} $. Moreover, directly applying Lemma \ref{product_rule}:
\begin{multline*}
\norm[1]{P^{a,b,c}_x}_{\infty} \leq \norm{k_x}_{\infty}  \norm{l}_{\infty} +  \norm{k}_{\infty} \norm{l_x}_{\infty} \leq \underbrace{\norm{\xi}_{W^{1,\infty}} e^{C_LT}}_{\norm{k_x}_{\infty}~\leq}
e^{C_LT} + \norm{\xi}_{\infty}\underbrace{C_LT}_{ \norm{l_x}_{\infty} ~\leq} \leq \\ \leq \norm{\xi}_{W^{1,\infty}} e^{C_LT} + C_LT\norm{\xi}_{\infty} \leq e^{C_LT} \norm{\xi}_{W^{1,\infty}} .
\end{multline*}
Similarly,
\begin{equation*}
\norm[1]{P^{a,b,c}_h}_{\infty}\leq \norm{k_h}_{\infty}  \norm{l}_{\infty} +  \norm{k}_{\infty} \norm{l_h}_{\infty} \leq \underbrace{\norm{\xi_h}_{\infty} + C_LT \norm{\xi_x}_{\infty}}_{\norm{k_h}_{\infty}~\leq} e^{C_LT} + \norm{\xi}_{\infty}\underbrace{C_LT}_{ \norm{l_h}_{\infty}~\leq }\leq e^{C_LT}\norm{\xi}_{W^{1,\infty}} .
\end{equation*} 
Then, we estimate H\"older constants in variable $h$: 
\begin{multline*}
\norm[1]{P^{a,b,c}_h}_{\alpha,h} \leq 2\norm{k_h}_{\infty} \norm{l_h}_{\infty}+ 
\norm{k}_{\infty}\norm{l_h}_{\alpha,h} + 
\norm{k_h}_{h,\alpha}\norm{l}_{\infty} \leq 2\underbrace{\norm{\xi}_{W^{1,\infty}}e^{C_LT}}_{\norm{k_h}_{\infty}~\leq}
 \underbrace{C_LT}_{\norm{l_h}_{\infty}~\leq}
+ \\+ \norm{\xi}_{\infty} \underbrace{G_LT}_{\norm{l_h}_{\alpha,h}~\leq} 
+ \underbrace{\big(\norm{\xi_h}_{\alpha,h} + C_{L}T^{\alpha} \norm{\xi_h}_{\alpha,x} + C_{L}T \norm{\xi_x}_{\alpha,h} + C_{L}T^{1+\alpha}\norm{\xi_x}_{\alpha,x} + G_{L}T\norm{\xi_x}_{\infty}\big)}_{\norm{k_h}_{\alpha,h}~\leq}e^{C_LT} \leq \\ \leq
G_L T\norm{\xi}_{W^{1,\infty}} +e^{C_LT} \norm{\xi_h}_{\alpha,h} +  C_{L} T^{\alpha} \norm{\xi_h}_{\alpha,x} + C_LT \norm{\xi_x}_{\alpha,h} + C_LT \norm{\xi_x}_{\alpha,x}.
\end{multline*}
To estimate $\norm[1]{P^{a,b,c}_h}_{\alpha,x}$, it is sufficient to study the right part of maximum:
\begin{multline*}
\norm{k_h}_{\alpha,x}\norm{l}_{\infty} + \norm{k_h}_{\infty} \norm{l_x}_{\infty} + \norm{k_x}_{\infty} \norm{l_h}_{\infty}+\norm{l_h}_{\alpha,x} \norm{k}_{\infty} \leq \\ \leq
\norm{k_h}_{\alpha,x}e^{C_LT} + \underbrace{\norm{\xi}_{W^{1,\infty}}e^{C_LT}}_{\norm{k_h}_{\infty}~\leq} \underbrace{C_LT}_{\norm{l_x}_{\infty}~\leq} + \underbrace{\norm{\xi}_{W^{1,\infty}}e^{C_LT}}_{\norm{k_x}_{\infty}~\leq} \underbrace{C_LT}_{\norm{l_h}_{\infty}~\leq}+\underbrace{H_LT}_{\norm{l_h}_{\alpha,x}~\leq} \norm{\xi}_{\infty} \leq \\
\leq \underbrace{\big(\norm{\xi_h}_{\alpha,x} e^{C_LT} + C_LT\norm{\xi_x}_{\alpha,x}  +  H_LT\norm{\xi}_{W^{1,\infty}}\big)}_{\norm{k_h}_{\alpha,x}~\leq} e^{C_LT}+ H_L T\norm{\xi}_{W^{1,\infty}} \leq \\
\leq  e^{C_LT} \norm{\xi_h}_{\alpha,x} + C_LT \norm{\xi_x}_{\alpha,x} + H_LT \norm{\xi}_{W^{1,\infty}}.
\end{multline*}
Then, we study H\"older constants in variable $x$:
\begin{multline*}
\norm[1]{P^{a,b,c}_x}_{\alpha,h} \leq \norm{k_x}_{\alpha,h}\norm{l}_{\infty} + \norm{k_x}_{\infty} \norm{l_h}_{\infty} + \norm{k_h}_{\infty} \norm{l_x}_{\infty}+\norm{l_x}_{\alpha,h} \norm{k}_{\infty} \leq \\ \leq \norm{k_x}_{\alpha,h} e^{C_LT} + \underbrace{\norm{\xi}_{W^{1,\infty}} e^{C_LT}}_{\norm{k_x}_{\infty}~\leq}\underbrace{C_LT}_{\norm{l_h}_{\infty}~\leq} +  \underbrace{\norm{\xi}_{W^{1,\infty}} e^{C_LT}}_{\norm{k_h}_{\infty}~\leq}\underbrace{C_LT}_{\norm{l_x}_{\infty}~\leq} + \underbrace{H_LT}_{\norm{l_x}_{\alpha,h}~\leq} \norm{\xi}_{\infty} \leq \\ \leq
\underbrace{\big(\norm{\xi_x}_{\alpha,h} + C_LT^{\alpha}\norm{\xi_x}_{\alpha,x}  + H_LT \norm{\xi}_{W^{1,\infty}} \big)}_{ \norm{k_x}_{\alpha,h}~\leq}e^{C_LT} + H_LT\norm{\xi}_{W^{1,\infty}} \leq \\ \leq
e^{C_LT} \norm{\xi_x}_{\alpha,h} + C_LT^{\alpha}\norm{\xi_x}_{\alpha,x}  + H_LT\norm{\xi}_{W^{1,\infty}}.
\end{multline*}
Finally, similarly as before, to estimate $\norm[1]{P^{a,b,c}_x}_{\alpha,x}$ we first study the right part of maximum:
\begin{multline*}
2\norm{k_x}_{\infty} \norm{l_x}_{\infty}+ 
\norm{k}_{\infty}\norm{l_x}_{\alpha,x} + 
\norm{k_x}_{x,\alpha}\norm{l}_{\infty} \leq \\ \leq
2\underbrace{\norm{\xi}_{W^{1,\infty}}e^{C_LT}}_{\norm{k_x}_{\infty} ~\leq}\underbrace{C_LT}_{\norm{l_x}_{\infty}~\leq} + \norm{\xi}_{\infty}\underbrace{H_LT}_{\norm{l_x}_{\alpha,x}~\leq} + \underbrace{\big(\norm{\xi_x}_{\alpha,x} e^{C_LT} +H_LT \norm{\xi}_{W^{1,\infty}} \big)}_{\norm{k_x}_{x,\alpha}~\leq} e^{C_LT} \leq \\ \leq
e^{C_LT} \norm{\xi_x}_{\alpha,x}  + H_LT \norm{\xi}_{W^{1,\infty}} 
\end{multline*}
concluding the proof. 
\end{proof}
\begin{proof}[Proof of Lemma \ref{est_dual_sol_fin_fcns}]
We apply directly Corollary \ref{estimates_for_solutions} together with estimates for $P^{a,b,c}$ provided by Corollary \ref{P_est} and estimates for $Q^{a,b,c}$ provided by Corollary \ref{Q_est}. First, note that due to these bounds, the costants $C_q, H_q, G_q$ appearing in the estimates in Corollary \ref{estimates_for_solutions} are bounded by $C_L, H_L$ and $G_L$ respectively. Therefore,
\begin{multline*}
\norm{\varphi_{\xi,h}}_{\infty} \leq \norm{P_h}_{\infty}e^{C_LT} + C_LT \norm{P}_{\infty} \leq \underbrace{\norm{\xi_h}_{\infty} + C_LT \big(\norm{\xi_x}_{\infty} + \norm{\xi}_{\infty}\big)}_{\norm{P_h}_{\infty} ~\leq}+ C_LT \underbrace{\norm{\xi}_{\infty}e^{C_LT}}_{\norm{P}_{\infty} ~\leq}\leq \\ \leq \norm{\xi_h}_{\infty} + C_LT \big(\norm{\xi_x}_{\infty} + \norm{\xi}_{\infty}\big) 
\end{multline*}
Clearly $\norm{\varphi_{\xi}}_{W^{1,\infty}} \leq e^{C_LT} \norm{P}_{W^{1,\infty}} \leq e^{C_LT} \norm{\xi}_{W^{1,\infty}} $. Then, we move to study H\"older bounds. To make these messy computations clear, we always start by rewriting estimates from Corollary \ref{estimates_for_solutions}:
\begin{multline*}
\norm{\varphi_{\xi,h}}_{\alpha, h} \leq \norm{P_h}_{\alpha,h}e^{C_LT}
+ G_L T \norm{\xi}_{W^{1,\infty}} \leq \\ \leq 
\underbrace{\big(G_LT\norm{\xi}_{W^{1,\infty}} +e^{C_LT}\norm{\xi_h}_{\alpha,h} +  C_{L}T^{\alpha} \norm{\xi_h}_{\alpha,x} + C_LT\norm{\xi_x}_{\alpha,h}+ C_LT\norm{\xi_x}_{\alpha,x}\big)}_{\norm{P_h}_{\alpha,h}~\leq}e^{C_LT} + G_LT\norm{\xi}_{W^{1,\infty}}  \\
\leq
e^{C_LT} \norm{\xi_h}_{\alpha,h} + G_L T\norm{\xi}_{W^{1,\infty}} + C_{L} T^{\alpha} \norm{\xi_h}_{\alpha,x} + C_L T\norm{\xi_x}_{\alpha,h}+ C_L T\norm{\xi_x}_{\alpha,x},
\end{multline*}
\begin{multline*}
\norm{\varphi_{\xi,x}}_{\alpha,x} \leq \norm{P_x}_{\alpha,x} + H_LT \norm{P}_{W^{1,\infty}} \leq \\ \leq
\underbrace{\max{\big( e^{C_LT}\norm{\xi}_{W^{1,\infty}} , ~ e^{C_LT}\norm{\xi_x}_{\alpha,x} + H_LT \norm{\xi}_{W^{1,\infty}}\big)}}_{\norm{P_x}_{\alpha,x}~\leq} + H_LT  \norm{\xi}_{W^{1,\infty}} \leq \\ \leq
e^{C_LT} \max{\big(\norm{\xi}_{W^{1,\infty}} , ~\norm{\xi_x}_{\alpha,x} \big)}  + H_L T  \norm{\xi}_{W^{1,\infty}},
\end{multline*}
\begin{multline*}
\norm{\varphi_{\xi,x}}_{\alpha,h} \leq \norm{P_x}_{\alpha,h} + H_LT\norm{P}_{W^{1,\infty}} \leq \\ \leq
\underbrace{\big(e^{C_LT}\norm{\xi_x}_{\alpha,h} + C_LT^{\alpha}\norm{\xi_x}_{\alpha,x}  + H_LT\norm{\xi}_{W^{1,\infty}} \big)}_{\norm{P_x}_{\alpha,h} ~\leq} +  H_L T\norm{\xi}_{W^{1,\infty}}  \leq \\ \leq e^{C_LT} \norm{\xi_x}_{\alpha,h} + C_L T^{\alpha}\norm{\xi_x}_{\alpha,x}  + H_L T\norm{\xi}_{W^{1,\infty}},
\end{multline*}
\begin{multline*}
\norm{\varphi_{\xi,h}}_{\alpha,x} \leq \max\Big(2e^{C_L T} \norm{P}_{W^{1,\infty}} , \norm{P_h}_{\alpha,x} + H_LT \norm{P}_{W^{1,\infty}} \Big) \leq \\ \leq
\max\Big(2e^{C_L T} \underbrace{\norm{\xi}_{W^{1,\infty}} e^{C_L T}}_{\norm{P}_{W^{1,\infty}} ~\leq} , \norm{P_h}_{\alpha,x} +H_L T \norm{\xi}_{W^{1,\infty}} \Big) \leq \\ \leq
\max\big(2e^{C_L T} \norm{\xi}_{W^{1,\infty}}  , ~\underbrace{ e^{C_LT}\norm{\xi_h}_{\alpha,x} + C_LT\norm{\xi_x}_{\alpha,x} + H_LT \norm{\xi}_{W^{1,\infty}}}_{\text{right part of maximum in estimate for }\norm{P_h}_{\alpha,x}\text{ in Corollary \ref{P_est}}} + H_LT \norm{\xi}_{W^{1,\infty}} \big) \leq \\ \leq
\max\Big(2e^{C_L T} \norm{\xi}_{W^{1,\infty}} , e^{C_L T}\norm{\xi_h}_{\alpha,x} + H_LT \norm{\xi}_{W^{1,\infty}} \Big) \leq \\ \leq
e^{C_L T} \max\big(2 \norm{\xi}_{W^{1,\infty}} , ~\norm{\xi_h}_{\alpha,x}\big) +  C_L T \norm{\xi_x}_{\alpha,x} + H_L T \norm{\xi}_{W^{1,\infty}}.
\end{multline*}
\end{proof}
Finally, we prove technical Lemma \ref{kicked_continuity}.
\begin{proof}[Proof of Lemma \ref{kicked_continuity}]
Note that integral part of the implicit equation \eqref{sol_dual_dep_onh} is always bounded in $W^{1,\infty}(\mathbb{R}^+)$ by $C_L \abs{t}$ due to presence of integral and boundedness of all components (in particular, implicit term is bounded due to Lemma \ref{est_dual_sol_fin_fcns}). Therefore, it is sufficient to estimate $\norm[1]{\xi(X_{b(h, \cdot)}(t-s,x))e^{\int_0^{t-s }c(h, X_{b(h,\cdot)}(u,x)) du} - \xi(x)}_{W^{1,\infty}(\mathbb{R}^+)}$. We compute using triangle inequality and standard Lipschitz estimates:
\begin{multline*}
\abs{\xi(X_{b(h, \cdot)}(t-s,x))e^{\int_0^{t-s }c(h, X_{b(h,\cdot)}(u,x)) du} - \xi(x)} \leq \\ \leq e^{C_Lt}\abs{\xi(X_{b(h, \cdot)}(t-s,x)) - \xi(x)} + \norm{\xi}_{\infty} \abs{e^{\int_0^{t-s }c(h, X_{b(h,\cdot)}(u,x)) du} - 1} \leq \\ \leq
e^{C_Lt} \norm{\xi_x}_{\infty} \abs{X_{b(h, \cdot)}(t-s,x) - X_{b(h, \cdot)}(0,x)} + e^{C_Lt} \int_0^{t-s }\abs[1]{c(h, X_{b(h,\cdot)}(u,x))} du \leq \\ \leq
e^{C_Lt} \norm{\xi_x}_{\infty}\int_0^{t-s} \abs[1]{b(h, X_{b(h,\cdot)}(u,x))}du +e^{C_Lt} \int_0^{t-s}\abs[1]{c(h, X_{b(h,\cdot)}(u,x))} du \leq C(C_L, \norm[1]{\xi}_{W^{1,\infty}})\abs{t}.
\end{multline*}
In particular, $\abs{X_{b(h, \cdot)}(t-s,x) - x}\leq C_L\abs{t}$ and $\abs[1]{e^{\int_0^{t-s }c(h, X_{b(h,\cdot)}(u,x)) du} - 1} \leq C_L\abs{t}$. Then, we estimate difference of derivatives:
\begin{multline*}
\abs[2]{{\partial_x}\Big(\xi(X_{b(h, \cdot)}(t-s,x))~e^{\int_0^{t-s }c(h, X_{b(h,\cdot)}(u,x)) du}\Big) - \xi_x(x)} \leq \\ \leq
\abs[2]{{\partial_x}\Big(\xi(X_{b(h, \cdot)}(t-s,x))\Big)~e^{\int_0^{t-s }c(h, X_{b(h,\cdot)}(u,x)) du} - \xi_x(x)} + \abs{\xi(X_{b(h, \cdot)}(t-s,x))~{\partial_x}e^{\int_0^{t-s }c(h, X_{b(h,\cdot)}(u,x)) du}}.
\end{multline*}
The second term is trivially bounded by $C(C_L, \norm[1]{\xi}_{\infty})\abs{t}$. We focus on the first one. Since exponential function is Lipschitz on bounded sets, we can add and subtract term $\xi_x(x)e^{\int_0^{t-s }c(h, X_{b(h,\cdot)}(u,x)) du}$. Then, it is sufficient to bound:
\begin{multline*}
\abs[2]{{\partial_x}\big(\xi(X_{b(h, \cdot)}(t-s,x))\big) - \xi_x(x)} = \abs[2]{\xi_x(X_{b(h, \cdot)}(t-s,x))\big({\partial_x}X_{b(h, \cdot)}(t-s,x)\big) - \xi_x(x)} \leq \\ \leq
\underbrace{\abs[2]{\big(\xi_x(X_{b(h, \cdot)}(t-s,x)) - \xi_x(x)\big)}}_{\leq~\norm[1]{\xi_x}_{\alpha,x}~\abs[1]{X_{b(h, \cdot)}(t-s,x) - X_{b(h, \cdot)}(0,x)}^{\alpha}}\underbrace{\big({\partial_x}X_{b(h, \cdot)}(t-s,x)\big)}_{\leq e^{C_L t} \text{ due to Lemma \ref{l0_wx}}} +
\abs[2]{\xi_x(x)\Big({\partial_x}X_{b(h, \cdot)}(t-s,x)-1\Big)} .  
\end{multline*}
Therefore, it is sufficient to estimate $\abs[2]{{\partial_x}X_{b(h, \cdot)}(t-s,x)-1}$. Using standard ODE theory:
\begin{multline*}
\abs[2]{{\partial_x}X_{b(h, \cdot)}(t-s,x)-1} \leq \abs[2]{{\partial_x}X_{b(h, \cdot)}(t-s,x)-{\partial_x}X_{b(h, \cdot)}(0,x)}
\leq \abs[2]{\int_0^{t-s} {\partial_x} b(h,X_{b(h, \cdot)}(u,x)) du} \leq \\ \leq \int_0^{t-s} \abs{b_x(h,X_{b(h, \cdot)}(u,x)) ~{\partial_x}X_{b(h, \cdot)}(u,x)}  du \leq C_L \abs{t},
\end{multline*}
due to Lemma \ref{l0_wx}. The proof is concluded.
\end{proof}
\section{Proof of Theorem \ref{thm_main_diff}}\label{perturbation_mdlfcns_app}
In this Section, we provide the proof of Theorem \ref{thm_main_diff}. We use notation and assume (C1)--(C3) from Section \ref{perturbation_mdlfcns_section}. We start with a general statement concering implicit equations:
\begin{lem}\label{intro_to_estimate}
Let $\xi$ and $\bar{\xi}$ be $C^{1+\alpha}$ functions defined on $[-\frac{1}{2}, \frac{1}{2}] \times \mathbb{R}^+$.  If $\varphi^h_{\xi,t_1}$ solves \eqref{f1_do_szacow} and $\bar{\varphi}^h_{\xi,t_2}$ solves \eqref{f2_do_szacow} then
$$
\sup_{(h,x,w) \in \mathcal{D}}\abs{\varphi^h_{\xi,t_1}(s_1+w,x) - \bar{\varphi}^h_{\bar{\xi},t_2}(s_2+w,x)} \leq e^{2\norm{\bar{q}}_{\infty} \abs{\Delta t}} \Big(\norm{p-\bar{p}}_{\infty} + \abs{\Delta t} \norm{q - \bar{q}}_{\infty} \Big) .
$$
Moreover, we have the bounds for differences of derivatives with respect to variable $h$:
\begin{multline}\label{est_general_for_pdh}
\sup_{(h,x,w) \in \mathcal{D}}\abs[2]{{\partial_h}\varphi^h_{\xi,t_1}(s_1+w,x) - {\partial_h}\bar{\varphi}^h_{\bar{\xi},t_2}(s_2+w,x)} \leq e^{2\norm{\bar{q}}_{\infty} \abs{\Delta t}} \Big(\norm[1]{p_h-\bar{p}_h}_{\infty} + \abs[1]{\Delta t} \norm[1]{\varphi^h_{\xi,t_1}}_{\infty} \norm[1]{q_h - \bar{q}_h}_{\infty}+\\+ \abs[1]{\Delta t} \norm[1]{\bar{q}_h}_{\infty} \sup_{(h,x,w) \in \mathcal{D}}\abs[1]{\varphi^h_{\xi,t_1}(s_1+w,x) - \bar{\varphi}^h_{\bar{\xi},t_2}(s_2+w,x)} +  \abs[1]{\Delta t} \norm[1]{q - \bar{q}}_{\infty} \Big) ,
\end{multline}
and in variable $x$:
\begin{multline}\label{est_general_for_pdx}
\sup_{(h,x,w) \in \mathcal{D}}\abs[2]{{\partial_x}\varphi^h_{\xi,t_1}(s_1+w,x) - {\partial_x}\bar{\varphi}^h_{\bar{\xi},t_2}(s_2+w,x)} \leq \norm[1]{p_x - \bar{p}_x}_{\infty} + \abs[1]{\Delta t} \norm[1]{\varphi^h_{\xi,t_1}}_{\infty} \norm[1]{q_x - \bar{q}_x}_{\infty} + \\ + \abs[1]{\Delta t} \norm[1]{\bar{q_x}}_{\infty} \sup_{(h,x,w) \in \mathcal{D}}\abs[1]{\varphi^h_{\xi,t_1}(s_1+w,x) - \bar{\varphi}^h_{\bar{\xi},t_2}(s_2+w,x)}.
\end{multline}
\begin{proof}
Equation \eqref{f1_do_szacow} is of the type
\begin{equation}\label{type_general}
f(h,x,s_1+w) = p(h,x,w) + \int_0^{\abs{\Delta t} - w}q(h,u,x) f(h,0,u+s_1+w)du
\end{equation}
and similarly, \eqref{f2_do_szacow} is of the same type with $f, p, q, s_1$ replaced with $\bar{f}, \bar{p}, \bar{q}, s_2$ where $f = \varphi^h_{\xi,t_1}$ and $\bar{f} = \bar{\varphi}^h_{\bar{\xi},t_2}$. Therefore, applying Bielecki norm \eqref{bielecki} in variable $w$:
\begin{multline*}
\norm{f(h,x,s_1+w) - \bar{f}(h,x,s_2+w)}_{\lambda} \leq \norm{p - \bar{p}}_{\infty} +  \abs{\Delta t} \norm{q-\bar{q}}_{\infty} \norm{f}_{\infty} + \\ +\sup_{(w,x,h) \in \mathcal{D}} \norm{\bar{q}}_{\infty} \int_0^{\abs{\Delta t} - w}e^{-\lambda (\abs{\Delta t} - w)} e^{\lambda(\abs{\Delta t} - u - w)} e^{-\lambda(\abs{\Delta t} - u - w)} \abs{f(h,0,u+s_1+w) - \bar{f}(h,0,u+s_2+w)}du \leq \\ \leq
\norm{p - \bar{p}}_{\infty} +  \abs{\Delta t} \norm{q-\bar{q}}_{\infty} \norm{f}_{\infty} + \norm{\bar{q}}_{\infty}\norm{f(h,x,s_1+w) - \bar{f}(h,x,s_2+w)}_{\lambda} \sup_{w \in [0, \abs{\Delta t}]}\int_0^{\abs{\Delta t} - w} e^{-\lambda u} du =\\=
\norm{p - \bar{p}}_{\infty} +  \abs{\Delta t} \norm{q-\bar{q}}_{\infty} \norm{f}_{\infty} + \norm{\bar{q}}_{\infty}\norm{f(h,x,s_1+w) - \bar{f}(h,x,s_2+w)}_{\lambda} \frac{1-e^{-\lambda \abs{\Delta t}}}{\lambda}.
\end{multline*}
Choosing, as always, $\lambda = \norm{\bar{q}}_{\infty}$ we obtain the first inequality. To see the second one, as in the proof of Lemma \ref{l6}, we differentiate \eqref{type_general} to obtain:
\begin{equation}\label{type_general_f_h}
f_h(h,x,s_1+w) = p_h(h,x,w) + \int_0^{\abs{\Delta t} - w}\Big(q_h(h,u,x) f(h,0,u+s_1+w) +q(h,u,x) f_h(h,0,u+s_1+w)\Big) du.
\end{equation} 
Notice that \eqref{type_general_f_h} treated as equation for $f_h$ is of the same type as \eqref{type_general}. Since
\begin{multline*}
\norm{\int_0^{\abs{\Delta t} - w} q_h(h,u,x) f(h,0,u+s_1+w) - \bar{q}_h(h,u,x) \bar{f}(h,0,u+s_2+w)}_{\infty} \leq \\ \leq \norm{f}_{\infty}\abs{\Delta t} \norm{q_h - \bar{q}_h}_{\infty} + \norm{\bar{q}_h}_{\infty} \abs{\Delta t} \norm{f - \bar{f}} _{\infty},
\end{multline*}
the conclusion follows. Finally, to establish an estimate for the difference of derivatives in $x$, we differentiate \eqref{type_general} with respect to $x$:
$$
f_x(h,x,s_1+w) = p_x(h,x,w) + \int_0^{\abs{\Delta t} - w}q_x(h,u,x) f(h,0,u+s_1+w)du
$$
and observe that the desired bound follows from triangle inequality.
\end{proof}
\end{lem}
Lemma \ref{intro_to_estimate} tells us that in order to bound the difference of solutions to \eqref{f1_do_szacow} and \eqref{f2_do_szacow} (and their derivatves), we have to estimate $\norm{p - \bar{p}}_{\infty}$, $\norm{q - \bar{q}}_{\infty}$, $\norm{p_x - \bar{p}_x}_{\infty}$, $\norm{q_x - \bar{q}_x}_{\infty}$, $\norm{p_h - \bar{p}_h}_{\infty}$ as well as $\norm{q_h - \bar{q}_h}_{\infty}$. Naturally, we begin with the result showing how flow $X_{b(h, \cdot)}$ is affected by perturbation in model function $b$.
\begin{lem}\label{perturbation_flow}
The flows $X_{b(h, \cdot)}$ and $X_{\bar{b}(h, \cdot)}$ satisfy the following bounds:
$$
\abs{X_{b(h, \cdot)}(u,x) - X_{\bar{b}(h,\cdot)}(u,x)} \leq e^{C_L\abs{\Delta t}} \abs{\Delta t}^2 ,
$$
$$
\abs{{\partial_h}X_{b(h, \cdot)}(u,x) - {\partial_h}X_{\bar{b}(h,\cdot)}(u,x)} \leq  e^{C_L\abs{\Delta t}} \abs{\Delta t}\abs{\Delta f_h} +  H_L \abs{\Delta t}^{1+2\alpha},
$$
$$
\abs{{\partial_x}X_{b(h, \cdot)}(u,x) - {\partial_x}X_{\bar{b}(h,\cdot)}(u,x)} \leq  H_L \abs{\Delta t}^{1+\alpha}.
$$
\end{lem}
\begin{proof}
Recall the bounds from Corollary \ref{summary_flow}. As always, we write:
\begin{multline*}
{\partial_u} X_{b(h,\cdot)}(u,x)
- {\partial_u} X_{\bar{b}(h,\cdot)}(u,x) = 
b(h, X_{b(h,\cdot)}(u,x)) - \bar{b}(h, X_{\bar{b}(h,\cdot)}(u,x)) \leq  \\ \leq
\underbrace{\abs{b(h, X_{b(h,\cdot)}(u,x)) - \bar{b}(h, X_{b(h,\cdot)}(u,x))}}_{\leq~\norm{b-\bar{b}}_{\infty}} +\underbrace{ \abs{\bar{b}(h, X_{b(h,\cdot)}(u,x)) - \bar{b}(h, X_{\bar{b}(h,\cdot)}(u,x)) }}_{\leq~\norm{\bar{b}_x}_{\infty} \abs{X_{b(h,\cdot)}(u,x) - X_{\bar{b}(h,\cdot)}(u,x)}}
\end{multline*}
so first assertion follows from Gronwall inequality. Similarly,
\begin{multline*}
{\partial_u} {\partial_h} X_{b(h,\cdot)}(u,x)
- {\partial_u}{\partial_h}  X_{\bar{b}(h,\cdot)}(u,x) = {\partial_h}b(h, X_{b(h,\cdot)}(u,x)) - 
{\partial_h}\bar{b}(h, X_{\bar{b}(h,\cdot)}(u,x)) \leq \\
 \leq \underbrace{\abs{b_h(h, X_{b(h,\cdot)}(u,x)) - \bar{b}_h(h, X_{\bar{b}(h,\cdot)}(u,x))}}_{
 \leq~\norm{b_h - \bar{b}_h}_{\infty} ~+~\norm{\bar{b}_h}_{\alpha,x}~\abs{X_{b(h,\cdot)}(u,x) - X_{\bar{b}(h,\cdot)}(u,x)}^{\alpha}} +
  \\ +
 \underbrace{\abs{b_x(h, X_{b(h,\cdot)}(u,x))~{\partial_h} X_{b(h,\cdot)}(u,x) - \bar{b}_x(h, X_{\bar{b}(h,\cdot)}(u,x))~{\partial_h} X_{\bar{b}(h,\cdot)}(u,x)}}_{
 \leq~\big(\norm{b_x - \bar{b}_x}_{\infty} ~+~\norm{\bar{b}_x}_{\alpha,x}~\abs{X_{b(h,\cdot)}(u,x) - X_{\bar{b}(h,\cdot)}(u,x)}^{\alpha}\big)\abs{\Delta t} C_L~+~ \norm{\bar{b}_x}_{\infty}\abs{ {\partial_h} X_{b(h,\cdot)}(u,x)
- {\partial_h}  X_{\bar{b}(h,\cdot)}(u,x)} }.
\end{multline*}
Finally,
\begin{multline*}
{\partial_u} {\partial_x} X_{b(h,\cdot)}(u,x)
- {\partial_u}{\partial_x}  X_{\bar{b}(h,\cdot)}(u,x) = 
{\partial_x}b(h, X_{b(h,\cdot)}(u,x)) - 
 {\partial_x}\bar{b}(h, X_{\bar{b}(h,\cdot)}(u,x)) \leq \\
 \leq b_x(h, X_{b(h,\cdot)}(u,x)) ~{\partial_x} X_{b(h,\cdot)}(u,x) - 
 \bar{b}_x(h, X_{\bar{b}(h,\cdot)}(u,x)) ~{\partial_x} X_{\bar{b}(h,\cdot)}(u,x) \leq \\ \leq
\Big|\underbrace{\Big(b_x(h, X_{b(h,\cdot)}(u,x)) - \bar{b}_x(h, X_{\bar{b}(h,\cdot)}(u,x)) \Big)}_
 {\leq~\norm{b_x - \bar{b}_x}_{\infty} ~+~ \norm{\bar{b}_x}_{\alpha,x} \abs{\Delta t}^{\alpha}\abs{\Delta f}^{\alpha} C_L}~{\partial_x} X_{b(h,\cdot)}(u,x) \Big| + \\ + \Big|\bar{b}_x(h, X_{\bar{b}(h,\cdot)}(u,x)) \Big({\partial_x} X_{b(h,\cdot)}(u,x) - {\partial_x} X_{\bar{b}(h,\cdot)}(u,x) \Big)\Big|
 \end{multline*}
concluding the proof.
\end{proof}
We move to study how perturbation in $b$ affects the function $\xi(h,X_{b(h,\cdot)}(s,y))$:
\begin{lem}\label{dif_in_xi}
Let $\xi, \bar{\xi} \in C^{1+\alpha}([-\frac{1}{2}, \frac{1}{2}] \times \mathbb{R}^+  )$. Then, 
\begin{equation}\label{dif_in_xi_eq}
\abs{\xi(h,X_{b(h,\cdot)}(s,y)) - \bar{\xi}(h,X_{\bar{b}(h,\cdot)}(s,y))} \leq \norm{\xi - \bar{\xi}}_{\infty} + e^{C_L\abs{\Delta t}} \abs{\Delta t}^2 \norm{\bar{\xi}_x}_{\infty},
\end{equation}
$$
\abs{\xi_h(h,X_{b(h,\cdot)}(s,y)) - \bar{\xi_h}(h,X_{\bar{b}(h,\cdot)}(s,y))} \leq \norm{\xi_h - \bar{\xi}_h}_{\infty} + e^{C_L\abs{\Delta t}} \abs{\Delta t}^{2\alpha}\norm{\bar{\xi}_h}_{\alpha,x} ,
$$
$$
\abs{\xi_x(h,X_{b(h,\cdot)}(s,y)) - \bar{\xi}_x(h,X_{\bar{b}(h,\cdot)}(s,y))} \leq \norm{\xi_x - \bar{\xi}_x}_{\infty} + e^{C_L\abs{\Delta t}}\abs{\Delta t}^{2\alpha} \norm{\bar{\xi}_x}_{\alpha,x}.
$$
\end{lem}
\begin{proof}
This Lemma follows from triangle inequality in the spirit of Remark \ref{estimate_convention} and bounds from Lemma \ref{perturbation_flow}. For instance:
\begin{equation*}
\abs{\xi(h,X_{b(h,\cdot)}(s,y)) - \bar{\xi}(h,X_{\bar{b}(h,\cdot)}(s,y))} \leq
\norm{\xi - \bar{\xi}}_{\infty} + \norm{\xi_x}_{\infty} \abs[1]{X_{{b}(h,\cdot)} - X_{\bar{b}(h,\cdot)}}.
\end{equation*}
\end{proof}
\begin{lem}\label{dif_in_c}
Let $b, \bar{b}, c, \bar{\vphantom{b}c} \in C^{1+\alpha}([-\frac{1}{2}, \frac{1}{2}] \times \mathbb{R}^+  )$ be model functions satisfying assumptions (C1) -- (C3). Then:
$$
\abs{e^{\int_0^{\abs{\Delta t} - w}{c}(h, X_{{b}(h,\cdot)}(u,x)) du} - e^{\int_0^{\abs{\Delta t} - w}\bar{c}(h, X_{\bar{b}(h,\cdot)}(u,x)) du}} \leq  e^{C_L\abs{\Delta t}}\abs{\Delta t}^2.
$$
\end{lem}
\begin{proof}
Since the map $[-{M}, M] \ni x \mapsto e^x$ is Lipschitz with some constant $e^M$, the result follows immidiately from Lemma \ref{dif_in_xi} (with $\xi$ replaced with $c$ so that $\norm{\xi - \bar{\xi}}_{\infty} \leq \abs{\Delta f} \leq C_L T$ due to (C2)).
\end{proof}
\begin{cor}\label{direct_estimate_on_fcns_dif_cor}
Let $p$, $\bar{p}$, $q$, $\bar{q}$ be as in Section \ref{perturbation_mdlfcns_section}. Then,
$$
\norm{p - \bar{p}}_{\infty} \leq \norm{\xi - \bar{\xi}}_{\infty} + e^{C_L\abs{\Delta t}} \abs{\Delta t}^2\big(1 + \norm{\bar{\xi}}_{\infty}\big),
\quad \quad \quad
\norm{q - \bar{q}}_{\infty} \leq  e^{C_L\abs{\Delta t}}\abs{\Delta t}.
$$
In particular, due to Lemma \ref{intro_to_estimate}, bound \eqref{direct_estimate_on_fcns_dif_eqn} follows.
\end{cor}
\begin{proof}
Using triangle inequality (like discussed in Remark \ref{estimate_convention}) as well as Lemmas \ref{dif_in_xi} and \ref{dif_in_c} we obtain the first assertion:
\begin{equation*}
\norm{p-\bar{p}}_{\infty} \leq  \big(\norm{\xi - \bar{\xi}}_{\infty} + \abs{\Delta t}^2e^{C_L\abs{\Delta t}}\big)e^{\abs{\Delta t} \norm{c}_{\infty}} + \norm{\bar{\xi}}_{\infty} \abs{\Delta t}^2 e^{C_L\abs{\Delta t}}.
\end{equation*}
The second follows by letting $\xi =a$ and $\bar{\xi} = \bar{a}$ in the first one so that $\norm{a - \bar{a}}_{\infty} \leq \abs{\Delta f} \leq C_L T$ due to assumption (C2).
\end{proof}
We then focus on the differences in derivatives.
\begin{lem}\label{xi_der_dif_bds}
Let $\xi, \bar{\xi} \in C^{1+\alpha}([-\frac{1}{2}, \frac{1}{2}] \times \mathbb{R}^+)$. Then,
\begin{multline}\label{d/hxi-xibar}
\abs{{\partial_h} \xi(h,X_{b(h,\cdot)}(u,x)) - {\partial_h}\bar{\xi}(h,X_{\bar{b}(h,\cdot)}(u,x))} \leq
\norm{\xi_h - \bar{\xi}_h}_{\infty} + C_L\abs{\Delta t}\norm{\xi_x - \bar{\xi}_x}_{\infty} +  \\ +
C_L \abs{\Delta t}^{2\alpha}\norm{\xi_h}_{\alpha,x} + C_L \abs{\Delta t}^{1+2\alpha}\norm{\xi_x}_{\alpha,x} +
C_L \abs{\Delta t} \abs{\Delta f_h}  \norm{\bar{\xi}_x}_{\infty}  + H_L\abs{\Delta t}^{1+\alpha}\norm{\bar{\xi}_x}_{\infty} .
\end{multline}
Moreover,
\begin{multline}\label{d/xxi-xibar}
\abs{{\partial_x} \xi(h,X_{b(h,\cdot)}(u,x)) - {\partial_x}\bar{\xi}(h,X_{\bar{b}(h,\cdot)}(u,x))} \leq e^{C_LT}\norm{\xi_x - \bar{\xi}_x}_{\infty} +\\+ C_L \abs{\Delta t}^{2\alpha} \norm{\bar{\xi}_x}_{\alpha,x}    + H_L \abs{\Delta t}^{1+\alpha} \norm{\bar{\xi}_x}_{\infty}. 
\end{multline}
\end{lem}
\begin{proof}
Clearly, ${\partial_h} \xi(h,X_{b(h,\cdot)}(u,x)) = \xi_h(h,X_{b(h,\cdot)}(u,x)) + \xi_x(h,X_{b(h,\cdot)}(u,x)) ~{\partial_h}X_{b(h,\cdot)}(u,x)$. Therefore, using Lemmas \ref{perturbation_flow} and \ref{dif_in_xi} we compute:
\begin{multline*}
 \abs[2]{{\partial_h} \xi(h,X_{b(h,\cdot)}(u,x)) - {\partial_h}\bar{\xi}(h,X_{\bar{b}(h,\cdot)}(u,x))} \leq 
 \underbrace{\abs[2]{\xi_h(h,X_{b(h,\cdot)}(u,x)) - \bar{\xi}_h(h,X_{\bar{b}(h,\cdot)}(u,x))}}_{:=~(A)} + \\ \abs[2]{{\partial_h}X_{b(h,\cdot)}(u,x) \underbrace{\big(\xi_x(h,X_{b(h,\cdot)}(u,x)) - \bar{\xi}_x(h,X_{\bar{b}(h,\cdot)}(u,x))\big)}_{:=~(B)}} + \norm{\bar{\xi}_x}_{\infty}\underbrace{\abs[2]{{\partial_h}X_{b(h,\cdot)}(u,x) - {\partial_h}X_{\bar{b}(h,\cdot)}(s,y) }}_{:=~(C)} \leq  \\ \leq
\underbrace{\norm{\xi_h - \bar{\xi}_h}_{\infty} + e^{C_L\abs{\Delta t}} \abs{\Delta t}^{2\alpha} \norm{\bar{\xi}_h}_{\alpha,x} }_{(A)~\leq } + \Big(\underbrace{\norm{\xi_x - \bar{\xi}_x}_{\infty} + e^{C_L\abs{\Delta t}} \abs{\Delta t}^{2\alpha} \norm{\bar{\xi}_x}_{\alpha,x} }_{(B)~\leq}\Big) \abs{\Delta t}C_L +\\ +
\norm{\bar{\xi}_x}_{\infty}\Big(\underbrace{ e^{C_L\abs{\Delta t}}\abs{\Delta t}\abs{\Delta f_h} + H_L\abs{\Delta t}^{1+2\alpha} }_{(C)~\leq} \Big).
$$
\end{multline*}
Similarly, using Lemma \ref{perturbation_flow},
\begin{multline*}
\abs[1]{{\partial_x} \xi(h,X_{b(h,\cdot)}(u,x)) - {\partial_x}\bar{\xi}(h,X_{\bar{b}(h,\cdot)}(u,x))} \leq
e^{C_L T}\norm[1]{\xi_x - \bar{\xi}_x}_{\infty} + C_L\abs[1]{\Delta t}^{2\alpha} \norm[1]{\bar{\xi}_x}_{\alpha,x} 
\norm[1]{\partial_x X_{b(h,\cdot)}(u,x)}_{\infty}+ \\+ H_L \abs[1]{\Delta t}^{1+\alpha} \norm[1]{\bar{\xi}_x}_{\infty},
\end{multline*}
so that \eqref{d/xxi-xibar} follows from Corollary \ref{summary_flow}.
\end{proof}
From Lemma \ref{xi_der_dif_bds}, we deduce bounds on terms with model functions:
\begin{cor}\label{a_der_dif_bds} Let $a, \bar{a}\in C^{1+\alpha}([-\frac{1}{2}, \frac{1}{2}] \times \mathbb{R}^+  )$ be two model functions satisfying assumptions (C1) -- (C3). Then:
\begin{equation*}
\abs{{\partial_h} a(h,X_{b(h,\cdot)}(u,x)) - {\partial_h}\bar{a}(h,X_{\bar{b}(h,\cdot)}(u,x))} \leq
\abs{\Delta f_h} + H_L\abs{\Delta t}^{2\alpha}.
\end{equation*}
Moreover,
\begin{equation*}
\abs{{\partial_x} a(h,X_{b(h,\cdot)}(u,x)) - {\partial_x}\bar{a}(h,X_{\bar{b}(h,\cdot)}(u,x))} \leq H_L \abs{\Delta t}^{\alpha}.
\end{equation*}
\end{cor}
\begin{cor}\label{dif_in_der_c_cbar}
Let $c, \bar{c}\in C^{1+\alpha}([-\frac{1}{2}, \frac{1}{2}] \times \mathbb{R}^+  )$ be two model functions satisfying assumptions (C1) -- (C3). Then:
\begin{equation}\label{dhc-cbar}
\Big|{{\partial_h}e^{\int_0^{\abs{\Delta t} - w}{c}(h, X_{{b}(h,\cdot)}(u,x)) du} - {\partial_h}e^{\int_0^{\abs{\Delta t} - w}\bar{c}(h, X_{\bar{b}(h,\cdot)}(u,x)) du}}\Big| \leq  e^{C_L\abs{\Delta t}} \abs{\Delta t} \abs{\Delta f_h}  + H_L \abs{\Delta t}^{1+2\alpha} .
\end{equation}
Moreover,
\begin{equation}\label{dxc-cbar}
\Big|{\partial_x}e^{\int_0^{\abs{\Delta t} - w}{c}(h, X_{{b}(h,\cdot)}(u,x)) du} - {\partial_x} e^{\int_0^{\abs{\Delta t} - w}\bar{c}(h, X_{\bar{b}(h,\cdot)}(u,x)) du}\Big| \leq H_L \abs{\Delta t}^{1+\alpha}  .
\end{equation}
\end{cor}
\begin{proof}[Proof of Corollaries \ref{a_der_dif_bds} and \ref{dif_in_der_c_cbar}]
Corollary \ref{a_der_dif_bds} follows from the direct application of Lemma \ref{xi_der_dif_bds} (set $\xi = a$ and $\bar{\xi} = \bar{a}$). To prove Corollary \ref{dif_in_der_c_cbar}, note that there are two differences to be bounded: difference in exponential function and difference in exponents (arising from computing derivative). The first one is easily bounded by a term with much higher power of $\abs{\Delta t}$ than required due to Lemma \ref{dif_in_c}. For the second one, we just have to multiply bounds from Corollary \ref{a_der_dif_bds} by $\abs{\Delta t}$ (i.e. the length of integration interval). Bound for the difference of derivatives in $x$ is obtained similarly.
\end{proof}
Finally, we conclude the proof of Theorem \ref{thm_main_diff}:
\begin{proof}[Proof of estimate \eqref{precise_estimate_diff_dx_eq}]
Clearly, we want to apply the bound \eqref{est_general_for_pdx} from Lemma \ref{intro_to_estimate}. We start from estimating term $\norm[1]{p_x - \bar{p}_x}_{\infty}$. Note that:
\begin{multline*}
p_x(h,u,x) = \partial_x\Big(\xi(h,X_{b(h, \cdot)}(\abs{\Delta t}-u,x))\Big)~e^{\int_0^{\abs{\Delta t}-u }c(h, X_{b(h,\cdot)}(u,x)) du} + \\ + 
\xi(h,X_{b(h, \cdot)}(\abs{\Delta t}-u,x))~\partial_x\Big(e^{\int_0^{\abs{\Delta t}-u }c(h, X_{b(h,\cdot)}(u,x)) du}\Big),
\end{multline*}
and similarly for $\bar{p}_x$. Therefore, to bound $\norm[1]{p_x - \bar{p}_x}_{\infty}$ with triangle inequality, we have to sum up the following estimates: 
\begin{enumerate}[itemsep=+1ex, label=(\Alph*):]
\item{$\norm[1]{\partial_x \xi(h,X_{b(h, \cdot)}(\cdot,\cdot))-\partial_x \bar{\xi}(h,X_{\bar{b}(h, \cdot)}(\cdot, \cdot))}_{\infty} e^{C_L\abs{\Delta t}}$ using bound \eqref{d/xxi-xibar},}
\item{$\norm[1]{\bar{\xi}_x}_{\infty} \norm[1]{\partial_x X_{\bar{b}(h,\cdot)}(\cdot,x)}_{\infty} \norm[1]{e^{\int_0^{\abs{\Delta t}-u}c(h, X_{b(h,\cdot)}(u,x)) du}-e^{\int_0^{\abs{\Delta t}-u}\bar{c}(h, X_{\bar{b}(h,\cdot)}(u,x)) du}}_{\infty}$ using Lemma \ref{dif_in_c} and Corollary \ref{summary_flow},}
\item{$\norm[1]{\xi(h,X_{b(h, \cdot)}(\cdot,\cdot))-\bar{\xi}(h,X_{\bar{b}(h, \cdot)}(\cdot, \cdot))}_{\infty} e^{C_L\abs{\Delta t}} \abs{\Delta t} \norm[1]{c_x}_{\infty}\norm[1]{\partial_x X_{b(h,\cdot)}(\cdot,x)}_{\infty}$ using  inequality \eqref{dif_in_xi_eq} and bound for the flow from Corollary \ref{summary_flow},}
\item{$\norm[1]{\bar{\xi}}_{\infty}\norm[1]{
\partial_xe^{\int_0^{\abs{\Delta t}-u }c(h, X_{b(h,\cdot)}(u,x)) du}- \partial_xe^{\int_0^{\abs{\Delta t}-u }\bar{c}(h, X_{\bar{b}(h,\cdot)}(u,x)) du}\big)\big)}_{\infty}$ using inequality \eqref{dxc-cbar}.}
\end{enumerate}
Therefore,
\begin{multline*}
\norm[1]{p_x - \bar{p}_x}_{\infty} \leq 
\underbrace{e^{C_LT}\norm{\xi_x - \bar{\xi}_x}_{\infty} +   C_L\abs{\Delta t}^{2\alpha}\norm{\bar{\xi}_x}_{\alpha,x}  + H_L\abs{\Delta t}^{1+\alpha}\norm{\xi_x}_{\infty} }_{(A) \leq } + 
 \underbrace{e^{C_L\abs{\Delta t}}\abs{\Delta t}^2\norm[1]{\bar{\xi}_x}_{\infty}  }_{(B) \leq} +
\\ +
 \underbrace{\big(\norm[1]{\xi - \bar{\xi}}_{\infty} + e^{C_L\abs{\Delta t}}\norm[1]{\xi_x}_{\infty}\abs{\Delta t}^2\big) C_L\abs{\Delta t} }_{(C) \leq } +
\underbrace{H_L \abs{\Delta t}^{1+\alpha}\norm[1]{\bar{\xi}}_{\infty}  }_{(D) \leq }.
\end{multline*}
Now, we study the second term $ \abs[1]{\Delta t} \norm[1]{\varphi^h_{\xi,t_1}}_{\infty} \norm[1]{q_x - \bar{q}_x}_{\infty}$ in \eqref{est_general_for_pdx}. From above, by setting $\xi = a$ and $\bar{\xi} = \bar{a}$, we easily obtain $\norm[1]{q_x - \bar{q}_x}_{\infty} \leq  H_L \abs{\Delta t}^{\alpha}$ (which is not as good as possible but sufficient for our future targets). Moreover, recall from Lemma \ref{lem_imp_exi} that $\norm[1]{\varphi^h_{\xi,t_1}}_{\infty} \leq \norm[1]{\xi}_{\infty}e^{C_L\abs{\Delta t}}$ so that we arrive at
\begin{equation*}
\abs[1]{\Delta t} \norm[1]{\varphi^h_{\xi,t_1}}_{\infty} \norm[1]{q_x - \bar{q}_x}_{\infty} \leq \norm{\xi}_{\infty}\abs{\Delta t}^{1+\alpha} H_L,
\end{equation*}
a term which is already included in the bound for $(D)$ after replacing $\norm{\bar{\xi}}_{\infty}$ with $\max{\big(\norm{{\xi}}_{\infty}, \norm{\bar{\xi}}_{\infty} \big)}$. Finally, from \eqref{direct_estimate_on_fcns_dif_eqn}, we easily obtain bound for $$\abs[1]{\Delta t} \norm[1]{q_x}_{\infty} \sup_{(h,x,w) \in \mathcal{D}}\abs[1]{\varphi^h_{\xi,t_1}(s_1+w,x) - \bar{\varphi}^h_{\bar{\xi},t_2}(s_2+w,x)} \leq
\abs[1]{\Delta t} C_L \norm{\xi - \bar{\xi}}_{\infty} + \abs{\Delta t}^3 C_L\big(2 + \norm{\bar{\xi}}_{\infty}\big).
$$
Summation of the three bounds above concludes the proof.
\end{proof}
\begin{proof}[Proof of estimate \eqref{precise_estimate_diff_dh_eq}]
Again, we want to apply the bound \eqref{est_general_for_pdh} from Lemma \ref{intro_to_estimate}. We start from estimating term $\norm[1]{p_h - \bar{p}_h}_{\infty}$. Note that:
\begin{multline*}
p_h(h,u,x) = \partial_h\Big(\xi(h,X_{b(h, \cdot)}(\abs{\Delta t}-u,x))\Big)~e^{\int_0^{\abs{\Delta t}-u }c(h, X_{b(h,\cdot)}(u,x)) du} + \\ + 
\xi(h,X_{b(h, \cdot)}(\abs{\Delta t}-u,x))~\partial_h\Big(e^{\int_0^{\abs{\Delta t}-u }c(h, X_{b(h,\cdot)}(u,x)) du}\Big),
\end{multline*}
and similarly for $\bar{p}_h$. Therefore, to bound $\norm[1]{p_h - \bar{p}_h}_{\infty}$ with triangle inequality, we have to sum up the following estimates: 
\begin{enumerate}[itemsep=+1ex, label=(\Alph*):]
\item{$\norm[1]{\partial_h \xi(h,X_{b(h, \cdot)}(\cdot,\cdot))-\partial_h \bar{\xi}(h,X_{\bar{b}(h, \cdot)}(\cdot, \cdot))}_{\infty} e^{C_L\abs{\Delta t}}$ using bound \eqref{d/hxi-xibar},}
\item{$\norm[1]{\partial_h \bar{\xi}(h,X_{\bar{b}(h, \cdot)}(\cdot, \cdot))}_{\infty} \norm[1]{e^{\int_0^{\abs{\Delta t}-u}c(h, X_{b(h,\cdot)}(u,x)) du}-e^{\int_0^{\abs{\Delta t}-u}\bar{c}(h, X_{\bar{b}(h,\cdot)}(u,x)) du}}_{\infty}$ using Lemma \ref{pierwsza_funkcja} and Lemma \ref{dif_in_c},}
\item{$\norm[1]{\xi(h,X_{b(h, \cdot)}(\cdot,\cdot))-\bar{\xi}(h,X_{\bar{b}(h, \cdot)}(\cdot, \cdot))}_{\infty} e^{C_L\abs{\Delta t}} \norm[1]{\int_0^{\abs{\Delta t}-u }\partial_h c(h, X_{b(h,\cdot)}(u,x)) du}_{\infty}$ using  inequality \eqref{dif_in_xi_eq} and the bound for the flow from Corollary \ref{summary_flow},}
\item{$\norm[1]{\bar{\xi}}_{\infty}\norm[1]{
\partial_he^{\int_0^{\abs{\Delta t}-u }c(h, X_{b(h,\cdot)}(u,x)) du}- \partial_he^{\int_0^{\abs{\Delta t}-u }\bar{c}(h, X_{\bar{b}(h,\cdot)}(u,x)) du}\big)\big)}_{\infty}$ using inequality \eqref{dhc-cbar}.}
\end{enumerate}
Therefore,
\begin{multline*}
\norm[1]{p_h - \bar{p}_h}_{\infty} \leq 
\underbrace{e^{C_L\abs{\Delta t}} \norm{\xi_h - \bar{\xi}_h}_{\infty} + C_L\abs{\Delta t}\norm{\xi_x - \bar{\xi}_x}_{\infty} +
C_L\abs{\Delta t}^{2\alpha}\norm{\xi_h}_{\alpha,x} + C_L\abs{\Delta t}^{1+2\alpha}\norm{\xi_x}_{\alpha,x}}_{(A) \leq } +\\+ \underbrace{
C_L \abs{\Delta t} \abs{\Delta f_h} \norm{\bar{\xi}_x}_{\infty}  + H_L\abs{\Delta t}^{1+\alpha} \norm{\bar{\xi}_x}_{\infty}}_{(A) \leq {\text{ continued}}} + 
 \underbrace{\big(\norm[1]{\bar{\xi}_h}_{\infty} + C_L  \abs{\Delta t}\norm[1]{\bar{\xi}_x}_{\infty} \big) e^{C_L\abs{\Delta t}}\abs{\Delta t}^2}_{(B) \leq} +\\+
 \underbrace{\big(\norm[1]{\xi - \bar{\xi}}_{\infty} +e^{C_L\abs{\Delta t}}\abs{\Delta t}^2 \norm[1]{\xi_x}_{\infty}\big) C_L\abs{\Delta t} }_{(C) \leq } + 
\underbrace{\norm[1]{\bar{\xi}}_{\infty}\big(e^{C_L\abs{\Delta t}} \abs{\Delta t}\abs{\Delta f_h}   +  H_L\abs{\Delta t}^{1+2\alpha} \big)}_{(D) \leq } \leq \\
\leq 
e^{C_L\abs{\Delta t}} \norm{\xi_h - \bar{\xi}_h}_{\infty} +
C_L\abs{\Delta t}\norm{\xi_x - \bar{\xi}_x}_{\infty} +
C_L \abs{\Delta t} \norm[1]{\xi - \bar{\xi}}_{\infty}  + 
C_L\abs{\Delta t}^{2\alpha} \norm{\xi_h}_{\alpha,x} + \\+ 
C_L \abs{\Delta t}^{1+2\alpha}\norm{\xi_x}_{\alpha,x} +
C_L \abs{\Delta t} \abs{\Delta f_h}  \big(\norm{\bar{\xi}_x}_{\infty}+\norm{\bar{\xi}}_{\infty} \big) +  H_L\abs{\Delta t}^{1+2\alpha}\max\big(\norm{\xi}_{W^{1,\infty}}, \norm{\bar{\xi}}_{W^{1,\infty}} \big).
\end{multline*}
Now, we study the second term $ \abs[1]{\Delta t} \norm[1]{\varphi^h_{\xi,t_1}}_{\infty} \norm[1]{q_h - \bar{q}_h}_{\infty}$ in \eqref{est_general_for_pdh}. From above, by setting $\xi = a$ and $\bar{\xi} = \bar{a}$ we easily obtain $\norm[1]{q_h - \bar{q}_h}_{\infty} \leq \abs{\Delta f_h} e^{C_L\abs{\Delta t}} +  H_L \abs{\Delta t}^{2\alpha}$ (which is again not as good as possible but sufficient for our future targets). Moreover, recall from Lemma \ref{lem_imp_exi} that $\norm[1]{\varphi^h_{\xi,t_1}}_{\infty} \leq \norm[1]{\xi}_{\infty}e^{C_L\abs{\Delta t}}$ so that we arrive at
\begin{equation*}
\abs[1]{\Delta t} \norm[1]{\varphi^h_{\xi,t_1}}_{\infty} \norm[1]{q_h - \bar{q}_h}_{\infty} \leq \norm[1]{\xi}_{\infty} \abs{\Delta f_h}\abs{\Delta t} e^{C_L\abs{\Delta t}} +  \norm[1]{\xi}_{\infty} H_L \abs{\Delta t}^{1+2\alpha},
\end{equation*}
term which, after increasing constants $C_L$ and $H_L$, is already involved in bound for $\norm[1]{p_h - \bar{p}_h}_{\infty}$. Next, from \eqref{direct_estimate_on_fcns_dif_eqn}, we easily obtain bound for $\abs[1]{\Delta t} \norm[1]{q_h}_{\infty} \sup_{(h,x,w) \in \mathcal{D}}\abs[1]{\varphi^h_{\xi,t_1}(s_1+w,x) - \bar{\varphi}^h_{\bar{\xi},t_2}(s_2+w,x)}$:
$$
\abs[1]{\Delta t} C_L \norm{\xi - \bar{\xi}}_{\infty} + \abs{\Delta t}^3 C_L\big(2 + \norm{\bar{\xi}}_{\infty}\big).
$$
Finally, from Corollary \ref{direct_estimate_on_fcns_dif_cor} we obtain that $\abs[1]{\Delta t} \norm[1]{q - \bar{q}}_{\infty} e^{C_L \abs{\Delta t}} \leq \abs[1]{\Delta t}^2 C_L$. Summation of the four bounds above concludes the proof.
\end{proof}
\end{appendix}

%%% The format of bibliography items is as in the following examples:
%%% 
%%% \bib{yellowmonster}{book}{
%%%   author={Bousfield, A.K.},
%%%   author={Kan, D.M.},
%%%   title={Homotopy Limits, Completions and Localizations},
%%%   date={1972},
%%%   series={Lecture Notes in Mathematics},
%%%   volume={304},
%%%   publisher={Springer-Verlag},
%%%   address={Berlin-New York}
%%% }

%%% \bib{HA}{book}{
%%%   author={Quillen, Daniel G.},
%%%   title={Homotopical Algebra},
%%%   series={Lecture Notes in Mathematics},
%%%   volume={43},
%%%   publisher={Springer-Verlag},
%%%   address={Berlin-New York},
%%%   date={1967}
%%% }

%%% \bib{serre:shfs}{article}{
%%%   author={Serre, Jean-Pierre},
%%%   title={Homologie Singuli\`ere des Espaces Fibr\'es.  Applications},
%%%   journal={Ann. of Math. (2)},
%%%   date={1951},
%%%   volume={54},
%%%   pages={425--505}
%%% }

 % \end{biblist}
%\end{bibdiv}

\end{document}